\documentclass[a4paper,12pt,centertags,reqno]{amsart}
\usepackage{amsmath,amssymb,amsthm,graphicx,mathrsfs,url}
\usepackage[colorlinks=true,linkcolor=Red,citecolor=PineGreen]{hyperref}
\usepackage{enumitem}
\usepackage[french, english]{babel}
\usepackage[dvipsnames]{xcolor}
\usepackage[T1]{fontenc}
\usepackage[utf8]{inputenc}
\usepackage{csquotes}
\usepackage[all]{xy}
\usepackage{comment}
\usepackage{amssymb}
\usepackage{wasysym}
\usepackage{graphicx}
\usepackage{mathtools}
\usepackage{multicol}
\usepackage[bottom]{footmisc}
\usepackage[a4paper, margin=1in]{geometry}
\usepackage[dvipsnames]{xcolor}
\usepackage{enumitem}
\usepackage{dsfont}
\usepackage[depth=2]{bookmark}
\usepackage{pgfplots}
\usepackage{csquotes}
\usepackage[section]{placeins}
\pgfplotsset{compat=1.17}
\usepackage[alphabetic,msc-links]{amsrefs}
\renewcommand{\MR}[1]{}
\usepackage[all]{nowidow}
\usepackage{standard}
\usepackage{algorithm}
\usepackage{algpseudocode}
\usepackage{listings}
\usepackage[capitalise,noabbrev,nameinlink]{cleveref}
\usepackage{mathrsfs}
\usepackage{lmodern}

\newcommand{\A}{\mathrm{A}}
\newcommand{\Q}{\mathrm{Q}}
\renewcommand{\P}{\mathrm{P}}

\newcommand{\I}{\mathrm{I}}
\newcommand{\K}{\mathrm{K}}
\newcommand{\dd}{\mathrm{d}}

\newcommand{\CC}{\mathbb{C}}
\providecommand{\C}{}
\renewcommand{\C}{\mathbb{C}}
\newcommand{\RRR}{\mathbb{R}}
\newcommand{\RR}{\mathbb{R}}

\newcommand{\ZZ}{\mathbb{Z}}
\newcommand{\NN}{\mathbb{N}}

\DeclareMathOperator{\vol}{vol}

\setlist[enumerate,1]{label = (\alph*), ref = (\alph*)}
\theoremstyle{plain}
\newtheorem{theorem}{Theorem}[section]
\newtheorem*{theorem*}{Theorem}
\newtheorem{proposition}[theorem]{Proposition}
\newtheorem{corollary}[theorem]{Corollary}
\newtheorem{lemma}[theorem]{Lemma}

\theoremstyle{definition}
\newtheorem{definition}[theorem]{Definition}
\newtheorem{remark}[theorem]{Remark}

\theoremstyle{remark}

\newtheorem*{example*}{Example}

\newcounter{rulec}
\renewcommand\therulec{(\alph{rulec})}
\makeatletter
\newcommand\cuts[1]{\refstepcounter{rulec}\label@noarg{hyp:#1}\textup{\therulec}\quad}
\makeatother
\Crefname{rulec}{Case}{Cases}

\newcounter{ruleca}
\renewcommand\theruleca{(\alph{ruleca})}
\makeatletter
\newcommand\conjs[1]{\refstepcounter{ruleca}\label@noarg{hypa:#1}\textup{\theruleca}\quad}
\makeatother
\Crefname{ruleca}{Case}{Cases}

\newcounter{rulecb}
\renewcommand\therulecb{(\alph{rulecb})}
\makeatletter
\newcommand\conjsb[1]{\refstepcounter{rulecb}\label@noarg{hypb:#1}\textup{\therulecb}\quad}
\makeatother
\Crefname{rulecb}{Case}{Cases}

\setlist{
  listparindent=\parindent,
  parsep=\parskip,
}
\counterwithin{figure}{section}
\title{Geodesic Lévy flights on Zoll surfaces}

\author[Chaubet]{Yann Chaubet}
\address{Laboratoire de mathématiques Jean Leray, Nantes Université}
\email{yann.chaubet@univ-nantes.fr}

\author[Godfried]{Emanuel J\'ozsef Godfried}
\address{The School of Mathematics \& Statistics, The University of Melbourne}
\email{egodfried@student.unimelb.edu.au}

\author[Tzou] {Leo Tzou}
\address{The School of Mathematics \& Statistics, The University of Melbourne}
\email{leo.tzou@gmail.com}

\counterwithin{equation}{section}

\begin{document}

\begin{abstract}
We study the mean first capture time of isotropic Lévy flights on Zoll surfaces, namely the expected time for a geodesic Lévy process to reach a shrinking geodesic ball. While the leading-order asymptotics are universal, we prove that the first correction term encodes subtle geometric information. More precisely, it is completely determined by the local singularity type of the conjugate locus, quantified by the degree of the conjugate point. This yields a hierarchy of asymptotic regimes governed by the Lévy exponent.
\end{abstract}

\maketitle

\section{Introduction}

Let $(\Sigma, g)$ be a smooth closed Riemannian surface and $(X_t)_{t \geqslant0}$ be an isotropic L\'evy process on $\Sigma$ in the sense of Applebaum--Estrade \cite{applebaum2000}.
%(see also \cites{mijatovic2021levy, mramor2021stochastic} and references therein for more general L\'evy processes on manifolds). 
We assume that $(X_t)$ is a pure jump process with associated radial L\'evy measure on $\RR_+$ given by 
\begin{equation}\label{eq:levymeasure}
\dd \mu(s) = \frac{c_{\alpha}\dd s} {s^{1+2\alpha}} \quad \text{ for some } \alpha\in \left]0,1\right[,
\end{equation}
where $c_{\alpha} = 4^\alpha \Gamma(1 + \alpha) / \pi \Gamma(-\alpha)$ is a normalization constant. Informally, the process $(X_t)$ can be obtained as a limit of jump processes whose jumps follow geodesic trajectories, with uniformly distributed directions and jump lengths governed by the Lévy measure \eqref{eq:levymeasure}. 

In this article we are interested in the asymptotic behavior of the expected time required for the process $(X_t)$ to reach a small ball. More precisely, fix a point $p_\star \in \Sigma$. For any $p \in \Sigma$ and $\varepsilon > 0$, let
$$
u_{\star, \varepsilon}(p) = \mathbf E\Bigl[\inf\Bigl\{t \geqslant 0~:~X_t \in B(p_\star, \varepsilon)\Bigr\}\, \Bigl| \,X_0 = p\Bigr]
$$
be the expected time for the process to reach the ball $B(p_\star, \varepsilon) \subset \Sigma$, when starting from $p$. Then $u_{\star, \varepsilon} : \Sigma \to [0, \infty]$ is measurable and we denote by
$$
\overline{u}_{\star, \varepsilon} = \frac{1}{|\Sigma|} \int_\Sigma u_{\star, \varepsilon}(p) \dd \vol_g(p)
$$
the mean expected stopping time. 

In \cite{chaubet2025levy}, it was shown that if $\Sigma$ is either the round sphere, the flat torus or is Anosov (in the sense that its geodesic flow has the Anosov property), then we have 
\begin{equation}\label{eq:asympmean}
\overline{u}_{\star, \varepsilon} \sim \overline{c}_\alpha \varepsilon^{2(\alpha - 1)} \quad \text{as } \varepsilon \to 0,
\end{equation}
for some $\overline{c}_\alpha > 0.$ Moreover, for any $p \neq p_\star$ it holds 
\begin{equation}\label{eq:hitbounded}
u_{\star,\varepsilon}(p)
=
\overline{u}_{\star,\varepsilon}
+
\mathcal O(1)
\end{equation}
except if $\Sigma$ is the round sphere, $4\alpha < 1$, and $p$ is the antipodal point of $p_\star$, in which case 
$$
\overline{u}_{\star, \varepsilon} - u_{\star, \varepsilon}(p) \sim \tilde c_\alpha \,\varepsilon^{-1 + 4\alpha} \quad \text{ as } \varepsilon \to 0,
$$
for some other constant $\tilde c_\alpha > 0.$

%The round sphere therefore suggests a link between conjugate points and the asymptotic behavior of hitting times. 
The round sphere therefore provides the first indication that conjugate points may influence the asymptotic behavior of hitting times. The main purpose of this article is to show that, on Zoll surfaces \cite{zoll1903Zoll}, this phenomenon is in fact much more subtle than suggested by the antipodal singularity of the sphere. Zoll surfaces are Riemannian surfaces all of whose geodesics are simple closed curves of common length. While this property may suggest a highly rigid geometry, Zoll metrics form an infinite-dimensional class \cites{Besse1978Zoll, guillemin1976radon} whose conjugate loci can exhibit remarkably intricate singular behaviors. Our main result shows that the asymptotic behavior of $u_{\star,\varepsilon}$ near a conjugate point is entirely determined by the local singularity type of the conjugate locus.

%The round sphere therefore suggests a link between conjugate points and the asymptotic behavior of hitting times. The main purpose of this article is to show that, on Zoll surfaces, this phenomenon is considerably richer than suggested by the antipodal singularity of the sphere. Although Zoll metrics form a highly constrained class of geometries, they still admit a remarkably rich variety of conjugate loci and singular structures. We show that the asymptotic behavior of $u_{\star,\varepsilon}$ near a conjugate point is entirely determined by the local singularity type of the conjugate locus.

%This suggests that conjugate points play a crucial role in the asymptotic behavior of the deviation of $u_{\star, \varepsilon}$ from its mean value $\overline{u}_{\star, \varepsilon}$. The purpose of this paper is to show that conjugate points do indeed govern the asymptotic behavior of $u_{\star,\varepsilon}$, but in a much subtler way than suggested by the round sphere: the asymptotic behavior of $u_{\star,\varepsilon}$ near a conjugate point $p$ is determined by the local singularity structure of the conjugate locus of $p_\star$ near $p$. 

{}{}{}
{To quantify this phenomenon, we use the \emph{cuspidal degree} of a
conjugate point $p$ of $p_\star$. It is defined as the infimum in
$\mathbb N_{\geqslant 1}\cup\{\infty\}$ of all integers $k$ such that the
following property holds. In a neighborhood of $p$, the conjugate locus of
$p_\star$ is canonically parametrized by a finite union of plane curves, and
for each such curve there exists an integer $\ell\leqslant k$ such that, in
suitable local coordinates, its $\ell$-jet is given by
$u\mapsto (u^{\ell-1},u^\ell)$. A more detailed discussion of this notion is given in \S\ref{subsec:degree}.} In particular for the round sphere, the degree of the antipodal point is infinite. The degree should not be confused with the \textit{order} of conjugacy, which always equals 1 on surfaces for dimensional reasons. Rather, the degree measures the singularity type of the exponential map near conjugate vectors, see Proposition \ref{prop:normalformexp} below.
%and Appendix \ref{app:morin}.

In what follows we write $f(\varepsilon) \sim g(\varepsilon)$ as $\varepsilon \to 0$ if $f(\varepsilon)/g(\varepsilon) \to 1$. Our main result goes as follows.

\begin{theorem}\label{thm:main}
Assume that $\Sigma$ is {}{}{}{an orientable} Zoll surface. Then if $\overline{c}_\alpha$ is the constant from \eqref{eq:asympmean} one has
$$
\overline{u}_{\star, \varepsilon} \sim\overline{c}_\alpha \varepsilon^{2(\alpha - 1)} \quad \text{ as } \varepsilon \to 0.
$$
Suppose $\alpha < 1/2$. Then $u_{\star, \varepsilon}$ is continuous and for any $p \in \Sigma$, the following holds.
\begin{enumerate}[label=\emph{(\roman*)}]
\item If $p$ is not a conjugate point of $p_\star$, then $u_{\star, \varepsilon}(p) = \overline{u}_{\star, \varepsilon} + \mathcal O(1)$.
\item If $p$ is a conjugate point of $p_\star$ of degree $k \in \mathbb N_{\geqslant 1}$ and $4 \alpha < k/(k+1)$ then there is $c > 0$ such that
$$
\overline{u}_{\star, \varepsilon} - u_{\star, \varepsilon}(p)  \sim c\, \varepsilon^{-1 + 4\alpha + (k+1)^{-1}}.
$$
If $4\alpha \geqslant k / (k+1)$ then $u_{\star, \varepsilon}(p) - \overline{u}_{\star, \varepsilon} = \mathcal O(1).$
\item If $p$ is conjugate to $p_\star$ of degree $k = \infty$ and $4 \alpha < 1$ then for small $\varepsilon$ we have
$$
D^{-1} \varepsilon^{-1 + 4\alpha} f(\varepsilon) \leqslant \overline{u}_{\star, \varepsilon} - u_{\star, \varepsilon}(p) \leqslant D\varepsilon^{-1 + 4\alpha} f(\varepsilon)
$$
for some $D > 0$ and some function $f : \RR_+^* \to \left]0,2\right]$ which is a slowly varying correction term in the sense that
$$
\varepsilon^{-\delta} f(\varepsilon) \underset{\varepsilon \to 0}{\longrightarrow} \infty \quad \text{for every } \delta > 0.
$$
%If moreover there is a sub-interval $I \subset S_p\Sigma$ of directions reaching $p_\star$, then $f$ can be chosen to be $1$. If $4\alpha \geqslant 1$ then $u_{\star, \varepsilon}(p) - \overline{u}_{\star, \varepsilon} = \mathcal O(1)$.
\item If $p_\star$ has only one conjugate point $p$ then $p$ is of infinite degree and if $4\alpha < 1$ then there is $c > 0$ such that
$$
\overline{u}_{\star, \varepsilon} - u_{\star, \varepsilon}(p) \sim c\,\varepsilon^{-1+4\alpha}.
$$
\end{enumerate}
\end{theorem}

%We emphasize that the situations described in the above result may indeed occur. Using the infinite dimensional family of Zoll metrics constructed by Guillemin near the round sphere \cite{guillemin1976radon} together with a transversality argument, we show that for every $k\in\mathbb N$ there exist a Zoll surface $(\Sigma,g)$ and points $p,q\in\Sigma$ such that $q$ is conjugate to $p$ with degree $k$, see Theorem~\ref{thm:alldegree}.

This result shows that different local geometries of the conjugate locus produce different asymptotic signatures in the mean first capture time. The above hierarchy is non-vacuous. Indeed, using Guillemin’s local deformation theory for Zoll metrics \cite{guillemin1976radon} and a transversality argument, we show that every finite degree can be realized by a conjugate pair on a Zoll surface. More precisely, we prove that for every $k\in\mathbb N$ there exist a Zoll surface $(\Sigma,g)$ and points $p,q\in\Sigma$ such that $q$ is conjugate to $p$ with degree $k$, see Theorem~\ref{thm:alldegree}. {}{}{}{We also show that there exist Zoll surfaces admitting conjugate points of infinite degree but whose conjugate locus are not reduced to a point, see Lemma~\ref{lem:infinitedegree}.} The round sphere provides a distinguished example of infinite degree, where the entire conjugate locus collapses to a single point.

Zoll surfaces originate from the classical construction of non-round metrics on $\mathbb S^2$ all of whose geodesics are closed by Zoll \cite{zoll1903Zoll}. The first infinitesimal obstruction to deforming the round metric within this class was identified by Funk \cite{Funk1913}, and the converse problem was later solved by Guillemin through the Radon transform and a Nash-Moser argument \cite{guillemin1976radon}. More generally, the global geometry of Zoll surfaces has been the subject of extensive study in Riemannian geometry \cites{Weinstein1974,Besse1978Zoll,LeBrun2002Zoll,mazzucchelli2018characterization}.

Beyond these geometric aspects, Zoll manifolds have provided a fundamental
testing ground for spectral and microlocal analysis. The periodicity of the
geodesic flow leads to strong constraints on the spectrum of geometric
operators, while still allowing subtle spectral phenomena \cites{duistermaat1975spectrum,weinstein1977asymptotics,colin1979spectre,UribeZelditch1993,Zelditch1996,Zelditch1997,MaciaRiviere2016}.

In this paper, we exploit a different aspect of the same geometry: despite the
global rigidity imposed by the periodicity of all geodesics, Zoll surfaces can
exhibit a rich hierarchy of singular conjugate loci, and these singularities are
detected by the asymptotic behavior of L\'evy hitting times.

%The proof of Theorem~\ref{thm:main} relies on the analysis of the generator of the process. The latter is given by \cite{applebaum2000} 
%\begin{equation}\label{eq:generator}
%\A = \lim_{\varepsilon \to 0} c_\alpha \int_\varepsilon^\infty \pi_* (\varphi_t^* - \I)\pi^* \frac{\dd t}{t^{1 + 2\alpha}} : \mathscr C^\infty(\Sigma) \to \mathscr D'(\Sigma),
%\end{equation}
%where $\pi : S\Sigma \to \Sigma$ denotes the unit tangent bundle projection and $(\varphi_t)$ is the geodesic flow on $S\Sigma$. A key step is a microlocal description of the generator. We show that $\A$ decomposes into a pseudodifferential contribution and a Fourier integral contribution. The latter has a particularly rigid structure: its canonical relation is the graph of an exact involutive symplectomorphism of $T^*\Sigma\setminus 0$ which is canonically defined using the geometry of conjugate points, see Proposition \ref{prop:arho}. This decomposition is what allows us to relate the singularities of the Green kernel of $\A$ to the local geometry of the conjugate locus.

\subsection*{Strategy of proof}
The proof of Theorem~\ref{thm:main} relies on the analysis of the generator
\begin{equation}\label{eq:generator}
\A = \lim_{\delta \to 0} c_\alpha \int_\delta^\infty \pi_* (\varphi_t^* - \I)\pi^* \frac{\dd t}{t^{1 + 2\alpha}} : \mathscr C^\infty(\Sigma) \to \mathscr D'(\Sigma),
\end{equation}
of the process $(X_t)$ introduced in \cite{applebaum2000}. Here $\mathscr D'(\Sigma)$ is the space of distributions on $\Sigma$ while $\pi:S\Sigma\to\Sigma$ denotes the unit tangent bundle projection and $(\varphi_t)$ is the geodesic flow. A first step consists in establishing the connection between $\A$ and the expected hitting time. More precisely, adapting the arguments from \cite{chaubet2025levy}, we prove that $u_{\star,\varepsilon}$ is characterized as the solution of a non-local boundary value problem involving $\A$. This reduces the study of hitting times to the analysis of an almost inverse of the generator.

A key step is therefore to obtain a microlocal description of $\A$. We show that it decomposes
{{}{}{}{}{}
as 
$$
\A = \A_\alpha + \mathrm B
$$
where $\A_\alpha$ is an elliptic pseudo-differential operator of order $2\alpha$ and $\mathrm B$ is a Fourier integral operator of order $-1$ whose canonical relation is the graph of an exact, involutive symplectomorphism 
$$
G : T^*\Sigma \setminus \underline 0 \to T^*\Sigma \setminus \underline 0
$$
which is canonically determined by the geometry of conjugate points, see Propositions \ref{prop:generator} and \ref{prop:arho}.}
%into a pseudodifferential contribution and a Fourier integral contribution. The latter admits a canonical relation which is the graph of an exact involutive symplectomorphism of $T^*\Sigma\setminus0$, canonically determined by the geometry of conjugate points; see Proposition~\ref{prop:arho}.
Since $G$ is involutive, compositions involving the Fourier integral part generate no new canonical relations. In particular, parametrices of $\A$ remain within the same microlocal class. This yields a precise description of the Green kernel of $\A$ and allows us to relate its singularities to the local geometry of the conjugate locus. {{}{}{}{}{}The local analysis underlying this geometry relies on normal forms for the exponential map corresponding to the different degrees of conjugate points. Such normal forms are in the spirit of the classical work of Warner \cite{warner1965conjugateRiemann}, which, in arbitrary dimension, treats the cases corresponding to degrees~$1$ and~$\infty$ in our terminology.} {}{}{}{In turn, this provides the key tool for deriving sharp asymptotics for $u_{\star,\varepsilon}$.}

\subsection*{The non-orientable case}
{}{}{}
{
%If $(\Sigma, g)$ is a Zoll surface which is not orientable, then we have
%\begin{equation}\label{eq:hitbounded}
%u_{\star, \varepsilon}(p) = \overline{u}_{\star, \varepsilon} + \mathcal O(1) 
%\end{equation}
%as $\varepsilon \to 0$, for any $p \neq p_\star$. Indeed, in that case $\Sigma$ is diffeomorphic to $\RR P^2$ and the conjugate locus of $p_\star$ is given by $p_\star$ itself, see Theorem \ref{thm:zolltopology} (in fact a stronger result holds and $\Sigma$ must be isometric to $\RR P^2$ by a result of Green \cite{green1963wiedersehensflachen}, see Remark \ref{rem:green}). In that case, as we shall see in \S\ref{subsec:notorientable}, the generator $\mathrm A$ of the process is an elliptic pseudo-differential operator as in the torus or Anosov case; this yields \eqref{eq:hitbounded}.
If $(\Sigma,g)$ is a non-orientable Zoll surface, then \eqref{eq:hitbounded} holds as $\varepsilon\to0$, for every $p\neq p_\star$. Indeed, by Theorem~\ref{thm:zolltopology}, $\Sigma$ is diffeomorphic to $\RR P^2$ and $p_\star$ is conjugate only to itself along geodesics through $p_\star$ (in fact a result of Green \cite{green1963wiedersehensflachen} implies that $(\Sigma,g)$ must be isometric to the standard projective plane; see Remark~\ref{rem:green}). As we shall see in \S\ref{subsec:notorientable}, the absence of non-trivial conjugate points implies that the generator $\mathrm A$ is an elliptic pseudodifferential operator, exactly as in the torus and Anosov cases and the analysis from \cite{chaubet2025levy} yields \eqref{eq:hitbounded}.
}

\subsection*{Related works}
Geodesic L\'evy flights on general Riemannian manifolds were first introduced in \cite{applebaum2000} where the formula for the generator $\A$ in \eqref{eq:generator} was derived. The work of \cite{tully2024levy} proved that, on {\em general} Riemannian manifolds, $\A$ as defined in \eqref{eq:generator} is a Fourier integral operator when time integral is over a {\em finite} interval. However, the only detailed analysis of $\A$ when the time integral is {\em infinite} is in \cite{chaubet20
25levy} where the authors showed that it is an elliptic pseudodifferential operator on negatively curved manifolds and the flat torus. As for the first arrival time function $u_{\star, \varepsilon}(p)$, the leading asymptotic was computed in \cite{chaubet2025levy} for Anosov manifolds, the round sphere, and the flat torus---on the sphere, it was also observed a special case of Theorem \ref{thm:main} (iv). 	

\subsection*{Organization of the article}

The article is organized as follows. In \S\ref{sec:zoll}, we study the geometry of conjugate loci on Zoll surfaces and derive normal forms for the exponential map near singular conjugate points. In \S\ref{sec:microlocal} we recall standard facts about pseudo-differential and Fourier integral operators. In \S\ref{sec:generator}, we analyze the generator of the Lévy process and describe the microlocal structure of a parametrix for its inverse. In \S\ref{sec:expected}, we express the expected stopping time as the solution of an integral equation involving the generator and prove Theorem~\ref{thm:main}. Finally, in Appendix~\ref{app:morin} we give normal forms for singular maps in dimension $2$ and in Appendix~\ref{app:topology} we recall classical facts on Zoll surfaces.

\subsection*{Acknowledgment}
Y. C. acknowledges the hospitality of University of Melbourne and the inspiration provided by Bells Beach during the preparation of this work. E.J.G.\  was partially supported by Melbourne Research Scholarship and by Stichting dr Hendrik Muller's Vaderlandsch Fonds. L.T. was partially supported by ARC DP260103195 and ARC DP220101808. This project was partially supported by an International Emerging Actions (IEA) grant funded by the CNRS.

\section{Preliminaries on Zoll surfaces}\label{sec:zoll}
In what follows we let $(\Sigma, g)$ be a closed oriented Riemannian surface and denote by $M = S\Sigma$ its unit tangent bundle.

\subsection{Geometry of the unit tangent bundle}\label{ssec:geometry-unit-tangent-bundle}
We recall here some classical facts about geometry of surfaces, see e.g. \cite{singer1976lecture}*{\S7.2}, \cite{PaternainSaloUhlmann2023}*{\S3.5.1} or \cite{chaubet_dynamical}*{\S10.4.1.1} . We have the Liouville one-form $\alpha \in \Omega^1(M)$ defined by
$$
\langle \alpha(z), w \rangle = \langle \dd \pi(z) \cdot w, v \rangle, \quad {z = } (p,v) \in M, \quad w \in T_{(x,v)}M.
$$
Then $\alpha$ is a contact form (that is, $\alpha \wedge \dd \alpha$ is a volume form on $M$) and it turns out that the geodesic vector field $X$ is the Reeb vector field associated to $\alpha$, that is, it satisfies
$$
\iota_X \alpha = 1, \quad \iota_X \dd \alpha = 0,
$$
where $\iota$ denote the interior product. We  set $\beta = R_{\pi/2}^*\alpha$ where for $\theta \in \RRR$, we denote by $R_{\theta} : M \to M$ the rotation of angle $\theta$ in the fibers {(which is defined thanks to the orientation of $\Sigma$)}. Then the volume form $\vol_g$ of $\Sigma$ satisfies
\begin{equation}\label{eq:volume}
  \pi^*\vol_g = \alpha \wedge \beta.
\end{equation}
We denote by $\psi$ the connection one-form, that is, the unique one-form on $M$ satisfying
\begin{equation}\label{eq:structural}
  \iota_V\,\psi = 1, \quad \dd \alpha = \psi \wedge \beta, \quad \dd \beta = \alpha \wedge \psi, \quad \dd \psi = -(\kappa \circ \pi) \alpha \wedge \beta,
\end{equation}
where $V$ is the vector field generating $(R_\theta)_{\theta \in \RR}$ and $\kappa$ is the Gauss curvature of $\Sigma$. Then $(\alpha, \beta, \psi)$ is a global frame of $T^*M$. We denote {by} $H$ the vector field on $M$ such that $(X, H, V)$ is the dual frame of $(\alpha, \beta, \psi)$. We then have the following commutation relations
\begin{equation}\label{eq:commutation}
  [V,X] = H, \quad [V,H] = -X, \quad [X, H] = (\kappa \circ \pi) V.
\end{equation}
The orientation of $M$ will be chosen so that $(X, H, V)$ is positively oriented.
For each $z \in M$, we shall denote by
$$
\mathbf V(z) = \ker \dd \pi(z) \subset T_zM
$$
the vertical bundle at $z$. In fact one has
$$
\mathbf V(z) = \RRR V(z) = T_z\Lambda_{\pi(z)} \quad \text{where} \quad \Lambda_{\pi(z)} = S_{\pi(z)} \Sigma \subset M
$$
is the set unit tangent vectors based at $\pi(z)$. Let $(\varphi_t)_{t \in \RRR}$ denote the geodesic flow, that is the flow generated by $X$.

For any $(q, \eta) \in T^*\Sigma \setminus 0$, we let $v_\eta \in S_q \Sigma$ be the unique unit vector such that the frame $(\eta^\sharp, v_\eta)$ is orthonormal and positively oriented, where $^\sharp : T^*_q\Sigma \simeq T_q\Sigma$ is the musical isomorphism. For any $(q,\eta)$ we will denote
$$
z_\eta^+ = (q, v_\eta) \quad \text{and} \quad z_\eta^- = (q, -v_\eta).
$$
    Let $\mathbf H^\star(z) = \ker V(z) \cap \ker X(z) = \RR \beta(z)$ for each $z$. The line bundle $\mathbf H^\star \to M$ is called the horizontal co-bundle on $M$. Then we have smooth isomorphisms
    \begin{equation}\label{eq:iotapm}
    \iota_{\pm} : T^*\Sigma \setminus \underline 0 \longrightarrow \mathbf H^\star \setminus \underline 0, \quad (q, \eta) \longmapsto (z_\eta^\pm, \dd \pi(z_\eta^\pm)^\top \eta).
    \end{equation}
    Denote by $\lambda_{T^*M}$ and $\lambda_{T^*\Sigma}$ the canonical Liouville one-forms on $T^*M$ and $T^*\Sigma$, respectively, and let $\lambda_{\mathbf H^\star}$ the restriction of $\lambda_{T^*M}$ to $\mathbf H^\star$. Then we have
     \begin{equation}\label{eq:iotapmlambda}
    \iota_\pm^*\lambda_{\mathbf H^\star} = \lambda_{T^*\Sigma}
    \end{equation}
    as it follows from straightforward verification. Finally we note that if $\widetilde R_\pi = (\dd R_\pi)^{-\top} : T^*M \to T^*M$ is the symplectic lift of $R_\pi$ then $\widetilde R_\pi$ preserves $\mathbf H^\star$ and 
    \begin{equation}\label{eq:iotapmrpi}
    \widetilde R_\pi \circ \iota_\pm = \iota_\mp.
    \end{equation}

\subsection{First conjugacy time}

\begin{definition}
  We say that a vector $z \in M$ is conjugate at time $t \in \RRR$ if the intersection
  \begin{equation}\label{eq:intersection}
    \Bigl(\dd \varphi_t(z) \cdot \mathbf V(z)\Bigr) \cap \mathbf V(\varphi_t(z))
  \end{equation}
  is not reduced to $\{0\}$.
\end{definition}

Since $\mathbf V(z) = \RRR V(z)$ is of dimension $1$, the condition \eqref{eq:intersection} is equivalent to the fact that there is $a \in \RRR^*$ such that
\begin{equation}\label{eq:dephitv}
\dd \varphi_t(z) V(z) = a V(\varphi_t(z)).
\end{equation}
Moreover, note that $\ker \alpha = \RRR H \oplus \RRR V$ is preserved by the geodesic flow. In particular $z$ is conjugate at time $t$ if and only if 
$$
\langle \beta(\varphi_t(z)), \dd \varphi_t(z)V(z)\rangle = 0,
$$ 
which reads
$
  \langle \varphi_{t}^*\beta(z), V(z) \rangle = 0. 
$
Another equivalent formulation is
\begin{equation}\label{eq:dphithstar}
\dd \varphi_t(z)^{-\top} \mathbf H^\star(z) = \mathbf H^\star(\varphi_t(z)).
\end{equation}

From now on we assume that $\Sigma$ is a Zoll surface of period $2\pi$, which means that the geodesic flow is periodic of period $2\pi$ and for any $z \in S\Sigma$, the map 
\begin{equation}\label{eq:gammaz}
\gamma_z : \RR/2\pi\ZZ \to \Sigma, \quad t \mapsto \pi \circ \varphi_t(z)
\end{equation}
is injective. In particular this implies that for every $z \in M$, one has $\dd \varphi_{2\pi}(z) = \I : T_zM \to T_zM$, hence according to \eqref{eq:intersection} there exists $t \in \RRR_+$ so that $z$ is conjugate at time~$t$. This allows us to make the following 

\begin{definition}
  For any $z \in M$, we denote by
  $$
  \tau(z) = \inf\{t \in \RRR_+~:~z \text{ is conjugate at time } t\}
  $$
  the first conjugacy time when starting from $z$.
\end{definition}

\begin{lemma}[Properties of the conjugacy time]\label{lem:conjtime}
    The map $\tau : M \to \RRR_+$ is smooth and takes  values in $\left]0, 2 \pi \right[$. Moreover we have $\tau \circ \mathrm R_\pi = 2\pi - \tau$ and a point $z \in M$ is conjugate at time $t \in \RRR$ if and only if $t \in 2 \pi \ZZ \cup (\tau(z) + 2 \pi \ZZ)$.
\end{lemma}

\begin{proof}
  Let $z \in M$ and $f_z(t) = \langle\varphi_t^*\beta(z),V(z)\rangle$. Then $f_z$ is smooth. Moreover \eqref{eq:structural} yields
  $$
  \partial_t \langle \varphi_t^*\beta, V\rangle = \langle \varphi_t^* \mathcal L_X \beta, V\rangle = \langle \varphi_t^* \psi, V\rangle
  $$
  and similarly $\partial_t \langle \varphi_t^* \psi, V\rangle = -(\kappa \circ \pi \circ \varphi_t) \langle \varphi_t^*\beta, V\rangle$. Hence we get, if $\kappa_z(t) = \kappa(\pi(\varphi_t(z)))$,
    \begin{equation}\label{eq:sturm}
    f_z''(t) + \kappa_z(t) f_z(t) = 0.
    \end{equation}
  In particular $f_z(0) = f_z(2\pi) = 0$ and $f'_z(0) = 1$. Since $f_z$ is periodic of period $2\pi$, we have that the first zero of $f_z$ (which is $\tau(z)$) must happen before $2\pi$ and satisfies $f'_z(\tau(z)) < 0$. By compactness and the implicit function theorem, we get that $z \mapsto \tau(z)$ is smooth. The last property follows from the identity $\mathrm R_\pi \varphi_t \mathrm R_\pi = \varphi_{-t}$ and the fact there is at most one conjugate time in $\left]0, 2\pi\right[$, see Appendix \ref{app:topology}.
\end{proof}

It will be useful in what follows to introduce the \textit{conjugacy map}
\begin{equation}
    \Phi : M \to M, \quad z \mapsto \varphi_{\tau(z)}(z).
    \label{eq:conjugacy-map}
\end{equation}

\begin{lemma}[Properties of the conjugacy map]\label{lem:conjmap}
    The following holds for the conjugacy map $\Phi$.
    \begin{enumerate}[label=\emph{(\roman*)}]
    \item The map $\Phi$ is a smooth diffeomorphism of $M$ with no fixed points and such that $\Phi^2 = \mathrm{Id}$ and $\Phi \circ R_\pi = R_\pi \circ \Phi$.
    \item There holds $\tau \circ \Phi + \tau = 2\pi$. 
    \item The symplectic lift $\widetilde \Phi = \dd \Phi^{-\top}$ of $\Phi$ induces an isomorphism
    $
    \widetilde \Phi : \mathbf H^\star\longrightarrow \mathbf H^\star
    $
    and 
    $$
    \widetilde \Phi^*\lambda_{\mathbf H^\star} = \lambda_{\mathbf H^\star}.
    $$
    \end{enumerate}
\end{lemma}
\begin{proof}
    By Lemma \ref{lem:conjtime} one has that $\Phi$ is smooth and $\tau \circ \Phi = 2\pi - \tau$. The latter relation combined with the fact that $\varphi_{2\pi} = \mathrm{Id}$ yields $\Phi^2 = \mathrm{Id}$. The fact that $\Phi$ has no fixed points follows immediately from the fact that $0 < \tau(z) < 2 \pi$ and our Zoll hypothesis. For the identity $\Phi \circ R_\pi = R_\pi \circ \Phi$ we simply notice that $\tau \circ R_\pi \circ \Phi = \tau$ which yields $\Phi^{-1} = R_\pi \Phi R_\pi$. Now we turn to the last point. It suffices to show that $\widetilde \Phi(\beta(z)) \in \RR \beta(\Phi(z))$ or equivalently that $\dd \Phi(z) : \ker \beta(z) \to \ker \beta(\Phi(z))$. First, we have
    $$
    \dd \Phi (z) X(z) = \partial_s|_{s = 0} \varphi_{\tau(\varphi_s(z))}(\varphi_s(z)) = (X\tau)(z)X(\Phi(z))  + \dd \varphi_{\tau(z)}(z) \cdot X(z) \in \RR X(\Phi(z)).
    $$
    since $X$ is preserved by $\dd \varphi_t$. On the other hand since $\tau(z)$ is the conjugacy time  \eqref{eq:dephitv} yields
    \begin{equation}\label{eq:ddphiv}
    \dd \Phi(z) V(z) = (V \tau)(z) X(\Phi(z)) + a V(\Phi(z)) \in \ker \beta(\Phi(z)).
    \end{equation}
    Since $\widetilde \Phi$ is a diffeomorphism on $T^*M$ its restriction to $\mathbf H^\star$ is also a diffeomorphism. Finally notice that $\widetilde \Phi^*\lambda_{T^*M} = \lambda_{T^*M}$ since it is the symplectic lift of $\Phi$. Since $\widetilde \Phi$ preserves $\mathbf H^\star \setminus \underline 0$ we get $\widetilde \Phi^*\lambda_{\mathbf H^\star} = \lambda_{\mathbf H^\star}$.
\end{proof}

For $p \in \Sigma$ we introduce the sets
%{\color{red} elements of the sphere bundle are sometimes called $z$ and sometimes called $(p,v)$. Consistency?}
$$
\widetilde{\mathrm{Conj}}(p) = \{\tau(p,v) v~:~v \in \Lambda_p\} \subset T_p \Sigma \quad \text{and} \quad {\mathrm{Conj}}(p) = \pi(\Phi(\Lambda_p)) \subset \Sigma
$$
which are called tangent conjugate locus and conjugate locus of $p$ respectively. We say that a pair $(p,q)$ is conjugate if $q \in \mathrm{Conj}(p)$. This is equivalent to the fact that $p \in \mathrm{Conj}(q)$ in which case we say that $p$ and $q$ are conjugate. Note that by Lemma \ref{lem:conjtime} and injectivity of \eqref{eq:gammaz}, we have
\begin{equation}\label{eq:pnotinconjp}
p\notin \mathrm{Conj}(p) \quad \text{for all }p \in \Sigma.
\end{equation}

\subsection{Degree of a conjugate vector and normal form of the exponential map}\label{subsec:degree}
As mentioned in the introduction, we define the degree of a conjugate point, which describes the singularity type of the exponential map near a conjugate vector. It is worth mentioning that this notion already appears implicitly in the work of Waters \cite{waters2017bifurcations, waters2019conjugate} on bifurcations of singularities of conjugate loci on surfaces.
\begin{definition}\label{def:degree}
    The \textit{degree} of $z \in M$ is defined by
    $$
    \deg(z) = \inf \bigl\{k \in \mathbb{N}~:~V^k \tau(z) \neq 0\bigr\} \in \mathbb{N} \cup \{\infty\}.
    $$
\end{definition}
Note that the degree is different from the \textit{order} of the conjugacy
$$
\mathrm{ord}(z) = \dim \Bigl(\dd \varphi_{\tau(z)}(z)\cdot \mathbf V(z) \cap \mathbf V\bigl(\varphi_{\tau(z)}(z)\bigr) \Bigr)
$$
which, in our case, is always equal to $1$ since $\dim \Sigma = 2.$
\begin{lemma}\label{lem:reversing_degrees}
    For any $z \in M$ there holds $\deg(z) = \deg(\mathrm R_\pi(\Phi(z))).$
\end{lemma}
\begin{proof}
    Let $\widehat \Phi= \mathrm R_\pi \circ \Phi$. Then $\tau \circ \widehat \Phi= \tau$, hence $V \tau = \dd \tau\bigl( \dd \widehat \Phi\cdot V\,\bigr)$. Now write
    \begin{equation}\label{eq:defab}
        \dd \varphi_t(z) \cdot V(z) = a(t,z) H(\varphi_t(z)) + b(t,z) V(\varphi(t,z))
    \end{equation}
    where $a, b \in \mathscr C^\infty(\RRR \times M)$. Note that by definition of $\tau$ we have
    $$
    a(\tau(z), z) = 0 \quad \text{and} \quad b(\tau(z), z) \neq 0 \quad \text{for all }z \in M.
    $$
    In particular one may compute
    $$
    \dd \Phi(z) \cdot V(z) = (V\tau)(z) X(\Phi(z)) + b(\tau(z), z) V(\Phi(z)).
    $$
    Since $(\mathrm R_\pi)_* X = -X$ and $(\mathrm R_\pi)_* V = V$ we get for any $u \in \mathscr C^\infty(M)$
    $$
    V(u \circ \widetilde \Phi) = -(V \tau) (Xu \circ \widetilde \Phi) + b_\tau (Vu \circ \widetilde \Phi)
    $$
    where $b_\tau$ is the map $z \mapsto b(\tau(z), z)$. By induction one obtains that for each $\ell \geqslant 1$ one has
    $$
    V^\ell(u \circ \widetilde \Phi) = \sum_{k=1}^\ell (V^k\tau) v_k + (b_\tau)^\ell (V^\ell u \circ \widetilde \Phi)
    $$
    for some $v_k \in \mathscr C^\infty(M)$ for $k = 1, \dots, \ell$. Now recall that $V^\ell\tau = V^\ell(\tau \circ \widetilde \Phi)$. In particular, since $b_\tau \neq 0$ we obtain
    $$
    \Bigl[V^k\tau(z) = 0\ \text{ for all } \ \quad k = 1, \dots, \ell \Bigr] \quad \implies \quad \Bigl[V^k\tau(\widehat \Phi(z)) = 0\  \text{ for all } \  k = 1, \dots, \ell \Bigr]
    $$
    Hence $\deg(\widetilde \Phi(z)) \geqslant \deg(z)$ for each $z$. However $\widetilde \Phi^2 = \mathrm{Id}$ so we get the reverse inequality. This completes the proof.
\end{proof}

\begin{proposition}[Normal form of the exponential map near conjugate points]\label{prop:normalformexp}
    Let $z = (p,v) \in M$ of degree $k \in \mathbb{N} \cup \{\infty\}$. Then there exist open neighbourhoods $U \subset T_p\Sigma$ and $V \subset \Sigma$ of $v$ and $q = \pi(\Phi(z)) \in \Sigma$ as well as diffeomorphisms $\kappa : U \to U' \subset \RRR^2$ and $\eta : V \to V' \subset \RRR^2$ such that $\kappa(v) = 0 = \eta(q)$, such that, setting
    $$
    \tilde f = \eta \circ \exp_p \circ\,\kappa^{-1} : U' \to V',
    $$
    the following holds.
    \begin{enumerate}[label=\emph{(\roman*)}]
        \item If $k = 1$, then $\tilde f(s,u) = (s, u^2)$.
        \item If $k \in \NN_{\geqslant 2}$, then $\tilde f(s,u) = (s, su + u^{k+1}\rho(s,u))$ for some smooth $\rho : \RR^2 \to \RR$ such that $\rho(0) = 1$.
        \item If $k = \infty$, then $\tilde f(s,u) = (s, su + \varphi(s,u))$ for some smooth $\varphi$ defined near the origin in $\RRR^2$ such that $\partial_u^\ell\varphi(0) = 0$ for each $\ell = 0,1,\dots$
    \end{enumerate}
\end{proposition} 
{
{}{}{}{}{}
This proposition can be viewed as a two-dimensional refinement of a classical result of Warner, see \cite[Theorem~3.3]{warner1965conjugateRiemann}. In arbitrary dimension, Warner proves an analogue of {(i)}, corresponding in our terminology to degree~$1$, together with a very special case of {(iii)}, namely the situation where $\deg(z)=\infty$ for all $z$ in an interval of $\Lambda_p$. Restricting to surfaces allows us to obtain a much more precise description of the local normal forms and to cover all finite degrees.
}

\begin{remark}
Although we do not need it for our purpose, it is natural to ask whether the function $\rho$ appearing in {(ii)}  can be chosen equal to $1$. For general $1$-singular maps, this turns out to be a non trivial question: indeed, if $k = 2$ or $k = 3$ then $\rho$ can be chosen equal to $1$ but for general $k \geqslant 4$ this is no longer the case. We refer to Remark \ref{rem:rieger} for a more detailed discussion.
\end{remark}

\begin{proof}
    By Proposition \ref{prop:normalform} it suffices to show that $f : T_p\Sigma \to \Sigma$ is $1$-singular at $v$ of degree $k$ where $f = \exp_p$. We identify $T_p\Sigma \setminus 0$ with $\RRR_+ \times \Lambda_p$ via $(r, v) \mapsto r v$. In those coordinates $\exp_p(r,v) = \pi(\varphi_r(p,v))$, hence we have, with $z = (p,v)$,
    $$
    \partial_r f(r,v) = \dd \pi(X(\varphi_r(z)))  \quad \text{and} \quad \partial_v f(r,v) = \dd \pi(\dd \varphi_r(z) \cdot V(z)).
    $$
    Now letting $\gamma_z : \RRR \to \Sigma$ be the geodesic starting $z$, one gets, with the notations of \eqref{eq:defab},
    $$
    \partial_r f(r,v) = \dot \gamma_z(r) \quad \text{and} \quad \partial_v f(r,v) = a(r,z) \dot \gamma_z(r)^\perp.
    $$
    Hence for $0 < r < 2\pi$ we have that $\dd f(r,v)$ is not invertible if and only if $r = \tau(z)$. Let $C = \{(r,v)~:~\tau(z) = r\}$ be the critical set. Note that, with the notation of the proof of Lemma \ref{lem:conjtime} one has $a(r,z) = f_z(r)$, hence 
    $$
    a(r,z) = 0 \quad \implies \quad \partial_ra(r,z) = f_z'(\tau(z)) \neq 0.
    $$
    This implies that $(r,v)$ is a regular point of $\det \dd f$ for any $(r,v) \in C$ (here we take any local frame to compute the determinant). In particular $f$ is $1$-singular at every $(r,v) \in C$. Note that on $C$ one has $\ker \dd f(r,v) = \RRR \partial_v$. Moreover, the map $\tilde \lambda : (r,v) \mapsto \tau(z) - r$ is a defining function for $C$, that is $C = \{\tilde \lambda = 0\}$ and $\dd \lambda|_C \neq 0$. Hence by Lemma \ref{lem:degree} the degree of $f$ at $(r,v) \in C$ is given by
    $$
    \inf\bigl\{k \in \mathbb N~:~\partial_v^k \lambda(r,v) \neq 0\bigr\} = \inf \bigl\{k \in \mathbb{N}~:~V^k \tau(z) \neq 0\bigr\} = \deg(z).
    $$
    This completes the proof.
\end{proof}
{}{}{}{
We conclude this paragraph by the following remark, which justifies the definition of the degree we gave in the introduction.
\begin{remark}
Let $p_\star\in\Sigma$. The conjugate locus of $p_\star$ is naturally
parametrized by
\begin{equation}\label{eq:param}
\Lambda_{p_\star}\longrightarrow \mathrm{Conj}(p_\star),
\qquad
z\longmapsto \exp_{p_\star}(\tau(z)z)=\pi(\Phi(z)).
\end{equation}
Let $z_\star\in\Lambda_{p_\star}$ and assume that $\deg(z_\star)=k\in\mathbb N$.
By Proposition~\ref{prop:normalformexp}, one can choose local coordinates
$(s,u)$ near $\tau(z_\star)z_\star$ in $T_{p_\star}\Sigma$ such that the
lifted conjugate locus $\widetilde{\mathrm{Conj}}(p_\star)$ is locally given by
$$
s+k u^k\rho(s,u)+u^{k+1}\partial_u\rho(s,u)=0
$$
with $\rho(0) = 1$. The implicit function theorem ensures that there exists a smooth function $s = s(u)$ defined in a neighborhood of $0$ such that
$$
s(u)=-k u^k\rho(s(u),u)-u^{k+1}\partial_u\rho(s(u),u),
$$
which gives $s(u) = -ku^{k} + \mathcal O(u^{k+1})$.
The parametrization of $\mathrm{Conj}(p_\star)$ induced by
\eqref{eq:param} has, in suitable local coordinates at
$\exp_{p_\star}(\tau(z_\star)z_\star)$, the form
$$
c(u)=\bigl(s(u),s(u)u + u^{k+1}\rho(s(u),u)\bigr).
$$
In particular, after a linear change of coordinates and a reparametrization
of $u$, its $(k+1)$-jet is given by
$
(u^k,u^{k+1}).
$
\end{remark}
}

\subsection{Construction of Zoll surfaces with conjugate pairs of arbitrary degree}
For any $p,q \in \Sigma$ we denote by $\Lambda_{p \to q} = \Lambda_p \cap \Phi^{-1}(\Lambda_q)$ the set of conjugate vectors at $p$ reaching $q$ at the conjugacy time. 
\begin{definition}\label{def:degreepair}
Let $(p, q)$ be a conjugate pair. The \textit{degree} of $(p,q)$ is defined as 
$$
\deg(p,q) = \max_{z \in \Lambda_{p\to q}} \deg(z) = \max_{z \in \Lambda_{q \to p}} \deg(z).
$$
\end{definition}
Note that the degree is well defined thanks to Lemma \ref{lem:reversing_degrees}. As we shall see, the degree of $(p,q)$ governs the asymptotic behaviour of the expected time need for a Levy process $(X_t)$ to reach a small ball centred at $q$. The goal of this section is to prove the following result.
\begin{theorem}\label{thm:alldegree}
Let $k \in \mathbb N_{\geqslant 1}$. Then there exists a Zoll surface $(\Sigma, g)$ and $p_\star,q_\star \in \Sigma$ such that $\deg(p_\star,q_\star) = k$.
\end{theorem}
To that end, we will use a transversality argument together with a result of Guillemin \cite{guillemin1976radon}, which provides an infinite dimensional family of Zoll metrics near the round sphere parametrized by odd functions on $\mathbb S^2$. More precisely, let
$
\mathscr C^\infty_{\mathrm{odd}}(\mathbb S^2)
$
denote the space of smooth odd functions on $\mathbb S^2$, that is, functions satisfying
$
\rho(p)=-\rho(Jp)
$
for $p \in \mathbb S^2$ where $J$ is the antipodal map. We denote by $\bar g$ the round metric on $\mathbb S^2$. Guillemin's result then reads as follows.

\begin{theorem}[\cite{guillemin1976radon}*{Theorem 1}]\label{thm:guillemin}
There is a smooth map $\sigma : \mathscr C^\infty_{\mathrm{odd}}(\mathbb S^2) \to \mathscr C^\infty(\mathbb S^2)$ such that for any $\rho \in \mathscr C^\infty_{\mathrm{odd}}(\mathbb S^2)$ the metric $e^{2\sigma(\rho)} \bar g$ is a Zoll metric, where $\bar g$ is the round metric on the sphere. Moreover $\sigma(0) = 0$ and for $\rho \in \mathscr C^\infty_{\mathrm{odd}}(\mathbb S^2)$ we have 
\begin{equation}\label{eq:tangentzoll}
\partial_a|_{a=0} \sigma(a\rho) = \rho.
\end{equation}
\end{theorem}

\begin{proof}[Proof of Theorem \ref{thm:alldegree}]
Let $k \geqslant 2$. We will construct $\Sigma$ as a small perturbation of the round metric on $\mathbb S^2$. Fix a point $p_\star \in \mathbb S^2$. Let $g$ be a metric on $\mathbb S^2$ and $S_g\Sigma = \{(p,v)\in T\Sigma~:~|v|_g = 1\}$ the associated unit tangent bundle. We identify $({T_{p_\star}\mathbb{S}^2}\setminus \underline 0) / \RR_+$ with $\mathbb S^1 = R/2\pi\mathbb Z$ hence this yields an identification of the fiber $S_{p_\star, g} \Sigma = \{(p_\star, v) \in T_{p_\star} \Sigma~:~|v|_g=1\}$ with $\mathbb S^1$. For any $\theta \in \mathbb S^1$ we denote by $\gamma_{g,\theta} : \mathbb \RR \to \mathbb S^2$ the geodesic starting from $p_\star$ in the $\theta$ direction and by $\tau_g(\theta) \in \RR_+ \cup \{\infty\}$ the first positive zero of the unique solution $f_{g,\theta}$ of the equation
\begin{equation}\label{eq:sturm-infinitesimal}
f_{g, \theta}''(t) + \kappa_{g}(\gamma_{g,\theta}(t)) f_{g, \theta}(t) = 0 \quad \text{with} \quad f_{g,\theta}(0) = 0 \quad \text{and} \quad f_{g, \theta}'(0) = 1,
\end{equation}
where $\kappa_{g}(\gamma_{g, \theta}(t))$ is the Gauss curvature of $(\mathbb S^2, g)$ at $\gamma_{g, \theta}(t)$. Recall from \eqref{eq:sturm} that if $(\mathbb S^2, g)$ is Zoll then $0 < \tau_g(\theta) < 2\pi$ is the first conjugacy time in the direction $\theta$. 

Next, let $(g_a)_{a \in \RR}$ be a smooth family of metrics with $g_0= \bar g$. We fix $\theta_\star \in \mathbb S^1$ and define $\dot f, \dot \kappa : \RR \to \RR$ by
$$
\dot f(t) = \partial_a|_{a = 0} f_{g_a,\theta_\star}(t) \quad \text{and} \quad \dot \kappa(t) = \partial_a|_{a=0} (\kappa_{g_a} \circ \gamma_{g_a, \theta_\star})(t) = (\partial_a|_{a=0} \kappa_{g_a})(\gamma_{\bar g, \theta_\star}(t)).
$$
The last equality is due to the fact that $\kappa_{\bar g} = 1$ is constant. Note also that $f_{\bar g, \theta_\star}(t) = \sin(t)$. Hence \eqref{eq:sturm-infinitesimal} yields 
$$
\dot f''(t) + \dot f(t) = - \dot \kappa(t) \sin(t), \quad \dot f(0) = 0, \quad \dot f'(0) = 0.
$$
Solving the above equation one sees that $\dot f$ is given by
\begin{equation}\label{eq:dotf}
\dot f(t) = \int_0^t \dot \kappa(s) \sin(s) \sin(s-t) \dd s
\end{equation}
Next, let $\dot \tau(\theta_\star) = \partial_a|_{a = 0} \tau_{g_a}(\theta_\star)$. We have $\tau_{\bar g}(\theta_\star) = \pi$ and $f_{\bar g, \theta_\star}'(\pi) = -1$. Hence differentiating the equation $f_{g_a, \theta_\star}(\tau_{g_a}(\theta_\star)) = 0$ we get $-\dot \tau + \dot f(\pi) = 0$. Inserting this in \eqref{eq:dotf} yields
\begin{equation}\label{eq:dottau1}
\dot \tau(\theta_\star) = \int_0^\pi \dot \kappa(s) \sin(s)^2 \dd s.
\end{equation}
In the particular case where $a \mapsto \sigma_a$ is a smooth family of smooth functions on the sphere with $\sigma_0 = 0$ and $g_a = e^{2\sigma_a} \bar g$ then we have $\kappa_{g_a} = e^{-2\sigma_a}(1 - \Delta_{\bar g} \sigma_a)$ hence $\partial_{a}|_{a=0} \kappa_{g_a} = -(2 + \Delta_{\bar g}) \rho$ where $\rho = \partial_a|_{a=0}\sigma_a$. Putting everything together yields
\begin{equation}\label{eq:dottau}
\dot \tau(\theta_\star) = -\int_0^\pi \sin(s)^2 (2 + \Delta_{\bar g}) \rho(s, \theta_\star) \dd s
\end{equation}
where we use the radial coordinates based at $p_\star,$ that are given by
$$
\left]0,\pi\right[ \times \mathbb S^1 \to \mathbb S^2, \quad (s, \theta) \mapsto \gamma_{\bar g, \theta}(s).
$$
In those coordinates, we have $\bar g = \dd s^2 + \sin^2(s) \dd \theta^2$ hence
\begin{equation}\label{eq:Delta}
\Delta_{\bar g} = \partial_s^2 + \cot(s)\partial_s + \sin(s)^{-2} \partial_\theta^2.
\end{equation}
For $\chi \in \mathscr C^\infty_c(\left]0, \pi\right[)$ note that 
\begin{equation}\label{eq:ipp}
\int(2\chi(s) + \chi''(s) + \cot(s) \chi'(s)) \sin^2(s) \dd s = \int \chi(s) \dd s
\end{equation}
as it follows by integration by parts. Now we let $\varrho \in \mathscr C^\infty(\mathbb S^1)$ and we set
$$
\rho(s,\theta) = \chi(s) \varrho(\theta), \quad (s, \theta) \in \left]0, \pi\right[\times \mathbb S^1.
$$
We moreover ask that 
\begin{equation}\label{eq:ask}
\varrho(\theta + \pi) = - \varrho(\theta), \quad \chi(\pi - s) = \chi(s) \quad \text{and} \quad \int \chi(s) \dd s = -1.
\end{equation}
Then $\rho \in \mathscr C^\infty_{\mathrm{odd}}(\mathbb S^2)$ and combining \eqref{eq:dottau}, \eqref{eq:Delta} and \eqref{eq:ipp} we get
\begin{equation}\label{eq:dottau2}
\dot \tau(\theta_\star) = \varrho''(\theta_\star) + \varrho(\theta_\star).
\end{equation}
In what follows for any $k \in \mathbb N$ and $\theta \in \mathbb S^1$ we will denote
$$ 
v_{k, \varrho}(\theta) = \bigl(\varrho^{(\ell)}(\theta) + \varrho^{(\ell + 2)}(\theta)\bigr)_{\ell = 1, \dots, {k}}  \in \RR^{k}.
$$
\begin{lemma}\label{lem: span condition}
For every $k = 2,3, \dots$ there exists $N \in \mathbb N$ and a family of functions $(\varrho_j)_{j=1,\dots, 2N}$ satisfying the first condition of \eqref{eq:ask} such that the following holds.
\begin{enumerate}[label=\emph{(\roman*)}]
\item For any $1 \leqslant j \leqslant N$ we have
\begin{equation}\label{eq:important}
\varrho_j^{(\ell)}(\theta_\star) = 0 \quad \text{ for every } \ell = 1, \dots, k + 1. 
\end{equation}
Moreover we have the non-degeneracy condition
 \begin{equation}\label{eq:nondeg}
\underset{j = 1, \dots, N}{\mathrm{span}} v_{k, \varrho_j}(\theta) = \RR^{k} \quad \text{whenever } \theta \neq \theta_\star, J\theta_\star.
\end{equation}
%\begin{equation}\label{eq:important}
%\varrho_j^{(\ell)}(\theta_\star) = 0 \quad \text{ for every } \ell = 1, \dots, k + 1 \qquad 
%\end{equation}
%and which satisfies the nondegeneracy condition
% \begin{equation}\label{eq:nondeg}
%\underset{j = 1, \dots, N}{\mathrm{span}} v_{k, \varrho_j}(\theta) = \RR^{k} \quad \text{whenever } \theta \neq \theta_\star, J\theta_\star.
%\end{equation}
\item The family $(\varrho_j)_{N+1\leqslant j \leqslant 2N}$ is non degenerate at $\theta_\star$, that is
 \begin{equation}\label{eq:spanthetastar}
\underset{j = N+1, \dots, 2N}{\mathrm{span}} v_{k, \varrho_j}(\theta_\star) = \RR^{k}.
\end{equation}
\end{enumerate}
\end{lemma}
\begin{proof}[Proof of Lemma \ref{lem: span condition}] Let $k \geqslant 2$ and put $m = k + 2$. Then for any $\varrho \in \mathscr C^\infty(\mathbb S^1)$ one has
$$
v_{k, \varrho}(\theta) = h(w_{m,\varrho}(\theta)) \quad \text{where} \quad w_{m,\varrho}(\theta) = (\varrho^{(\ell)}(\theta))_{\ell = 1, \dots, m}
$$
for some surjective linear map $h \in \mathscr L(\RR^m, \RR^k)$. Hence for a finite family $(\varrho_j)_j$ we have
 \begin{equation}\label{eq:wm}
{\mathrm{span}}_{j} w_{m, \varrho_j} = \RR^m \quad \implies \quad {\mathrm{span}}_j v_{k, \varrho_j} = \RR^k.
\end{equation}
In what follows, we will say that a finite family $(\varrho_j)$ of functions is non degenerate at $\theta$ if the family $(v_{k, \varrho_j}(\theta))_j$ spans $\RR^k$. From \eqref{eq:wm} this will be true as soon as $(w_{m, \varrho_j}(\theta))_j$ generates $\RR^m$. From this it is clear that we can find a family $(\varrho_j)_{j=N+1, \dots, 2N}$ so that (ii) holds.

Next, we find a family $(\varrho_j)_{j=1,\dots,N}$ satisfying (i). First, remark that for any $\theta \neq \theta_\star, J\theta_\star$, it is not hard to find a finite family $(\varrho_j)$ of functions supported far from $\theta_\star$ and $J\theta_\star$, satisfying the first condition in \eqref{eq:ask} and which is non degenerate at $\theta$. By continuity, this family is also non degenerate at every $\theta'$ close enough to $\theta$. Hence to prove (i) it suffices to find a finite family $(\varrho)_j$ of functions which satisfies \eqref{eq:important} and which is non degenerate at every $\theta$ in a neighborhood of $\{\theta_\star, J\theta_\star\}$. To that end, take any family $(\varrho_j)_{j=1, \dots, m}$ of odd functions on $\mathbb S^1$ such that $\varrho_j(s) = s^{m+j}$ near $s=0$, where $s = \theta - \theta_\star$. Then for each $j = 1, \dots, m$ one has
$$
w_{m,\varrho_j}(s) = \Bigl((m+j) \cdots (m+j - \ell + 1)s^{m+j-\ell}\Bigr)_{\ell = 1, \dots, m}.
$$
Let $W(s)$ be the $m \times m$ matrix whose $(j, \ell)$ entry is $(m+j) \cdots (m+j - \ell + 1)s^{m+j-\ell}$. Then from a direct computation we see that
$$
\det W(s) = s^{m^2}\frac{(2m)!}{m!}\prod_{r=1}^{m-1} r!.
$$
In particular the family $(w_{m,\varrho_j}(\theta))_{j=1, \dots, m}$ is not degenerate for any $\theta$ close to $\theta_\star$. Hence we obtain (i) and the lemma is proven.
%In particular to 
%i) By a partition of unity argument, it suffices to show that we can construct a family of functions $\{\varrho_j\}_{j=1,\dots, N}$, $\varrho_j: \mathbb R\to \mathbb R$ and an $\epsilon>0$ such that for each $j=1,\dots, N$, $0=\varrho_j^{(\ell)}(0)$ for $\ell = 1,\dots, k+1$  while ${\rm span}_{j = 1,\dots, N}(v_{k, \varrho_j}(x)) = \mathbb R^k$ for all $x\in (-\epsilon, \epsilon)\setminus 0$.
%To this end, let $\varrho_j(x) := x^{k+2+j}$ with $j = 1,\dots, k+2$. Then 
%$$v_{k, \varrho_j}(x) = (\varrho^{(\ell)}_{j}(x) + \varrho_j^{(\ell + 2)}(x))_{\ell = 1,\dots,k} =\left(\frac{(k+2+j)!}{(k+2+j-\ell)!} x^{k+2+j-\ell} + \frac{(k+2+j)!}{(k+j-\ell)!} x^{k+j-\ell}\right)_{\ell = 1,\dots,k} .$$
\end{proof}
Given Lemma \ref{lem: span condition}
For any $\mathbf a = (a_j) \in \RR^{2N}$, we denote 
$$
g_\mathbf a = e^{2 \sigma(\rho_\mathbf a)} \overline g \quad \text{where} \quad \rho_\mathbf a(s, \theta) = \chi(s) \sum_{j=1}^{2N} a_j \varrho_j(\theta).
$$
Here $\sigma$ is Guillemin's map from Theorem \ref{thm:guillemin}. Next, introduce the map
$$
F_{k-1} : \RR^{2N} \times \mathbb S^1 \to \RR^{k-1}, \quad (\mathbf a, \theta) \mapsto \bigl(\partial_\theta \tau_{g_\mathbf a}(\theta), \dots, \partial_\theta^{k-1} \tau_{g_\mathbf a}(\theta)\bigr),
$$
and define $F_{k-1, \star} : \RR^{2N} \to \RR^{k-1}$ by $F_{k-1, \star}(\mathbf a) = F_{k-1}(\mathbf a, \theta_\star)$. Then by \eqref{eq:dottau2} and \eqref{eq:tangentzoll} one has
\begin{equation}\label{eq:differential}
\dd F_{k-1, \star}(0) \cdot \mathbf a = \sum_{j=1}^{2N} a_jv_{k-1, \varrho_j}(\theta_{\star}).
\end{equation}
Thus by \eqref{eq:spanthetastar} the map $F_{k-1, \star}$ is a submersion near $0$ and there is $\varepsilon > 0$ such that 
\begin{eqnarray}\label{eq: def of Mstar}
M_\star = F_{k-1, \star}^{-1}\{0\} \cap B(0, \varepsilon)
\end{eqnarray}
is a smooth submanifold of $\RR^{2N}$ of co-dimension $k-1$. 
Moreover the linear map
$$
\dd F_{k,\star}(0)|_{T_0M_\star}
:
T_0M_\star \to \RR^{k}
$$
is not the zero map. To see this, simply let $G_{k} \subset \RR^{2N}$ be any subspace such that $G_k \oplus \ker \dd F_{k,\star}(0) = \RR^{2N}$. Then by \eqref{eq:spanthetastar} and \eqref{eq:differential} we have $\dim G_k = k$. However 
$$
\dim T_0M_\star = \dim \ker \dd F_{k-1,\star}(0) = 2N - k + 1.$$
Hence $G_k$ must intersect $T_0M_\star$ and $\dd F_{k,\star}(0)|_{T_0M_\star}$ is not the zero map. Hence there exists $\mathbf v\in T_0M_\star$ such that
$
\dd F_{k,\star}(0)\mathbf v\neq0.
$
Next, since the map
$
\theta \mapsto \partial_{\mathbf a}F_{k}(0,\theta)\mathbf v
$
is continuous and non-zero at $\theta=\theta_\star$, taking $\delta>0$ small we have
$
\partial_{\mathbf a}F_{k}(0,\theta)\mathbf v\neq0
$
for every $\theta$ such that $|\theta-\theta_\star|<\delta$.
Since 
$
\tau_{g_{\mathbf a}}(\theta+\pi)=2\pi - \tau_{g_{\mathbf a}}(\theta),
$
 the same non-vanishing property holds near $J\theta_\star$. Therefore, after choosing an open cone
$\widetilde C\subset T_0M_\star$ around $\mathbf v$ small enough and reducing $\delta$, there is $c > 0$ such that
\begin{equation}\label{eq:linearized}
|\partial_{\mathbf a}F_{k}(0,\theta)\mathbf w| \geqslant c|\mathbf w|
\quad \text{ for every}
\quad
(\mathbf w,\theta)\in \widetilde C\times I_{\star,\delta}
\end{equation}
where $I_{\star,\delta}=\{\theta:\min(|\theta-\theta_\star|,|\theta-J\theta_\star|)<\delta\}$. Note that $M_\star$ inherits a Riemannian structure from $\RR^{2N}$ and we denote by $\psi$ the exponential map $U \to M_\star$ defined on a neighborhood $U \subset T_0M_\star$ of $0$. We let $C = \psi(\widetilde C \cap U) \subset M_\star$. Now we write for $\mathbf a \in M_\star$
$$
F_k(\mathbf a, \theta) = \widetilde F_k((\psi^{-1}(\mathbf a), \theta) \quad \text{where} \quad \widetilde F_k(\mathbf v, \theta) = F_k(\psi(\mathbf v), \theta).
$$
Since $\dd \psi(0) = \mathrm{Id}$ we have for $\mathbf w \in T_0M_\star$ and $\theta \in \mathbb S^1$
\begin{equation}\label{eq:comparetangent}
\partial_\mathbf v \widetilde F_{k}(0, \theta) \mathbf w = \partial_\mathbf a F_k(0, \theta) \mathbf w.
\end{equation}
On the other hand, uniformly for $\theta \in \mathbb S^1$ and $\mathbf a \in \RR^{2N}$ we have
$$
F_{k}(\mathbf a,\theta)
=
\partial_{\mathbf a}F_{k}(0,\theta)\mathbf a
+
o(|\mathbf a|).
$$
Combining this with \eqref{eq:comparetangent} and \eqref{eq:linearized} we obtain
$$
\widetilde{F}_{k}(\mathbf v,\theta)\neq 0 
\quad \text{ for every}
\quad
(\mathbf v,\theta)\in(\widetilde C\setminus 0)\times I_{\star,\delta}.
$$
Equivalently, setting $C = \psi(\widetilde C) \subset M_\star$, one obtains
\begin{equation}\label{eq:positive}
F_{k}(\mathbf a,\theta)\neq0 
\quad \text{ for every}
\quad
(\mathbf a,\theta)\in(C \setminus 0)\times I_{\star,\delta}.
\end{equation}
%%%
Next, we consider the restriction
$$
G : M_\star \times \mathbb S^1 \setminus I_{\star, \delta} \to \RR^{k}, \quad (\mathbf a, \theta) \mapsto F_{k}(\mathbf a, \theta),
$$
of $F_{k}$ on $M_\star \times \mathbb S^1 \setminus I_{\star, \delta}$. Then for each $\theta \notin I_{\star, \delta}$ the map
$$
\partial_\mathbf a G(0, \theta)|_{T_0M_\star} : T_0M_\star \to \RR^{k+1}
$$
is surjective. Indeed, $T_0M_\star = \ker \dd F_{k-1, \star}(0)$. Hence by \eqref{eq:differential} and \eqref{eq:important} we have that $\mathbf a \in T_0M_\star$ whenever $a_j = 0$ for $j > N$. The surjectivity of $\partial_\mathbf a G(0, \theta)|_{T_0M_\star}$ then follows from \eqref{eq:nondeg}. By compactness of $\mathbb S^1 \setminus I_{\star, \delta}$ we obtain that up to reducing $\varepsilon$ in \eqref{eq: def of Mstar}, the map $G: M_\star \times \mathbb S^1 \setminus I_{\star,\delta}\to \mathbb R^k$ is a submersion. This means that it is transversal on $\{0\} \subset \RR^{k}$ in the sense of \cite{laudenbach2011transversalite}*{\S5.1.1}. Hence by \cite{laudenbach2011transversalite}*{Théorème p. 93} for almost every $\mathbf a \in M_\star$ the map 
$$
F_{k}(\mathbf a, \cdot) : \mathbb S^1 \setminus I_{\star, \delta} \to \RR^{k}
$$
is also transversal on $\{0\} \subset \RR^{k}$. However since $k\geq 2$ by assumption, we have $\dim(\mathbb S^1 \setminus I_{\star, \delta}) = 1 < k$ and thus 
\begin{equation}\label{eq:want}
F_{k}(\mathbf a, \theta) \neq 0 \quad \text{for any }\theta \in \mathbb S^1 \setminus I_{\star, \delta}.
\end{equation}
The set of parameters $\mathbf a\in M_\star$ satisfying
\eqref{eq:want} has full measure in $M_\star$. Since
$C\setminus\{0\}$ is a non-empty open subset of
$M_\star$, it intersects this full-measure set. We choose
$\mathbf a$ in this intersection. Since $\mathbf a \in M_\star$, we obtain by \eqref{eq:positive} 
$$
F_{k-1}(\mathbf a, \theta_\star) = 0 \qquad \text{and} \qquad F_{k}(\mathbf a, \theta) \neq 0 \quad \text{for every } \theta \in \mathbb S^1.
$$
This means that for the Zoll surface $(\mathbb S^2, g_{\mathbf a})$, every conjugate vector has degree at most $k$, and the vector corresponding to $\theta_\star$ has degree exactly $k$. Setting $q_\star = \pi(\Phi(p_\star, \theta_\star))$ we obtain $\deg(p_\star, q_\star) = k$.
 
 Finally, consider the map $\omega : \mathbb S^1 \to \Sigma$ given by $\theta \mapsto \pi(\Phi(p_\star, \theta))$. Its image is the union of conjugate points of $p_\star$. Moreover by \eqref{eq:ddphiv} its derivative vanishes precisely where $\partial_\theta\tau_{g_\mathbf a}$ vanishes. By construction, every zero of $\partial_\theta\tau_{g_\mathbf a}$ is of finite order at most $k$, hence isolated. Thus there is a finite number of zeros. Hence there exists a point $q$ in the image of $\omega$ which is not the image by $\omega$ of a zero of $\partial_\theta\tau_{g_\mathbf a}$. Such a point satisfies $\deg(p_\star, q) = 1$.
\end{proof}

We conclude this section with the following result showing that there are situations satisfying point (iii) of Theorem \ref{thm:main} but not (iv).
\begin{lemma}
\label{lem:infinitedegree}
There exists a Zoll surface $(\Sigma, g)$ and $p_\star,q_\star \in \Sigma$ such that $\deg(p_\star,q_\star)=\infty$ but $\mathrm{Conj}(p_\star) \neq \{q_\star\}.$
\end{lemma}
\begin{proof}
{}{}{}{
To obtain a Zoll metric with conjugate points $p_\star \sim q_{\star}$ of infinite degree, we will create Zoll metrics of revolution as given in \cite{Besse1978Zoll}*{Corollary 4.16}. Consider the north and south poles $\mathrm n, \mathrm s \in \mathbb S^2$. We will consider the parametrization
$$
 \left]0, \pi \right[ \times \left[0, 2\pi\right[ \to \mathbb S^2 \setminus \{\mathrm n, \mathrm s\}, \quad (r, \vartheta) \mapsto (\sin r\cos\vartheta,\sin r\sin\vartheta,\cos r).
$$
It is shown in \cite{Besse1978Zoll}*{Corollary 4.16} that if $h: \left]-1,1\right[ \to [-1,1] $ is a smooth odd function such that $h(1) = 0 $, then the metric given in the above coordinates by
\begin{equation}\label{eq:zoll-metric-revolution}
    g = [1+ h(\cos(r))]^2 \dd r^2 + \sin^2 (r) \dd \vartheta^2  
\end{equation}
is a Zoll metric of revolution. 
Note that for $h=0 $ one obtains the round metric. In what follows we will construct an odd function $h$ so that for $(\mathbb S^2, g)$ there is a point $p_\star \in \mathbb S^2$ which a conjugate point of infinite order but whose conjugate locus is not reduced to a point.

To that end, let $0< \delta < 1/2 $ and let $\tilde h: \left[0,1\right[ \to [0,1]$ be smooth with $\supp \tilde h = [\delta, \frac{1}{2}]$. Let ${h:\left]-1,1\right[ \to [-1,1]}$ be the unique odd extension of $\tilde h$. 
Then the metric $g$ defined by \eqref{eq:zoll-metric-revolution}, gives a Zoll metric of revolution. The point $p_\star=(r_{p_\star},\vartheta_{p_\star}) = (\frac{\pi}{2},0)$ on the equator is conjugate to ${q_\star}=(r_{q_\star},\vartheta_{q_\star}) = (\frac{\pi}{2},\pi)$ along the geodesic defined by the equator. Note that the metric $g$ coincides with the round one near the equator. In particular the geodesics starting from $p_\star$ in a direction sufficiently close to that of the equator coincide with the geodesics of the round metric. This implies that on $(\mathbb S^2, g)$, the point $p_\star$ is conjugate with its antipodal point $q_\star = (\pi/2, \pi)$ with infinite degree. 

It remains to show that the conjugate locus of $p_\star$ is not reduced to $q_\star$. Let $\varphi \in \left[0,\pi\right[$ and consider the geodesic $\gamma$ of $(\mathbb S^2, g)$ starting at $p_\star$ in the direction making an angle $\varphi$ with the meridian $\{\vartheta = 0\}$. Denote by $p_c = (r_c,\vartheta_c)$ be the unique conjugate point of $p_\star$ along $\gamma$. Then by \cite[Proposition 4.35 and Equation 4.31]{Besse1978Zoll} one has the formula
\begin{equation*}
    \begin{split}
        \vartheta_c &= 2\pi - \arccos\left( \frac{\tan\varphi}{\tan r_{p_\star}} \right) - \arccos\left( \frac{\tan \varphi}{\tan r_c } \right) \\
        &\qquad + \int_{r_{p_\star}}^{\pi-\varphi}\frac{\sin(\varphi)\cdot  h(\cos\rho)}{\sin\rho(\sin^2\rho-\sin^2\varphi)^{1/2}}\dd \rho + \int_{r_c}^{\pi-\varphi}\frac{\sin(\varphi)\cdot  h(\cos\rho)}{\sin\rho(\sin^2\rho-\sin^2\varphi)^{1/2}}\dd \rho \ .
    \end{split}
\end{equation*}
In particular, if we assume that $r_c = r_{p_\star} = \pi/2$, then $\vartheta_c$ is given by 
\begin{equation}\label{eq:varthetac}
    \vartheta_c = \pi + 2 \int_{\pi/2}^{\pi-\varphi}\frac{\sin(\varphi)\cdot  h(\cos\rho)}{\sin\rho(\sin^2\rho-\sin^2\varphi)^{1/2}}\dd \rho.
\end{equation}
Choose $\varphi$ so that the interval
$
\left[\frac{\pi}{2},\pi-\varphi\right]
$
intersects the support of $h(\cos\rho)$. Since $h$ is the odd extension of a non-negative function and is not identically zero, we may furthermore choose $\varphi$ so that $h(\cos\rho)$ does not change sign and is not identically zero on the interval of integration. It follows that the integral in \eqref{eq:varthetac} does not vanish.
%\[
%\int_{\pi/2}^{\pi-\varphi}
%\frac{\sin(\varphi)\, h(\cos\rho)}
%{\sin\rho(\sin^2\rho-\sin^2\varphi)^{1/2}}
%\,\dd \rho
%\neq 0.
%\]
Hence if $r_c=\pi/2$ then
$
\vartheta_c
\neq \pi.
$
Therefore $p_c\neq q_\star$. On the other hand, if $r_c\neq \pi/2$, we have of course $p_c\neq q_\star$. Thus, for this choice of $\varphi$, the conjugate point $p_c$ of $p_\star$ along $\gamma$ is different from $q_\star$. We have therefore exhibited a conjugate point of $p_\star$ distinct from $q_\star$, which completes the proof.
%Since $h(\cos(\rho)) \neq 0$ and doesn't change sign for $\rho \in (\frac{\pi}{2}+ \varepsilon, \frac{\pi}{2} + 2\varepsilon)$, meaning that $q_\star $ is not conjugate to $p_\star $ along geodesics starting at $p_\star $ with initial angle $\varphi = \frac{\pi}{2} + 2\varepsilon$.
%
%Hence for every $\varepsilon>0 $, there is a subinterval $I_\varepsilon = (\frac{\pi}{2}-\varepsilon,\frac{\pi}{2}+\varepsilon) \subset S_{p_\star}\Sigma$ of directions for which $q_\star $ is conjugate to $p_\star $ along geodesics starting in  these directions. Furthermore $p_{\star} , q_{\star}$ are conjugate of degree $\infty$. On the other hand, since  $p_\star $ is not conjugate to $q_\star $ along all geodesics, it follows that $\conj(p_\star) \neq \{q_\star \}$, and we have case (iii) of Theorem \ref{thm:main}.
%Now taking the limit as $\varepsilon\to 0$, we see that $h_0 $ defines a valid $C^\infty $ Zoll metric of revolution $g_0 $, and that   $\deg(p_{\star},q_{\star}) = \infty$, but there does not exists a subinterval $I \subset S_{p_\star}\Sigma$ of directions along which $q_\star $ is conjugate to $p_{\star}$. 
}
\end{proof}

\section{Background material on microlocal analysis}
\label{sec:microlocal}

In this section we recall some basic facts about microlocal analysis.

\subsection{Pseudodifferential operators}\label{sssection:pdo}

We refer to \cites{Grigis-Sjostrand-94, hormander2007PDE3,hormander2009PDE4,Lefeuvre2025} for a general treatment. Let $\Sigma$ be a closed $n$-dimensional manifold. For $k \in \RR$ we define $S^k(T^*\Sigma) \subset C^\infty(T^*\Sigma)$ the space of symbols of order $k$ as the set of smooth functions $a$ satisfying the following bounds, in any coordinate chart $U \subset \RR^n$: for all $\gamma,\beta \in \mathbb{N}^n$, there exists $C = C(U,\alpha,\beta) > 0$ such that
\begin{equation}
\label{equation:symbol}
\forall (x,\xi) \in T^* U \simeq \RR^n \times \RR^n, \qquad |\partial^\gamma_\xi \partial^\beta_x a(x,\xi)| \leqslant C \langle \xi \rangle^{k-|\gamma|}.
\end{equation}
It can be checked that \eqref{equation:symbol} is invariant by diffeomorphism, which implies that $S^k(T^*\Sigma)$ is intrinsically defined on $\Sigma$.

We define $\Psi^{-\infty}(\Sigma)$, the set of \emph{smoothing operators}, as the space of linear operators on $\Sigma$ with smooth Schwartz kernel (with respect to some, hence any, volume form on $\Sigma$). Denote by $\mathrm{Op}$ any quantization procedure on $\Sigma$, given in a local coordinate patch $U \subset \RR^n$ by 
$$
\mathrm{Op}(a)f (x) = \dfrac{1}{(2\pi)^n} \int_{\RR^n_\xi} \int_{\RR^n_y} e^{i\xi\cdot(x-y)} a(x,\xi) f(y) \dd y \dd\xi,
$$
where $a \in S^k(T^*U)$ and $f \in C^\infty_{c}(U)$. The set of \emph{pseudodifferential operators} of order $k \in \RR$ is then defined as
$$
\Psi^k(\Sigma) = \left\{ \mathrm{Op}(a) + R ~|~ a \in S^k(T^*\Sigma), R \in \Psi^{-\infty}(\Sigma)\right\}.
$$
It can be checked that $\Psi^k(\Sigma)$ is intrinsically defined and independent on the choice of quantization $\mathrm{Op}$. We will denote by $\Psi_{\mathrm{cl}}^k(\Sigma)$ the set of \textit{classical} pseudo-differential operators, that are operators $A = \Op(a) \in \Psi^k(\Sigma)$ such that $a \in S^k(T^*M)$ admits an asymptotic expansion of the form
$$
a \sim \sum_{\ell \in \mathbb N} a_\ell
$$
where $a_\ell \in S^{k-r_\ell}(T^*M)$ is homogeneous of degree $k-r_\ell$ far from the zero section and $(r_\ell)$ is an increasing sequence such that $r_0 = 0$ and $r_\ell \to \infty$.

There exists a well-defined \emph{principal symbol map}
$$
\sigma : \Psi^k(\Sigma) \to S^k(T^*\Sigma)/S^{k-1}(T^*\Sigma)
$$
such that we have the following exact sequence:
$$
0 \longrightarrow \Psi^{k-1}(\Sigma) \longrightarrow \Psi^k(\Sigma) \longrightarrow S^k(T^*\Sigma)/S^{k-1}(T^*\Sigma) \longrightarrow 0.
$$
 An pseudo-differential operator $A \in \Psi^k(\Sigma)$ is said to be \emph{elliptic} if there exists $C > 0$ such that 
 $$
 \sigma(A) \geqslant C |\xi|^k \quad \text{ whenever} \quad |\xi|\geqslant C.
 $$
  The important property of elliptic operators on $\Sigma$ is that they are invertible modulo smoothing remainders, that is, one can find $B \in \Psi^{-k}(T^*\Sigma)$ and $R \in \Psi^{-\infty}(\Sigma)$ such that
$$
B A = \mathbf{1} +R.
$$
Such an operator $B$ is called a \emph{parametrix} for $A$.

\subsection{Fourier integral operators}\label{ssec:fio}
We refer to \cites{duistermaat1996fourier, hormander2009PDE4} for a general treatment.  If $P$ is a given smooth manifold we denote by $ T^*P \setminus \underline 0$ the cotangent bundle of $P$ with the zero section removed and by $\omega_P \in \Omega^2(T^*P)$ the canonical symplectic form. Let $P, Q$ be smooth manifolds. A \emph{canonical relation} $\mathscr C \subset T^*P\setminus \underline 0  \times T^*Q\setminus \underline 0  $
%such that the twisted relation 
%$$
%    \C' =\left\{ (p,\xi,q-\eta): (p,\xi,q,\eta) \in \C   \right\}\subset \T^*(M_i\times M_j) 
%$$
is a conical submanifold which is Lagrangian with respect to the symplectic form $\pi_P^*\,\omega_P - \pi_Q^*\omega_Q$ on the product $T^*P\setminus \underline 0  \times T^*Q\setminus \underline 0 $. Here $\pi_P$ and $\pi_Q$ are the projections from $T^*P\setminus \underline 0  \times T^*Q\setminus \underline 0 $ towards $T^*P\setminus \underline 0 $ and $T^*Q\setminus \underline 0 $, respectively. For $m \in \RR$, we denote by $\I^m(P, Q, \mathscr C)$ to be  the set of Fourier integral operators of order $m$ whose canonical relation is $\mathscr C$. Concretely, these are operators $\mathscr C_c^\infty(Q) \to \mathcal D'(P)$ whose Schwartz kernel can be written as a locally finite sum of oscillatory integrals of the form
$$
    K(p,q) = \int_{\RR^N} e^{i\varphi(p,q,\theta)} a(p,q,\theta) \dd \theta,
$$
where $\varphi$ is a non-degenerate phase function parametrizing $\mathscr C$ locally, that is $\dd \varphi  \neq 0 $ on the support of $a$ and we have locally
\begin{equation}\label{eq:localrelation}
    \left\{ (p,\partial_p\varphi,q, -\partial_q \varphi) :  \partial_\theta\varphi = 0 \right\}=\mathscr C \subset T^*P\setminus \underline 0  \times T^*Q\setminus \underline 0 .
\end{equation}
Here $a \in S^\mu(P \times Q \times \RR^N_\theta)$ is a classical symbol with
\begin{equation}\label{eq:order}
m = \mu + \frac{N}{2} - \frac{\dim P + \dim Q}{4}.
\end{equation}
In particular, for any $\A \in \I^m(P, Q, \mathscr C)$ the kernel $K_\A \in \mathscr D'(P \times Q)$ of $\A$ (with respect to any volume forms on $P$ and $Q$) satisfies
\begin{equation}\label{eq:wfsetfio}
\WF(K_\A) \subset \mathscr C' = \{(p,\xi, q, -\eta)~:~(p,\xi,q,\eta) \in \mathscr C\}.
\end{equation}
% More precisely,  the canonical relation associated to $\varphi$ is given by
% $$
% \mathscr C_\varphi = \left\{ \left(p, \partial_p \varphi, q, - \partial_q \varphi \right) ~|~ \partial_\theta \varphi = 0 \right\}.
% $$
Pseudo-differential operators on $ P $ correspond to the special case where $\mathscr C$ is the diagonal
$$
    \Delta(T^*P\setminus \underline 0 ) = \left\{ (p,\xi,p,\xi) ~:~ (p,\xi) \in T^*P\setminus \underline 0  \right\}.
$$
Next let $R$ be a third manifold and let $\mathscr C \subset T^*P\setminus \underline 0 \times T^*Q\setminus \underline 0 $ and $\mathscr D \subset T^*Q\setminus \underline 0 \times T^*R\setminus \underline 0  $ be two canonical relations. We say that their composition is \emph{clean} if the intersection
$$
    (\mathscr C \times \mathscr D) \cap (T^*P\setminus \underline 0  \times \Delta(T^*Q\setminus \underline 0 ) \times T^*R\setminus \underline 0 )
$$
is a clean intersection, that is, it is a smooth submanifold and its tangent space is equal to the intersection of the tangent spaces. The integer $e$ is the \emph{excess} of the composition, namely the dimension by which this intersection fails to be transverse.
If the composition of $\mathscr C$ and $\mathscr D$ is clean with excess $e$, then the composition of Fourier integral operators satisfies
\begin{equation}\label{eq:compositionfio}
 \I^{m}(P, Q,\mathscr C)\circ \I^{\ell}(Q, R,\mathscr D)
    \subset
    \I^{m+\ell+\frac e2}(P, R,\mathscr C\circ \mathscr D)
\end{equation}
where the composition $\mathscr C\circ \mathscr D \subset T^*P\setminus \underline 0   \times T^*Q\setminus \underline 0  $ is given by
$$
    \mathscr C\circ \mathscr D = \{(p,\xi, r, \rho)~:~\exists (q,\eta),\ (p,\xi, q, \eta) \in \mathscr C \text{ and } (q, \eta, r, \rho) \in \mathscr D\}.
$$
In particular, if 
$$
\mathscr C = \mathrm{graph}(G)= \{(G(q, \eta), (q, \eta))~:~(q,\eta) \in T^*P\setminus \underline 0 \}
$$
is the graph of a homogeneous symplectomorphism
$$
    G:T^*P\setminus \underline 0  \to T^*P\setminus \underline 0 ,
$$
then the composition is transverse and has zero excess. Hence, if we are given
$$
    \A \in \I^{m_1}(P, P,\mathrm{graph}(G))
    \quad \text{ and } \quad
    \mathrm B \in \I^{m_2}(P,P,\mathrm{graph}(H)),
$$
then
$\A \mathrm B \in \I^{m_1+m_2}(P, P,\mathrm{graph}(G\circ H))$. 

\subsection{Functional spaces}

\label{sssection:functional-spaces}

We now introduce the functional spaces we will be working with. We denote by $\Delta_g \leqslant 0$ the negative Hodge Laplacian acting on functions. For all $s \in \RR$, the operator $(\mathbf{1}-\Delta)^s$ defined using the spectral theorem (applied to the self-adjoint operator $\Delta_g$ on $L^2(\Sigma, \mathrm{vol}_g)$) is an invertible pseudodifferential operator of order $2s$.

For $s \in \RR$, $m \in (1,\infty)$ and $u \in C^\infty(\Sigma)$, we set
\begin{equation}
\label{equation:norm}
\|u\|_{W^{s,m}} = \|(\mathbf{1}-\Delta)^{s/2}u\|_{L^m},
\end{equation}
and define $W^{s,m}(\Sigma)$ to be the completion of $C^\infty(\Sigma)$ with respect to the norm \eqref{equation:norm}. Taking $m=2$, we retrieve the usual Sobolev spaces which we will rather denote by $H^s(\Sigma) = W^{s,2}(\Sigma)$. Note that the spaces $W^{s,m}(\Sigma)$ intrinsically defined, that is, they are independent of the choice of metric $g$, and changing the metric only replaces the norm \eqref{equation:norm} by an equivalent norm. 

The following boundedness result for pseudodifferential operators holds: for all $k \in \RR, A \in \Psi^{k}(\Sigma)$ and $s \in \RR$, $m \in (1,\infty)$,
\begin{equation}
\label{equation:boundedness}
A : W^{s+k, m}(\Sigma) \to W^{s,m}(\Sigma)
\end{equation}
is bounded.
More generally, let $\mathscr C \subset T^*\Sigma \setminus 0 \times T^*\Sigma \setminus 0$ be the graph of a homogeneous symplectomorphism and let $A \in \I^k(\Sigma, \Sigma,\mathscr C)$. Then, for all $s \in \RR$ and $m \in (1,\infty)$,
\begin{equation}
\label{equation:fio-boundedness}
A : W^{s+k+\mu_m,m}(\Sigma) \to W^{s,m}(\Sigma) \quad \text{where} \quad \mu_m = (n-1)\left| \frac1m - \frac12 \right|.
\end{equation}
is bounded, see \cite{seeger1991regularity}*{Corollary 2.4}.

Eventually, given $\Omega\subset \Sigma$ be an open subset with non-empty smooth boundary, we define, for $s\in \RR$ and $m \in [1,\infty)$, the spaces
$$ \dot W^{s, m} (\overline \Omega) = \{u\in W^{s,m}(\Sigma) \mid \supp(u) \subset \overline \Omega\}
\quad \text{and} \quad
\overline W^{s,m}(\Omega) = \{u|_{\Omega} ~|~ u\in W^{s,m}(\Sigma) \}.$$

\section{The Levy generator on Zoll surfaces}
\label{sec:generator}
In this section we analyse the Lévy flight generator, construct its parametrix, and analyse their behaviour.

%%% DECOMPOSITION GENERATOR %%%%
\subsection{Microlocal structure of the generator of Levy process}
From now assume that $\Sigma$ is an orientable Zoll surface. For any $\varrho \in \mathscr C^\infty(\RRR) \cap L^1(\RRR)$ we set 
\begin{eqnarray}\label{eq: def A rho}
\A_\varrho = \pi_*\left(\int \varrho(t) \varphi_{t}^* \dd t\right) \pi^* : \mathscr C^\infty(\Sigma) \to \mathscr C^\infty(\Sigma).
\end{eqnarray}
\begin{proposition}\label{prop:generator}
There exists $\varrho \in \mathscr C^\infty_c(\left]0, 2 \pi\right[)$, a constant $c > 0$ and a classical pseudo-differential operator $\P \in \Psi^{2\alpha}_{\mathrm{cl}}(\Sigma)$ with principal symbol $|\eta|^{2\alpha}_g$ such that 
\begin{equation}\label{eq:decompA}
\A = \P - c \I + \A_\varrho.
\end{equation}
\end{proposition}
\begin{proof}
Let $\delta > 0$ such that $\supp \tau \subset [\delta, 2\pi - \delta]$. Next, let $\chi, \varrho_j \in \mathscr C^\infty(\RRR_+, [0,1])$, $j=1,2$ satisfying $1 = \chi + \varrho_1 + \varrho_2$ and such that we have
$$
\chi(t) = 1 \text{ for } |t|\leqslant\delta / 3 \quad \text{and} \quad \chi(t) = 0 \text{ for } |t| \geqslant \delta / 2
$$
together with the support conditions
$$
\supp \varrho_1 \subset \bigcup_{\ell = 0}^\infty [2\ell \pi + \delta / 3, 2(\ell + 1)\pi - \delta / 3] \quad \text{and} \quad \supp \varrho_2 \subset \bigcup_{\ell =1}^\infty [2\pi \ell - \delta / 3, 2 \pi \ell + \delta / 3].
$$
Set $\varrho_{j, \alpha}(t) = \varrho_j(t) t^{-1-2\alpha}$ for $j = 1,2$. In what follows we will denote for any $\varrho$
$$
\tilde \A_{\varrho} = \A_\varrho - c_\varrho \I \quad \text{where } c_\varrho = \int \varrho(t) \dd t.
$$
Then by \eqref{eq:generator} we may write
$$
\A = \P_1 + \tilde \A_{\varrho_{1, \alpha}} + \tilde \A_{\varrho_{2, \alpha}}
$$
where we set
$$
\P_1 = \lim_{\varepsilon \to 0} \int_\varepsilon^\infty  \chi(t) t^{-1-2\alpha} \pi_* (\varphi_t^* - \I)\pi^* \dd t.
$$
Then $\P_1$ is the principal part of the generator $\A$ and we have $\P_1 \in \Psi^{2\alpha}_{\mathrm{cl}}$, see \cite{chaubet2025levy}. On the other hand, since $\varphi_{t + 2\pi}^* = \varphi_t^*$ for all $t$, one sees that
$$
\begin{aligned}
\tilde \A_{\varrho_{1, \alpha}} &= \int t^{-1-2\alpha}\varrho_{1}(t) \pi_*(\varphi_t^* - \I)\pi^* \dd t \\
 &= \sum_{\ell = 0}^\infty \int_0^{2\pi} \frac{\varrho_1(t + 2\pi \ell)}{(t + 2\pi \ell)^{1 + 2\alpha}} \pi_*(\varphi_t^* - \I)\pi^* \dd t \\
 &= \int \varrho(t) \pi_*(\varphi_t^* - \I)\pi^* \dd t, \quad \text{where } \varrho(t) = \sum_{\ell = 0}^\infty \frac{\varrho_1(t + 2\pi \ell)}{(t + 2\pi \ell)^{1 + 2\alpha}}.
 \end{aligned}
$$
Note that $\varrho \in \mathscr C^\infty_c(\left]0,2\pi\right[)$. Similarly, using again the periodicity of $(\varphi_t)$ one finds that $\tilde \A_{\varrho_{2, \alpha}} = \tilde \A_{\tilde \chi} = \A_{\tilde \chi} - c_{\tilde \chi} \I$, for some $\tilde \chi \in \mathscr C^\infty_c(\left[-\delta / 3, \delta / 3\right])$. Now looking at the kernel of $\A_{\tilde \chi}$ one sees that $\A_{\tilde \chi} \in \Psi^{-1}_\mathrm{cl}(\Sigma)$. Hence one obtains
$$
\A = \P_1 + \A_{\tilde \chi} - (c_\varrho + c_{\tilde \chi}) \I + \A_\varrho.
$$
Setting $\P = \P_1 + \A_{\tilde \chi} \in \Psi^{2\alpha}_{\mathrm{cl}}$ and $c = c_\varrho + c_{\tilde \chi}$ one obtains \eqref{eq:decompA}. Finally, since $\A_{\tilde \chi} \in \Psi^{-1}$ one obtains $\sigma(\P) = \sigma(\P_1) = |\eta|^{2\alpha}_g$, which completes the proof.
\end{proof}

%\subsection{Microlocal structure of the generator near conjugacy times}
Next we shall study the operator $\A_\varrho$ and show that it is a nice Fourier integral operator whose canonical relation is a graph. We consider
$$
G_\pm = \iota_\pm^{-1} \circ \widetilde \Phi \circ \iota_\pm : T^*\Sigma \setminus \underline 0 \longrightarrow T^*\Sigma \setminus \underline 0
$$
where $\iota_\pm$ are the diffeomorphisms from \eqref{eq:iotapm} and $\widetilde \Phi$ is the symplectic lift of $\Phi$, see Lemma \ref{lem:conjmap}. Note that by point (i) of Lemma \ref{lem:conjmap} and \eqref{eq:iotapmrpi} the maps $G_+$ and $G_-$ coincide and we shall denote
$$
G = G_+ = G_- : T^*\Sigma \setminus \underline 0 \longrightarrow T^*\Sigma \setminus \underline 0.
$$

\begin{lemma}\label{lem:G}
    The map $G$ is an exact symplectomorphism of $T^*\Sigma \setminus 0$ such that $G^2 = \mathrm{Id}$. 
\end{lemma}
\begin{proof}
    The map $G$ is obviously smooth. The fact that $G^2 = \mathrm{Id}$ follows immediately from the relation $\Phi^2 = \mathrm{Id}$. Next we show that $G$ is an exact symplectomorphism, that is, it preserves the Liouville one-form $\lambda_{T^*\Sigma}$ on $T^*\Sigma$. But this immediately follows from \eqref{eq:iotapmlambda} and point (iii) of Lemma \ref{lem:conjmap}.
\end{proof}

The main result of this paragraph is the following

\begin{proposition}\label{prop:arho}
Let $\varrho \in \mathscr C^\infty_c(\left]0, 2 \pi\right[)$. Then $\A_\varrho$ is a Fourier integral operator of order $-1$ with canonical relation $\mathscr C \subset T^*\Sigma \setminus 0 \times T^*\Sigma \setminus 0$ given by
$$
    \mathscr C = \mathrm{graph}(G) = \{(G(q, \eta), q, \eta)~:~(q, \eta) \in T^*\Sigma \setminus 0\}.
$$
\end{proposition}

Note that this result is reminiscent of the analysis of the geodesic X-ray transform in \cite{holman2018xray}. Related claims also appear in \cite{tully2024levy} in a more general framework. Since the geometry of the canonical relations is particularly simple in the present setting, we include a direct proof for the sake of completeness. We will start with the following result.

\begin{lemma}\label{lem:Rrho}
    The operator 
    \begin{equation}\label{eq:relationrrho}
        \mathrm R_{\varrho} = \int_{\RR} \varrho(t) \varphi^*_{-t} \dd t : \mathscr C^\infty(M)\to \mathscr C^\infty(M)
        \end{equation}
    is a Fourier integral operator of order $-\frac{1}{2}$ with a canonical relation given by 
    \begin{equation}
        \mathscr C_{\mathrm R_{\varrho }}= \left\{ (\varphi_t(z),\dd \varphi_t(z)^{-\top} \zeta, z, \zeta) : \langle \zeta, X(z) \rangle = 0,~t \in \supp \varrho\right\}.
    \end{equation}
\end{lemma}
\begin{proof}
    Using a partition of unity and flow box coordinates we may assume that $M = \RR^3$ and $X = \partial_1$ so that $\varphi_t(x_1, \bar x) = (x_1 + t, \bar x)$ for $(x_1, \bar x) \in \RR \times \RR^2$. Let $\chi \in \mathscr C^\infty_c(\RR^3)$. Then a straightforward computation yields that the kernel of $\chi \mathrm R_\varrho \chi$ is given by
    $$
    K_{\mathrm R_\varrho}(x,y) = (2\pi)^{-2} \int_{\RR^2} \chi(x) \chi(y) \varrho(y_1 - x_1) e^{i \bar \xi \cdot(\bar x - \bar y)} \dd \bar \xi.
    $$
    Now $(x,y) \mapsto \chi(x) \chi(y) \varrho(y_1 - x_1)$ is a symbol on $\RR^3 \times \RR^3 \times \RR^2$ of order $0$. Hence \eqref{eq:localrelation} and \eqref{eq:order} yield that $\mathrm R_\varrho$ is a Fourier integral operator with order $m = 1 - 3/2 = -1/2$ with local canonical relation
    $$
    \mathscr C = \{(x, \bar \xi, y, \bar \xi)~:~\bar x = \bar y,~ y_1 - x_1 \in \supp \varrho\}
    $$
    where for $z \in \RR^3$ we identify $\bar \xi \in \RR^2$ with $(0, \bar \xi) \in \RR^3 \simeq T_z^*\RR^3$. Since $X = \partial_1$ the space $\ker X(z)$ is naturally identified with $\{0\} \times \RR^2$ and $\dd \varphi_t(z)^{-\top}$ is just the identity map. Moreover the conditions $\bar x = \bar y$ and $y_1 - x_1 \in \supp \varrho$ are equivalent to the fact that $x = \varphi_t(y)$ for some $t \in \supp \varrho$. Hence the above canonical relation coincides with \eqref{eq:relationrrho}.
\end{proof}

\begin{proof}[Proof of Proposition \ref{prop:arho}]
	Because $\pi : M \to \Sigma$ is a submersion the pull back operator $\pi^* : \mathscr C(\Sigma) \to \mathscr C(M)$ is a Fourier integral operator of order $-1/4$ with canonical relation
	\begin{equation}\label{eq:relationpi*}
	\mathscr C_{\pi^*} = \bigl\{(z, \dd \pi(z)^\top \eta, q, \eta)~:~q = \pi(z),~\eta \in T^*_q\Sigma \setminus \underline 0\bigr\},
	\end{equation}
	see for example \cite{holman2018xray}*{Lemma 1}. Since $\pi_*$ is the formal adjoint operator of $\pi^*$ we have $\pi_* \in \I^{-1/4}(\Sigma, M, \mathscr C_{\pi_*})$ with
	$$
	\mathscr C_{\pi_*} = \mathscr C_{\pi^*}^\top = \bigl\{(q, \eta, z, \dd \pi(z)^\top \eta)~:~q = \pi(z),~\eta \in T^*_q\Sigma\setminus \underline 0\bigr\}.
	$$
	Consider the composition $\mathrm R_\varrho \circ \pi^*$. Our aim is to prove that the intersection 
	$$
	(\mathscr C_{\mathrm R_\varrho} \times \mathscr C_{\pi^*}) \cap (T^*M \times \Delta(T^*M) \times T^*\Sigma)
	$$
	is transversal. Denote by $\widehat{\mathscr D}$ the above intersection. Then it suffices to show that the projection map on the middle factors
	$$
	\Psi : \mathscr C_{\mathrm R_\varrho} \times \mathscr C_{\pi^*} \to T^*M \times T^*M
	$$
	is transversal to $\Delta(T^*M)$, that is for any $\gamma \in \widehat{\mathscr D}$ we have 
	\begin{equation}\label{eq:transversalityPsi}
	\operatorname{ran}\dd \Psi(\gamma) + T_{\Psi(\gamma)} \Delta(T^*M) = T_{\Psi(\gamma)}(T^*M \times T^*M).
	\end{equation}
	For any vector field $Z$ on $M$, we shall denote 
	$$
	h_Z : T^*M \to \RR, \quad (z, \zeta) \mapsto \langle \zeta, Z(z)\rangle.
	$$
	Then by Lemma \ref{lem:Rrho} and \eqref{eq:relationpi*} we obtain that the right projection $\mathscr C_{\mathrm R_\varrho} \to T^*M$ is a submersion on $\{h_X = 0\} \subset T^*M$ while the left projection $\mathscr C_{\pi^*} \to T^*M$ is a submersion on $\{h_V = 0\} \subset T^*M$. This implies that for any $\gamma \in \widehat{\mathscr D}$ we have, writing $\Psi(\gamma) = (z, \zeta, z, \zeta)$,
	$$
	\operatorname{ran}\dd \Psi(\gamma) = T_{(z,\zeta)}\{h_X = 0\} \times T_{(z,\zeta)}\{h_V = 0\} = \ker \dd h_X(z, \zeta) \times \ker \dd h_V(z,\zeta).
	$$
	Now the linear forms $\dd h_X(z, \zeta), \dd h_V(z,\zeta) \in T^*_{(z,\zeta)}T^*M$ are linearly independent hence the dimension of the intersection of their kernels is equal to $\dim T_{(z,\zeta)}T^*M - 2$. This yields
	$$
	\begin{aligned}
	\dim\Bigl(\operatorname{ran}\dd \Psi(\gamma) \cap T_{\Psi(\gamma)} \Delta(T^*M)\Bigr) &= \dim\bigl(\ker \dd h_X(z, \zeta) \cap \ker \dd h_V(z,\zeta)\bigr) = N -2 
	\end{aligned}
	$$
	where $N = \dim T^*M$. Now $\dim\operatorname{ran}\dd \Psi(\gamma) = 2(N- 1)$ and $\dim T_{\Psi(\gamma)}\Delta(T^*M) = N$ hence 
	$$
	\dim\Bigl(\operatorname{ran}\dd \Psi(\gamma) + T_{\Psi(\gamma)} \Delta(T^*M)\Bigr) = 2(N - 1) + N - (N - 2) = 2N
	$$
	and \eqref{eq:transversalityPsi} holds. Hence by \cite{duistermaat1996fourier}*{Theorem 2.4.1} we obtain that $\mathrm R_\varrho \pi^* \in \I^{-3/4}(M, \Sigma, \mathscr D)$ where $\mathscr D = \mathscr C_{\mathrm R_\varrho} \circ \mathscr C_{\pi^*}.$ Note that a point
	$
	\gamma = (\varphi_t(z), \dd \varphi_t(z)^{-\top} \zeta, z, \zeta, z, \zeta, q, \eta) \in \widehat{\mathscr D}
	$
	satisfies 
	$$
	q = \pi(z), \quad \zeta = \dd \pi(z)^\top \eta \quad \text{and} \quad \langle \zeta, X(z) \rangle = 0.
	$$
	This implies $ 0 = \langle \dd \pi(z)^\top \eta, X(z)\rangle = \langle \eta, \dd \pi(z) X(z) \rangle = \langle \eta, v \rangle$ where $z = (q,v)$. Thus $v = \pm v_\eta$. In particular we obtain 
	$$
	\mathscr D = \mathscr D_+ \sqcup \mathscr D_- \quad \text{where} \quad \mathscr D_\pm = \Bigl\{\Bigl(\bigl(\widetilde \varphi_t \circ\iota_\pm\bigr)(q,\eta),\,q,\eta)\Bigr)~:~t \in \supp \varrho,~(q,\eta) \in T^*\Sigma \setminus \underline 0\Bigr\},
	$$
	where $\widetilde \varphi_t(z, \zeta) = (\varphi_t(z), \dd \varphi_{t}(z)^{-\top} \zeta)$ is the symplectic lift of $\varphi_t$. Note that the manifolds $\mathscr D_+$ and $\mathscr D_-$ are disjoint since the $(\varphi_t)$-orbit of $z_\eta^+$ is always disjoint from that of $z_\eta^-$. 
	
	Next we consider the composition $\pi_* \circ (R_\varrho \pi^*)$. Since
$
\mathscr D=\mathscr D_+\sqcup \mathscr D_-,
$
it is enough to compose $\mathscr C_{\pi_*}$ with each branch
$\mathscr D_\pm$. We claim that these compositions are transverse. To see this, let
$$
\Psi_\pm:\mathscr C_{\pi_*}\times \mathscr D_\pm
\longrightarrow T^*M\times T^*M
$$
be the projection onto the middle factors. As above, the right projection
$\mathscr C_{\pi_*} \to T^*M$ is a submersion onto $\{h_V=0\}$. Next let $(z,\zeta) \in T^*M \setminus \underline 0$ which lies in the intersection between $\{h_V=0\}$ and the image of the left projection $\mathscr D_\pm \to T^*M$, which, by what precedes, is given by the image of the map
$$
\widetilde \Psi: \supp \varrho \times T^*\Sigma \setminus 0 \to T^*M, \quad  (t, q,\eta)\mapsto \widetilde\varphi_t\circ\iota_\pm(q,\eta).
$$
Since $h_V(z, \zeta) = 0$ by \eqref{eq:dphithstar} we must have $(z,\zeta) = \Psi(t, q,\eta)$ for some $(q,\eta) \in T^*\Sigma \setminus \underline 0$ and $t = \tau(z_\zeta^\pm)$.
Moreover the image of $\widetilde \Psi$ is contained in $\{h_X = 0\}$, hence we have the inclusion $\mathrm{ran}(\dd \Psi(t,q,\eta))\subset\ker \dd h_X(z,\zeta)$. This also implies $(z,\zeta) \in \mathbf H^\star$. Next recall that $\iota_\pm(q, \eta) \in \ker X(z_\eta^\pm)$ which implies 
$$
\widetilde \Psi(\tau(z_\eta^\pm), \iota_\pm(q, \eta)) = \widetilde \Phi(\iota_\pm(q,\eta))
$$
However $\widetilde \Phi \circ \iota_\pm$ is a diffeomorphism $T^*\Sigma \setminus \underline 0 \to \mathbf H^\star \setminus \underline 0$ and this implies
$$
\mathrm{ran}(\dd \Psi(t,q,\eta)) \supset T_{(z,\zeta)}\mathbf H^\star.
$$
Finally, as $(z, \zeta) \in \mathbf H^\star$ we can write $\zeta = c \beta(z)$ for some $c \neq 0$. From the proof of Lemma \ref{lem:conjtime}, we have 
$$
\dd \varphi_s(z)^{-\top}\zeta = a(s) \beta(\varphi_s(z)) + b(s) \psi(\varphi_s(z))
$$
for some functions $a,b : \RR \to \RR$ such that $a(0) = c$ and $b'(0) \neq 0$. Hence one obtains $\partial_s|_{s=0} h_V(\widetilde \varphi_s(z,\zeta)) = \partial_s|_{s=0}b(s) \neq 0$. This implies that $\mathrm{ran}(\dd \Psi(t,q,\eta))$ is not contained in $T_{(z,\zeta)}\mathbf H^\star \subset \ker \dd h_V(z,\zeta)$. Hence $\dd \Psi(t,q,\eta)$ is injective and its range is exactly $\ker \dd h_X(z,\zeta)$. Therefore we may argue as above to obtain that the middle projections $\Psi_\pm$ are transverse to $\Delta(T^*M)$. By
\cite{duistermaat1996fourier}*{Theorem 2.4.1}, we obtain 
$$
\A_\varrho \in \I^{-1}(\Sigma, \Sigma, \mathscr C_{\A_\varrho})
$$
where $\mathscr C_{\A_\varrho} = (\mathscr C_{\pi_*} \circ \mathscr D_+) \cup (\mathscr C_{\pi_*} \circ \mathscr D_-)$. It remains to identify the composed canonical relation. A point 
$(p,\xi, q, \eta)$ lies in $\mathscr C_{\pi_*}\circ\mathscr D_\pm$ iff there is $t \in \supp \varrho$ such that
$$
(z,\zeta)=\iota_\pm(q,\eta) 
\quad \text{and} \quad
\iota_\pm(p,\xi) = \widetilde \varphi_t(z, \zeta).
$$
The second condition implies $\widetilde \varphi_t(z, \zeta) \in \mathbf H^\star(\varphi_t(z))$. Since $\zeta\in\mathbf H^\star(z)$, this condition means precisely that
$t$ is the first conjugacy time along the geodesic issued from $z$. Thus
$
t=\tau(z)
$
and $\widetilde \varphi_t(z, \zeta) = \widetilde \Phi(z, \zeta)$. This in turn implies 
$$
(p, \xi) = (\iota_\pm^{-1} \circ \widetilde \Phi \circ \iota_\pm)(q, \eta) = G_\pm(q, \eta) = G(q, \eta).
$$
Hence both compositions $\mathscr C_{\pi_*} \circ \mathscr D_+$ coincide with the graph $\mathscr C$ of $G$. This completes the proof.
\end{proof}
	From \S\ref{ssec:fio} and Lemma \ref{lem:G} we see that Proposition \ref{prop:arho} implies the following
	
\begin{corollary}\label{cor:pseudofio}
For any even $\ell \in \mathbb N$ and pseudo-differential operators $(\mathrm Q_j)$ of order $m_j$ for $j=1,\dots, \ell+1$, we have
$$
    \mathrm Q_1 \A_\varrho \Q_2 \A_\varrho \cdots \Q_{\ell} \A_\varrho \Q_{\ell + 1} \in \Psi^{m_1 + \cdots + m_{\ell + 1} - \ell}(\Sigma).
$$
If $\ell$ is odd then this operator lies in $\mathrm I^{m_1 + \cdots + m_{\ell + 1} - \ell}(\Sigma, \Sigma, \mathscr C)$.
\end{corollary}

\subsection{Microlocal structure of the full parametrix}

In this paragraph we construct and give a microlocal description of the full parametrix of $\A$.

\begin{proposition}
For any $m \in\left[1, \infty\right[$ the unbounded operator $-\A$ acting on $L^m$ with domain $W^{2\alpha, m}$ is a Fredholm operator of index $0$. Moreover it has nonnegative discrete spectrum and its kernel consists only of constant functions.
\end{proposition}
\begin{proof}
The boundedness of $\A : W^{2\alpha, m} \to L^m$ follows immediatly from Propositions \ref{prop:generator} and \ref{prop:arho} since $\A_\varrho$ maps $W^{s, m} \to W^{s+1 + |1/2 - 1/m|, m}$ for each $s \in \RRR$. The latter bounds imply that $\A_\varrho : W^{2\alpha, m} \to L^m$ is compact. Since $\P$ is elliptic we obtain that $\A$ is a Fredholm operator of index $0$ and has discrete spectrum. Let $u \in W^{2\alpha, m}$ such that $\A u = 0$. Then $\P u = - \A_\varrho u$ hence $u \in W^{2\alpha + 1 + |1/2 - 1/m|, m}$. Iterating this process one sees that $u \in \mathscr C^\infty(\Sigma)$. However by \cite{chaubet2025levy}*{Proposition 1.3} we must have that $u$ is constant. Finally, if $\lambda$ is a nonzero eigenvalue of $-\A$ with eigenfunction $u$, the same ellipticity argument shows that $u$ is smooth. Applying again \cite{chaubet2025levy}*{Proposition 1.3} shows that $\lambda$ must be positive. This completes the proof.
\end{proof}

\begin{corollary}\label{cor:parametrix}
There exists a bounded operator $\A^+ : \mathscr C^\infty(\Sigma) \to \mathscr D'(\Sigma)$ such that 
$$
\A^+ \A = \A \A^+ = \mathrm{Id} - \Pi
$$
where $\Pi$ is the $L^2$ orthogonal projection on the space of constant functions. Moreover the operator $\A^+$ is an essentially self-adjoint and bounded map from $W^{s,m} \to W^{s + 2\alpha, m}$ for any $s \in \RRR$ and $m \in \left]1, \infty\right[$.
\end{corollary}
\begin{proof}
Let $\Pi_{s,m} = \Pi|_{W^{s,m}}$. Since $\A$ is Fredholm of index $0$ and since $\ker(\A) = \CC \cdot \mathbf 1$, we have
$$
W^{s,m} = \mathrm{Im}(\A) \oplus \CC \cdot \mathbf 1.
$$
Moreover, since $\A$ is self-adjoint and $\A(\mathbf 1) = 0$, we have
$
\mathrm{Im}(\A) \subset \ker \Pi_{s,m}.
$
Since both spaces have codimension $1$, it follows that
$$
\mathrm{Im}(\A) = \ker \Pi_{s,m}.
$$
Next, this implies that the map
$$
\A|_{\ker \Pi_{s + 2\alpha, m}} : \ker \Pi_{s + 2\alpha, m} \to \ker \Pi_{s,m}
$$
is bijective. Since $\mathrm{Im}(\A)$ is closed (because $\A$ is Fredholm), the open mapping theorem implies that its inverse is bounded. We define
$$
\A^+ : W^{s,m} \to W^{s + 2\alpha,m}
$$
by setting $\A^+|_{\CC \cdot \mathbf 1} = 0$ and declaring that $\A^+|_{\ker \Pi_{s,m}}$ is the inverse of
$
\A|_{\ker \Pi_{s + 2\alpha, m}}.
$
Then
$
\A \A^+ = \mathrm{Id} - \Pi_{s,m}
$
on $W^{s,m}$ and
$
\A^+ \A = \mathrm{Id} - \Pi_{s + 2\alpha,m}
$
on $W^{s + 2\alpha,m}$.
\end{proof}

\begin{proposition}\label{prop:fullparametrix}
One can find operators
$$
\mathrm S \in \Psi^{-2\alpha}_{\mathrm{cl}} + \Psi^{-4\alpha}_{\mathrm{cl}} + \Psi^{-2 - 6 \alpha} \quad \text{and} \quad \mathrm T \in \mathrm I^{-1-4\alpha}(\Sigma, \Sigma, \mathscr C)
$$
such that $\mathrm S$ has principal symbol $|\eta|_g^{-2\alpha} \mod S^{-4\alpha}_{\mathrm{cl}}$ and
$$
\A^+ = \mathrm S + \mathrm T.
$$
Moreover $\mathrm T = - \mathrm S \A_\varrho \mathrm S  \mod \I^{-3-8\alpha}(\Sigma, \Sigma, \mathscr C)$.

\begin{proof}
Let $\Q \in \Psi^{-2\alpha}_{\mathrm{cl}} + \Psi^{-4\alpha}_{\mathrm{cl}}$ be left and right parametrices for $\P - c \I$ so that 
$$
\Q(\P - c\I) = \mathrm I + \mathrm K
$$
for some smoothing operator $\K$. Set $\Q_1 = \Q - \Q \A_\varrho \Q$. Then
$$
\Q_1 \A = (\Q - \Q \A_\varrho \Q)(\P -c\I + \A_\varrho) = \I - \Q \A_\varrho \Q \A_\varrho + \K_1
$$
where $\K_1$ is smoothing. Now Corollary \ref{cor:pseudofio} yields $\Q \A_\varrho \Q \A_\varrho \in \Psi^{-2-4\alpha}$ hence we get $\Q_1 A = \I + \mathrm R_1$ with $\mathrm R_1 \in \Psi^{-2 - 4\alpha}$. Let $\mathrm R_2 \in \Psi^{-2-4\alpha}$ such that $(\I + \mathrm R_2)(\I + \mathrm R_1) = \I + \K_2$ where $\K_2$ is smoothing. Then 
$
(\I + \mathrm R_2) \Q_1 \A = \I + \K_2
$
and multiplying on the right by $\A^+$ yields
$$
\A^+ = (\I + \mathrm R_2) \Q_1 (\I - \Pi) + \K_3
$$
where $\K_3$ is smoothing. Now setting 
$$
\begin{aligned}
\mathrm S &= \Q + \mathrm R_2 \Q - (\I - \mathrm R_2) \Q_1 \Pi + \K_3 \in \Psi^{-2\alpha}_{\mathrm{cl}} + \Psi^{-4\alpha}_{\mathrm{cl}} + \Psi^{-2 - 6\alpha} \\
\text{and} \quad \mathrm T &= - (\I + \mathrm R_2) \Q A_\varrho \Q (\I - \Pi) \in \I^{-1 - 4 \alpha}(M,M, \mathscr C),
\end{aligned}
$$
we get the sought result, since $\mathrm R_2 \Q \A_\varrho \Q \in \I^{-3 - 8\alpha}(\Sigma, \Sigma, \mathscr C)$ by Corollary \ref{cor:pseudofio}.
\end{proof}
\end{proposition}

\section{Expected time for random searches}
\label{sec:expected}

In this section we fix $p_\star \in \Sigma$. Let $(X_t)_{t\geqslant 0}$ be the cadlag martingale on $\Sigma$ whose generator is $\A$, see \cite{applebaum2000}. For $p \in \Sigma$ and $\varepsilon > 0$ we let
\begin{eqnarray}\label{eq: ustar def}
u_{\star, \varepsilon}(p) = \mathbf E\Bigl[\inf\Bigl\{t \geqslant 0~:~X_t \in B(p_\star, \varepsilon)\Bigr\}\, \Bigl| \,X_0 = p\Bigr]
\end{eqnarray}
be the expected time for the process to reach the ball $B(p_\star, \varepsilon)$, starting from $p$. 

The goal of this section is to prove Theorem \ref{thm:main}. It is organized as follows. In \S\ref{subsec:expectedissolution} we adapt the results from \cite{chaubet2025levy} and prove that $u_{\star, \varepsilon}$ is the solution of some non-local equation involving $\A$. In \S\ref{subsec:integral} we take advantage of the preceding section and express the expected stopping time as an integral. In \S\ref{subsec:finitedegree} and \S\ref{subsec:infinitedegree} we use the normal form obtained in \S\ref{sec:zoll} and estimate the expected stopping time in the finite and infinite degree cases, respectively.

\subsection{The expected time as the solution of a non local equation}\label{subsec:expectedissolution}
 In what follows we set $\Omega_\varepsilon = \complement B(p_\star, \varepsilon)$.  For $\varepsilon>0$ smaller than the injectivity radius of $\Sigma$ we take geodesic coordinates around $p_\star$, as follows. Fix $(v_1, v_2)$ an orthonormal basis of $T_{p_\star}\Sigma$. Let $\mathbb B^2$ be the open Euclidean unit ball in $\RR^2$. Then we consider the diffeomorphism
\begin{eqnarray}\label{eq: def of rescaled geodesic coord}
\Psi_{\varepsilon} : \mathbb B^2 \to B(p_\star, \varepsilon), \quad x = (x_1, x_2) \mapsto \exp_{p_\star}(\varepsilon(x_1 v_1 + x_2 v_2)).
\end{eqnarray}
We also introduce the following notation
$$
\mathrm{err}(\varepsilon, \alpha) = 
\begin{cases}
\varepsilon^{2\alpha}, \quad &\text{if} \ \alpha < 1/2, \\
  \varepsilon|\log\varepsilon|, \quad &\text{if} \ \alpha = 1/2,  \\
\varepsilon^{2(1-\alpha)}, \quad &\text{if} \ \alpha > 1/2. 
\end{cases}
$$
Then reproducing the arguments of \cite{chaubet2025levy}*{\S5} we get the following

\begin{proposition}\label{prop:expectedissolution}
 For any $m\in\left]1, 1/\alpha\right[$, $u_{\star,\epsilon}$ defined in \eqref{eq: ustar def} is the unique distribution in $\mathscr D'(\Sigma)$ which satisfies 
\begin{equation}\label{eq:expectedissolution}
u_{\star, \varepsilon} \in \dot W^{2\alpha,m}\bigl(\overline{\Omega_\varepsilon}\bigr)\qquad \text{and} \qquad \A u_{\star, \varepsilon} = -1 \quad \text{ on } \quad \Omega_\varepsilon.
\end{equation}
Moreover, setting $F_\varepsilon = \A u_{\star, \varepsilon} + \mathbf{1}_{\Omega_\varepsilon} \in \dot L^m(\overline{B(p_\star, \varepsilon)})$ we have $F_\varepsilon|_{B(p_\star, \varepsilon)} \in \mathscr C^\infty(B(p_\star, \varepsilon))$ and
\begin{equation}\label{eq:expansionfeps}
F_\varepsilon(\Psi_{\varepsilon}(x)) = \frac{\pi |\Sigma|\varepsilon^{-2}}{(1-\alpha)} \frac{1}{(1 - |x|^2)^\alpha} + \mathcal O_{\dot L^m(\mathbb B^2)}(\varepsilon^{-2}\mathrm{err}(\varepsilon, \alpha)).
\end{equation}
Also, there is $c_\alpha > 0$ such that 
\begin{equation}\label{eq:expansionueps}
\overline{u}_{\star, \varepsilon} = c_\alpha |\Sigma| \varepsilon^{2(\alpha - 1)}(1 + \mathcal O(\mathrm{err}(\varepsilon, \alpha))) \quad \text{where} \quad \overline{u}_{\star, \varepsilon} = \frac{1}{|\Sigma|} \int_\Sigma u_{\star, \varepsilon}\dd \vol_g.
\end{equation}
Finally, if $\alpha < 1/2$, the expected time map $u_{\star, \varepsilon}$ is continuous in $\Omega_\varepsilon$.
\end{proposition}

\begin{proof}
This is exactly the content of Propositions 5.1 and 5.2 in \cite{chaubet2025levy} which include only the case where $\Sigma$ is a sphere. Most of the arguments are still valid for Zoll surfaces but for convenience of the reader we explain the global strategy and where changes are needed. 

\vspace{0.1cm}
\noindent \textbf{First step: construction of $F_\varepsilon$.} Notice that if \eqref{eq:expectedissolution} holds then $F_\varepsilon \in \dot L^m(\overline{B(p_\star, \varepsilon)})$ satisfies
\begin{equation}\label{eq:integralequation}
\A^+(F_\varepsilon - \mathbf 1_{\Omega_\varepsilon}) = - c_\varepsilon \quad \text{on } B(p_\star, \varepsilon) \qquad \text{and} \qquad \int F_\varepsilon \dd\vol_g = |\Omega_\varepsilon|
\end{equation}
where $c_\varepsilon = \bar u_{\star, \varepsilon} \in \RR$. Then \cite{chaubet2025levy}*{Proposition 5.1} says that if $\Sigma$ is the round sphere then there exists a unique pair $(F_\varepsilon, c_\varepsilon)$ satisfying \eqref{eq:integralequation}, which must satisfy \eqref{eq:expansionfeps} and \eqref{eq:expansionueps} where in the latter $\bar u_{\star, \varepsilon}$ should replace by $c_\varepsilon$. The only ingredient needed in the proof is the fact that one can write
$
\A^+ = \Q + \mathrm R
$
where $\Q \in \Psi^{-2\alpha}_{\mathrm{cl}}$ has principal symbol $c |\eta|^{-2\alpha}_g$ and the operator $\mathrm R$ satisfies, if $\beta = \min(1 + 2\alpha, 4\alpha)$,
\begin{equation}\label{eq:reps}
\|\mathrm R_\varepsilon\|_{\dot L^m(\overline{\mathbb B}) \to \overline W^{\beta}(\mathbb B)} \leqslant C \varepsilon^{2\alpha} \mathrm{err}(\varepsilon, \alpha), \quad \text{where} \quad \mathrm R_\varepsilon = (\Psi_\varepsilon^{-1})^* \mathrm R \Psi_\varepsilon^*,
\end{equation}
for some constant $C > 0$, see \cite{chaubet2025levy}*{Lemma 5.4}. In that reference the estimate \eqref{eq:reps} comes from the fact that if $\chi \in \mathscr C^\infty(\Sigma)$ is supported close enough to $p_\star$ then $\chi \mathrm R \chi$ is a sum of classical pseudo-differential operators of order at most $-\beta$. In our case, the operator $\chi \mathrm R \chi$ can also written as such but an extra term in $\Psi^{-2-6\alpha}$ has to be added, according to Proposition \ref{prop:fullparametrix}. However any operator in $\Psi^{-2-6\alpha}$ satisfy the estimate \eqref{eq:reps}. Hence \cite{chaubet2025levy}*{Proposition 5.1} is valid in our context and there exists a unique pair $(F_\varepsilon, c_\varepsilon)$ satisfying \eqref{eq:integralequation} and \eqref{eq:expansionueps}.

\vspace{0.1cm}
\noindent \textbf{Second step: uniqueness of $u_{\star, \varepsilon}$.} The second step consists in proving that \eqref{eq:expectedissolution} has a unique solution. This is the content of \cite{chaubet2025levy}*{Proposition 5.2}. The proof is carried out in \S5.3 of that reference. All those proofs given there are valid in our context, except for Lemma 5.8, which says that $e^{t\A}$ has smooth kernel for every $t > 0$. However since $\A$ is formally self adjoint, this follows from basic functional analysis and the fact that $\A$ is the sum of an elliptic pseudo-differential operator of order $2\alpha$ and the operator $\A_\varrho$ which maps continuously $W^{s,2} \to W^{s-1,2}$ for every $s \in \RR$, see \S\ref{sssection:functional-spaces}.

Combining those two steps, one sees that if we set $\tilde u_{\star, \varepsilon} = c_\varepsilon + \A^+ (\tilde F_\varepsilon - \mathbf 1_{\Omega_\varepsilon})$ where $(\tilde F_\varepsilon, \tilde c_\varepsilon)$ is the unique solution of \eqref{eq:integralequation}, then $\tilde u_{\star, \varepsilon} \in \dot W^{2\alpha, m}(\overline{\Omega_\varepsilon})$ is solution of \eqref{eq:expectedissolution} hence $\tilde u_{\star, \varepsilon} = u_{\star, \varepsilon}$. It remains to see that $u_{\star, \varepsilon}$ is continuous. For this, notice that Proposition \ref{prop:fullparametrix} and \S\ref{sssection:functional-spaces} yields that far from $B(p_\star, \varepsilon)$ the regularity of $\A^+ (F_\varepsilon- \mathbf 1_{\Omega_\varepsilon})$ is $W^{1 + 4\alpha - |1/2 - 1/m|, m}$ for each $1 < m < 1/\alpha$. The latter space injects in the space $\mathscr C(\Sigma)$ of continuous functions whenever $(1 + 4\alpha - |1/2 - 1/m|)m > 2$. This condition is compatible with $m < 1 / \alpha$ provided $\alpha < 1/2$. This completes the proof.
\end{proof}

\subsection{Reduction to an integral estimate}\label{subsec:integral}
In this subsection we will take advantage of Proposition \ref{prop:expectedissolution} to prove the reduce the estimation of $u_{\star, \varepsilon} - \bar u_{\star, \varepsilon}$ to an integral estimate. 

We first introduce some notations.  Since $\mathrm S \in \Psi^{-2\alpha}_{\mathrm{cl}} + \Psi^{-4\alpha}_{\mathrm{cl}} + \Psi^{-2-6\alpha}$ has principal symbol $c_\alpha |\eta|_g^{-2\alpha}$, there exists $C_\alpha>0$ such that for all $\delta>0$, there exists $ c > 0$ such that 
\begin{equation}\label{eq:noyauS}
C_\alpha(1-\delta)\,\dd(q,r)^{2(2\alpha - 1)} \leqslant K_{\mathrm S^2}(q,r) \leqslant C_\alpha(1+\delta)\,\dd(q,r)^{2(2\alpha - 1)} 
\end{equation}
for any $q \neq r \in \Sigma$ such that $\dd(q,r) < 2c < r_\mathrm{inj}(\Sigma)$, see \cite[Proposition 2.2, \S7.2, p. 6]{taylor2013partial}. We then introduce the function $G_\alpha \in \mathscr C(\RR^2)$ defined by
\begin{equation}\label{eq:Galpha}
G_\alpha(x) = \frac{\pi\, C_\alpha |\Sigma|}{1-\alpha} \int_{\mathbb B} |x-y|^{4\alpha - 2}(1-|y|^2)^{-\alpha} \dd y.
\end{equation}
The aim of this paragraph is to prove the following
\begin{proposition}\label{prop:reduceintegral}
Suppose $\alpha < 1/2$. Then for any $\delta > 0$ there are $c,C > 0$ and $\chi \in \mathscr C^\infty_c(B(p_\star, c))$ with $\chi(p_\star) = 1$ such that the following holds. For any $\varepsilon > 0$ small,
$$
\begin{aligned}
&\varepsilon^{4\alpha - 2} (1-\delta) \A_\varrho[\chi (G_\alpha \circ \Psi_\varepsilon^{-1})](p) - C \\
&\hspace{2cm}\leqslant  \overline{u}_{\star, \varepsilon} - u_{\star, \varepsilon}(p) \\
&\hspace{4cm} \leqslant \varepsilon^{4\alpha - 2} (1+\delta) \A_\varrho[\chi (G_\alpha \circ \Psi_\varepsilon^{-1})](p) + C.
\end{aligned}
$$
\end{proposition}
Note that by definition of $A_\varrho$ one has
\begin{equation}\label{eq:avarrhochi}
\A_\varrho [\chi (G_\alpha \circ \Psi_\varepsilon^{-1})](p) = \int_{T_p\Sigma}|v|^{-1} \varrho(|v|) \chi(\exp_p(v)) G_\alpha(\Psi_{\varepsilon}^{-1}(\exp_p(v))) \dd v.
\end{equation}

\begin{proof}[Proof of Proposition \ref{prop:reduceintegral}]
By Proposition \ref{prop:expectedissolution} and Corollary \ref{cor:parametrix} we have
$$
u_{\star, \varepsilon} - \bar u_{\star, \varepsilon} = \A_+ (\A u_{\star, \varepsilon} + 1) =  \A_+ (F_\varepsilon + \mathbf 1_{B(p_\star, \varepsilon)}).
$$
In the rest of this section, we assume $\alpha < 1/2$ and fix $p \in \Sigma$ distinct from $p_\star$. Then $u_{\star, \varepsilon}$ is continuous near $p$ as soon as $\dd(p,p_\star) > \varepsilon$ and we can write
\begin{eqnarray}\label{eq:ustar minus baru}
u_{\star, \varepsilon}(p) - \bar u_{\star, \varepsilon} = \A_+  (F_\varepsilon + \mathbf 1_{B(p_\star, \varepsilon)})(p).
\end{eqnarray}
By Proposition \ref{prop:fullparametrix} and \eqref{eq:wfsetfio}, the singular support of the kernel $K_{\A_+} \in \mathscr D'(\Sigma \times \Sigma)$ of $\A^+$ is included in the set
$$
\left\{(p,q) \in \Sigma \times \Sigma~:~p \in \mathrm{Conj}(q)\right\} \cup \Delta(\Sigma).
$$
In particular for $p \notin \mathrm{Conj}(p_\star)$ we have that $K_{\A_+}$ is smooth near $(p, p_\star)$. Hence by \eqref{eq:expansionfeps} one obtains
$$
(\A_+ F_\varepsilon)(p) = \int_{B(p_\star, \varepsilon)} K_{\A^+}(p, q) F_\varepsilon(q) \dd \vol_g(q) \to |\Sigma| K_{\A^+}(p, p_\star).
$$
as $\varepsilon \to 0$. Next, assume that $p \in \mathrm{Conj}(p_\star)$. By Proposition \ref{prop:fullparametrix} one can write the parametrix as 
$$
\A^+ = \mathrm S - \mathrm S \A_\varrho \mathrm S + \mathrm R
$$
for some $\mathrm R \in \I^{-3-8\alpha}(\Sigma, \Sigma, \mathscr C)$. Note that this means $\mathrm R : H^{-1-\delta}(\Sigma) \to H^{1+\delta}(\Sigma)$ for some $\delta>0$ and therefore has continuous Schwartz kernel $K_{\mathrm R} \in \mathscr C (\Sigma\times \Sigma)$ by Sobolev inclusions. This allows us to write
\begin{eqnarray}
\label{eq: RF bounded}
| {\mathrm R} F_\varepsilon(p)| \leq \|K_{\mathrm R}(p,q)\|_{L^\infty_q} \| F_\varepsilon\|_{L^1} \leq C
\end{eqnarray}
due to \eqref{eq:expansionfeps}. So we obtain 
$$
\mathrm R(F_\varepsilon + \mathbf 1_{B(p_\star, \varepsilon)})(p) = \mathcal O(1)
$$
as $\varepsilon \to 0.$ Next, notice that $\mathrm S$ is pseudo differential hence it has smooth kernel near $(p, p_\star)$. Moreover $\mathrm S A_\varrho \mathrm S\in  I^{-1-4\alpha}(\Sigma, \Sigma, \mathscr C)$ hence bounded $L^2 \to \mathscr C(\Sigma)$ by Sobolev inclusions. Those considerations tell us that for $p\in {\rm Conj}(p_\star)$
$$
\mathrm S(F_\varepsilon + \mathbf 1_{B(p_\star, \varepsilon)})(p) = \mathcal O(1)  \quad \text{and} \quad -\mathrm S A_\varrho \mathrm S\mathbf 1_{B(p_\star, \varepsilon)}(p) = \mathcal O(1)
$$
as $\varepsilon \to 0$. Hence we are left with analyzing the term
$
 -\mathrm S \A_\varrho \mathrm S F_\varepsilon(p).
$
The canonical relation of $\A_\varrho$ being a graph, one has
$[\mathrm S,\A_\varrho] \in \I^{-2-2\alpha}(\Sigma, \Sigma, \mathscr C)$ by symbolic calculus. Hence, writing 
$
 \mathrm S \A_\varrho \mathrm S = \A_\varrho \mathrm S^2 + [\mathrm S, \A_\varrho] \mathrm S,
$
one obtains
$$
\mathrm S \A_\varrho \mathrm S =  \A_\varrho \mathrm S^2 + \Q \quad \text{ for some } \Q \in \I^{-2-4\alpha}(\Sigma, \Sigma, \mathscr C).
$$
For the same reason as \eqref{eq: RF bounded}, $\Q F_\varepsilon(p) = \mathcal O(1)$. Hence we obtained 
\begin{equation}\label{eq:sas}
u_{\star, \varepsilon}(p) -  \overline{u}_{\star, \varepsilon} = -\A_\varrho \mathrm S^2 F_\varepsilon(p) + \mathcal O(1) \quad \text{as } \varepsilon \to 0. 
\end{equation}
Note that since $\alpha <  1/2$ we can take $m=2$. 
From now on we fix $\delta > 0$. Let $c > 0$ small enough so that the estimate \eqref{eq:noyauS} holds. Because $K_{\mathrm S^2}$ is smooth far from the diagonal, we have
\begin{equation}\label{eq:s^2f}
\mathrm S^2 F_\varepsilon(q) = \mathcal O(1)
\end{equation}
uniformly outside $B(p_\star, c)$. Now note that $|\Psi^{-1}_\varepsilon(q)| = \dd(q, p_\star) / \varepsilon$ for any $q \in B(p_\star, c)$. Note also that $(\Psi_\varepsilon^{-1})^*_{\dot L^2(\mathbb B) \to L^2(\Sigma)} \leqslant C \varepsilon$. Hence using \eqref{eq:expansionfeps} one gets
\begin{eqnarray}\label{eq: Feps principal term}
F_\varepsilon(q)  = \varepsilon^{-2}H_\alpha(\dd(q, p_\star) / \varepsilon) + D_\varepsilon(q)
\end{eqnarray}
where $H_\alpha : \RR^2 \to \RR_+$ is given by
\begin{equation}\label{eq:halpha}
 H_\alpha(x) = \frac{\pi |\Sigma|}{1 - \alpha}(1 - |x|^2)^{-\alpha} \mathbf 1_{|x| < 1}
\end{equation}
and $D_\varepsilon \in L^m(\Sigma)$ is a function supported in $B(p_\star, \varepsilon)$ which satisfies
\begin{eqnarray}\label{eq: Deps bound}
\|D_\varepsilon \circ \Psi_\varepsilon \|_{L^m(\mathbb B)} \leqslant C \varepsilon^{-2}\mathrm{err}(\varepsilon, \alpha) = C\varepsilon^{-2 + 2\alpha}.
\end{eqnarray}
Here we used that $\alpha<1/2$ and the fact that 
\begin{equation}\label{eq:psiepsvolg}
\Psi_\varepsilon^*{\vol_g} = \varepsilon^2 (1 + \mathcal O(\varepsilon)) \vol_{\RR^2}
\end{equation}
%In the proofs below, the following expansion is useful (see e.g. \cite[Lemma 3.2]{Nursultanov2023}), 
%\begin{eqnarray}\label{eq: dist in rescaled normal coordinate}
%d(\Psi_\varepsilon(x), \Psi_\varepsilon(y)) = \varepsilon |x-y| + \varepsilon |x-y| F(\varepsilon x, \frac{y-x}{|y-x|}, \varepsilon |x-y|)
%\end{eqnarray}
%where $F(x,\omega,t) :\mathbb D\times S^1\times \mathbb R_+ \to \mathbb R$ is smooth and $O(|x|) +O(t)$.
%Note that there is $C > 0$ such that one has
%$$
% C^{-1}|x-y| \leqslant \dd(\exp_{p_\star}(x), \exp_{p_\star}(y)) \leqslant C|x-y|
%$$
%for any $x,y \in \RR^2$ with $|x|, |y|\leqslant c$. In particular this yields
Since $\Psi_\varepsilon$ are geodesic coordinates, up to taking $c > 0$ small enough, one can assume that
\begin{eqnarray}\label{eq: dist in rescaled normal coordinate}
(1-\delta) \varepsilon |x-y| \leqslant \dd(\Psi_\varepsilon(x),\Psi_\varepsilon(y)) \leqslant (1+\delta)\varepsilon |x - y|
\end{eqnarray}
for any $x, y \in \RR^2$ with $|x|, |y|\leqslant 2c/\varepsilon$.
%Second, one has
%\begin{eqnarray}\label{eq: dist in rescaled normal coordinate2}
%\dd(\Psi_\varepsilon(x),\Psi_\varepsilon(y)) = \varepsilon |x-y|(1 + \mathcal O(\varepsilon))
%\end{eqnarray}
%uniformly for $x, y \in \mathbb B$.

\begin{lemma}\label{lem: S2Deps estimate}
If $\alpha < 1/2$ then there holds for any $x \in \RR^2$ such that $|x|\leqslant c/\varepsilon$
$$
|\mathrm S^2D_\varepsilon (\Psi_\varepsilon(x))| \leqslant C \varepsilon^{6\alpha-2} (1 + |x|)^{4\alpha-2}.
$$
\end{lemma}
\begin{proof}
Set $\tilde D_\varepsilon  = \varepsilon^2 (D_\varepsilon \circ \Psi_\varepsilon) \in\dot L^m(\mathbb B)$. Then there holds $\|\tilde D_\varepsilon\|_{L^m} \leqslant C_m\varepsilon^{2\alpha}$ for any $m\in ]1, 1/\alpha[$ thanks to \eqref{eq: Deps bound}. Using \eqref{eq:noyauS} and \eqref{eq: dist in rescaled normal coordinate} we have
\begin{eqnarray}\label{eq: S2Deps estimate1}
|S^2 D_\varepsilon(r)| \leqslant C \int_{B(p_\star, \varepsilon)} \dd(r,q)^{4\alpha-2}|D_\varepsilon(q)| \dd \vol_g (q) = C\varepsilon^{4\alpha-2}I_\varepsilon(x)
\end{eqnarray}
where $x = \Psi_\varepsilon^{-1}(r)$ and where $I_\varepsilon(x)$ is given by
$$I_\varepsilon(x) = \int_{\mathbb B}|x-y|^{4\alpha-2}| \tilde D_\varepsilon(y)| \dd y.$$
Now for $|x|>2$, we clearly have that
\begin{eqnarray}\label{eq: I for large x}
I_\varepsilon(x) \leqslant C |x|^{4\alpha-2}\|\tilde D_\varepsilon\|_{L^2} \leqslant C|x|^{4\alpha-2} \varepsilon^{2\alpha}.
\end{eqnarray}
Meanwhile, for $|x|\leqslant 2$,
\begin{eqnarray}\label{eq: I for small x}
I_\varepsilon (x) \leqslant  \||x-\cdot|^{4\alpha-2}\|_{L^2(\mathbb B)} \|\tilde D_\varepsilon\|_{L^2} \leqslant C \varepsilon^{2\alpha}
\end{eqnarray}
where we have used the fact that $\||x-\cdot|^{4\alpha-2}\|_{L^2(\mathbb B)}<C<\infty$ iff $2<1/\alpha$ which is our assumption. Combining \eqref{eq: S2Deps estimate1}, \eqref{eq: I for large x} and \eqref{eq: I for small x} we have the desired estimate.
\end{proof}
\begin{lemma}\label{lem:S2ofH}
We have the estimates
\begin{equation}\label{eq:galpha}
\begin{aligned}
&\varepsilon^{4\alpha - 2} (1 + \mathcal O(\varepsilon)) (1- \delta) G_\alpha(\Psi_\varepsilon^{-1}(r)) \\
&\hspace{3cm}\leqslant \varepsilon^{-2} \int C_\alpha \dd(r,q)^{2(2\alpha - 1)}  H_\alpha(\dd(q,p_\star)/\varepsilon) \dd \vol_g(q) \\
&\hspace{6cm}\leqslant \varepsilon^{4\alpha - 2} (1+ \delta) (1 + \mathcal O(\varepsilon)) G_\alpha(\Psi_\varepsilon^{-1}(r)) 
\end{aligned}
\end{equation}
uniformly for $r \in B(p_\star, c)$, where $G_\alpha \in \mathscr C( \RR^2)$ is given by \eqref{eq:Galpha}
\end{lemma}

Note that the above function $G_\alpha$ satisfies the bound
\begin{eqnarray}
\label{eq: G decay rate}
C^{-1}(1+|x|)^{4\alpha - 2} \leqslant G_\alpha(x) \leqslant C(1+|x|)^{{4\alpha-2}} \quad \text{for every } \ x \in \RR^2.
\end{eqnarray}

\begin{proof}
We let $r = \Psi_\varepsilon(x)$ and make the change of variable $\Psi_\varepsilon(q) = \varepsilon(y)$ into the integral appearing in \eqref{eq:galpha} to get
$$
\begin{aligned}
&\varepsilon^{-2} \int \dd(r,q)^{2(2\alpha - 1)}  H_\alpha(\dd(q,p_\star)/\varepsilon) \dd \vol_g(q) \\
& \hspace{2cm} = \varepsilon^{-2} \int_{\mathbb B} \dd(\Psi_\varepsilon(y),\Psi_\varepsilon(x))^{2(2\alpha - 1)}  H_\alpha(\dd(\Psi_\varepsilon(y),p_\star)/\varepsilon) (\Psi_\varepsilon^*\dd \vol_g)(y)
\end{aligned}
$$
Combining \eqref{eq:psiepsvolg} and \eqref{eq: dist in rescaled normal coordinate} we get the sought bounds.
%one obtains
%\begin{eqnarray}\nonumber
%&& \varepsilon^{-2} \int \dd(r,q)^{2(2\alpha - 1)}  H_\alpha(\dd(q,p_\star)/\varepsilon) \dd \vol_g(q) =\\\nonumber&& \varepsilon^{4\alpha-2} \int_{\mathbb B} |x-y|^{4\alpha-2}\left(1 + F\left(\varepsilon x, \frac{x-y}{|x-y|}, \varepsilon |x-y|\right)\right)^{2(2\alpha - 1)}  H_\alpha(y) (1 + \epsilon \mu(y)) dy\\
%\end{eqnarray}
%with $F(0,\omega,0) = 0$. Expanding 
%$$\left(1 + F\left(\varepsilon x, \frac{x-y}{|x-y|}, \varepsilon |x-y|\right)\right)^{2(2\alpha - 1)} = 1 + \varepsilon \tilde F_\epsilon(x, \frac{x-y}{|x-y|}, |x-y|)$$ then use non-negativity of $H_\alpha$ we get that 
%$$\left|\varepsilon^{-2} \int \dd(r,q)^{2(2\alpha - 1)}  H_\alpha(\dd(q,p_\star)/\varepsilon) \dd \vol_g(q) - \varepsilon^{4\alpha-2} G_\alpha(\Psi_\varepsilon^{-1}(r))\right| \leq C  \varepsilon^{4\alpha-1} G_\alpha(\Psi_\varepsilon^{-1}(r))$$
\end{proof}

Combining \eqref{eq:noyauS},  \eqref{eq: Feps principal term} and \eqref{eq: G decay rate} with Lemmas \ref{lem: S2Deps estimate} and \ref{lem:S2ofH} we obtain that for any $\varepsilon > 0$ small and every $r \in B(p_\star, c)$ there holds
\begin{equation}\label{eq:almostfin}
\begin{aligned}
 \varepsilon^{4\alpha - 2} (1- \delta)^3 G_\alpha(\Psi_\varepsilon^{-1}(r)) \leqslant  \mathrm S^2 F_\varepsilon(r) 
\leqslant  \varepsilon^{4\alpha - 2} (1+ \delta)^3 G_\alpha(\Psi_\varepsilon^{-1}(r)).
\end{aligned}
\end{equation}
Let $\chi \in \mathscr C^\infty(\Sigma)$ such that $\chi(q) = 1$ for $q \in B(p_\star, c)$ and $\chi(q) = 0$ for $q \notin B(p_\star, 2c)$. Then by \eqref{eq:s^2f} we have 
$$
\A_\varrho[(1 - \chi) \mathrm S^2F_\varepsilon](p) = \mathcal O(1)
$$
as $\varepsilon \to 0$. Thus $\A_\varrho \mathrm S^2F_\varepsilon(p) = \A_\varrho[\chi \mathrm S^2F_\varepsilon](p) + \mathcal O(1)$. Recalling \eqref{eq:sas} and using \eqref{eq:almostfin} we therefore get, by positivity of the operator $\A_\varrho$,
$$
\begin{aligned}
&\varepsilon^{4\alpha - 2} (1-\delta)^3 \A_\varrho[\chi (G_\alpha \circ \Psi_\varepsilon^{-1})](p) - C \\
&\hspace{2cm}\leqslant  \overline{u}_{\star, \varepsilon} - u_{\star, \varepsilon}(p) \\
&\hspace{4cm} \leqslant \varepsilon^{4\alpha - 2} (1+\delta)^3 \A_\varrho[\chi (G_\alpha \circ \Psi_\varepsilon^{-1})](p) + C.
\end{aligned}
$$
Since $\delta$ is arbitrary up to taking $c$ smaller we get the sought result.
%Now by definition of $A_\varrho$ one has
%\begin{equation}\label{eq:avarrhochi}
%\A_\varrho [\chi (G_\alpha \circ \Psi_\varepsilon^{-1})](p) = \int_{T_p\Sigma}|v|^{-1} \varrho(|v|) \chi(\exp_p(v)) G_\alpha(\Psi_{\varepsilon}^{-1}(\exp_p(v))) \dd v.
%\end{equation}
\end{proof}
%\begin{proof} This is a direct consequence of \eqref{eq:noyauS} combined with the estimates Lemma \ref{lem:S2ofH} then noting that the integrands are all non-negative.
%\end{proof}

\subsection{The finite degree case}\label{subsec:finitedegree}
Proposition \ref{prop:reduceintegral} motivates us to study the integral \eqref{eq:avarrhochi}. Hence we fix $\delta, c, C > 0$ and $\chi$ as in Proposition \ref{prop:reduceintegral}. To estimate \eqref{eq:avarrhochi}, we will consider the set 
$$
Y_\varrho(p,p_\star) = \{v \in T_p\Sigma~:~\exp_p(v) = p_\star,~|v| \in \supp \varrho\} \subset T_p\Sigma
$$
of vectors in $T_p\Sigma$ reaching $p_\star$ at time $1$. We also define 
\begin{eqnarray}\label{eq: def Yrhoc}
{Y}_{\varrho, \mathrm{c}}(p,p_\star) = Y_\varrho(p,p_\star) \cap \widetilde{\mathrm{Conj}}(p)
\end{eqnarray}
the set of vectors $v \in Y_\varrho(p,p_\star)$ such that $p_\star$ is conjugate to $p$ along $\gamma_{p,v}$. If $v \in Y(p,p_\star)$ is not a conjugate vector, then $\exp_p$ is a local diffeomorphism near $v$. This implies
$$
v \in Y_\varrho(p,p_\star) \setminus Y_{\varrho, \mathrm c}(p,p_\star) \quad \implies \quad v \text{ is isolated in } Y_\varrho(p, p_\star).
$$
In what follows, we let $k$ be the maximal degree among all conjugate directions at $p$ reaching $p_\star$, that is,
$$
k = \sup_{v \in Y_{\varrho, \mathrm c}(p,p_\star)} \deg(v)
$$
where $\deg(v)$ is the degree of $(p, v / |v|)$, see Definition \ref{def:degree}. This corresponds to the degree $\deg(p,p_\star)$ of the pair $(p,p_\star)$ see Definition \ref{def:degreepair}. If $v \in Y_\varrho(p,p_\star) \setminus Y_{\varrho, \mathrm c}(p,p_\star)$ then we will set $\deg(v) = 0$. Finally, note that Proposition \ref{prop:normalformexp} implies
\begin{equation}\label{eq:degisolated}
\deg(v) < \infty \quad \implies \quad v \text{ is isolated in }Y_{\varrho, \mathrm c}(p, p_\star).
\end{equation}

In this paragraph, we assume that $k < \infty$. In that case, \eqref{eq:degisolated} implies that $Y_{\varrho}(p, p_\star)$ is finite, and we write $Y_{\varrho}(p, p_\star) = \{v_1, \dots, v_N\}$. Let $j \in\{1, \dots, N\}$ and set $k_j = \deg(v_j) \geqslant 0$, so that $k = \max_j k_j$. Here by convention, we put $\deg(v_j) = 0$ if $v_j$ is not a conjugate direction. Then by Proposition \ref{prop:normalformexp} there exist neighborhoods $U_j \subset T_p\Sigma$, $V_j\subset \Sigma$ of $v_j\in U_j$ and $p_\star\in V_j$, respectively, as well as neighborhoods $U_j', V_j' \subset \mathbb B^2$ and diffeomorphisms $\kappa_j: U_j \to U_j'$ with $\kappa_j(v_j) = 0$ and $\eta_j : V_j \to V_j'$ with $\eta_j(p_\star) = 0$ such that for any $(s,u) \in U_j'$ one has
\begin{equation}\label{eq:psij}
\psi_j(s,u) = \begin{cases}
(s, u) & \text{ if } k_j = 0, \\
(s, u^{2}) &\text{ if } k_j = 1, \\
(s, su + u^{k_j + 1}\rho_j(s,u)) &\text{ if } k_j \geqslant 2, 
\end{cases}
\qquad \text{where} \quad
\psi_j =  \eta_j \circ {\exp_p} \circ\kappa_j^{-1}.
\end{equation}
for some smooth $\rho_j$ with $\rho_j(0) = 1$. Up to taking $c$ small enough, we may assume that
$$ 
\Bigl\{v \in T_p\Sigma~:~|v| \in \supp \varrho,~\exp_p(v) \in B(p_\star, c)\Bigr\} \subset \bigcup_{j=1}^N U_j.
$$

Note that $\Psi_\epsilon$ defined in \eqref{eq: def of rescaled geodesic coord} can be written as $\Psi_\varepsilon(x) = \Psi(\varepsilon x)$ where $\Psi$ is the unscaled geodesic coordinate. This gives that 
$$\Psi_\varepsilon^{-1}(q) = \varepsilon^{-1} \Psi^{-1}(q).$$
Next we let $\tilde \varrho(s) = s^{-1}\varrho(s)$. Then by a change of variable and \eqref{eq:avarrhochi} one obtains
\begin{equation}\label{eq:avarrhochi2}
\A_\varrho \left[\chi (G_\alpha \circ \Psi_\varepsilon^{-1})\right](p) = \sum_{j=1}^N \int_{U_j'} \chi_j(s,u) {G}_{\alpha}\Bigl(\varepsilon^{-1}(\Psi_j^{-1}\circ \psi_j)(s,u)\Bigr) \dd s \dd u
\end{equation}
where
$$
\chi_j(s,u) = \widetilde\varrho(|\kappa_j^{-1}(s,u)|) \det (\dd \kappa_j(s,u))  (\chi \circ \kappa_j)^{-1}(s,u) \quad \text{and} \quad \Psi_j = \eta_j \circ \Psi : \RR^2 \to \RR^2.
$$
Point (ii) of Theorem \ref{thm:main} will be a consequence of the following technical result.
\begin{lemma}\label{lem: leading term of integral}
Let $U \subset \RR^2$ be an open neighbourhood of the origin and $\Phi : U \to \Phi(U)\subset \mathbb R^2$ be a smooth diffeomorphism with $\Phi(0) = 0$. Let $\chi \in \mathscr C^\infty_c(U)$ such that $\chi(0) > 0$. Let $G : \RR^2 \to \RR_+$ be a continuous function 
%{\color{red} you will be applying this to $G_\alpha$. Is it a continuous function? Also in the proof of this lemma you apply Sard's thm to $G$ which will require more than just continuity (I think).}
such that for some $\omega > 0$ we have
\begin{equation}\label{eq:Ghypothesis}
G(x) \leqslant (1 + |x|)^{-\omega} \quad \text{ for all } x \in \RR^2.
\end{equation}
Finally let $\ell \geqslant 1$,
%{\color{red}I think you need $\ell\geq 2$ for the condition below to hold for the application you have in mind.}
$a \in \RR$ and $\rho \in \mathscr C^\infty(\RR^2, \RR_+)$ a positive, bounded smooth function such that $\rho(0) = 1$ and set 
$$
\psi(s,u) = (s, asu + u^\ell \rho(s,u)).
$$
Assume that $1 + 1/\ell < \omega$. 
%{\color{red} I think in your applications $\omega = 2-4\alpha$ with $\alpha<1/4$ (see \eqref{eq: G decay rate}). So $\omega$ is barely bigger than $1$ and I am not sure if the assumption $1+1/\ell<\omega$ can be satisfied unless $\ell$ is VERY large.}
Then we have the expansion as $\varepsilon\to 0$
\begin{equation}\label{eq:integralsimlemma}
\int_{\RR^2} \chi(s,u) G\bigl(\varepsilon^{-1} (\Phi \circ \psi)(s,u)\bigr) \dd s \dd u \sim  \varepsilon^{1 + 1/\ell} \chi(0) \int_{\RR^2} G\bigl(\dd \Phi(0) \cdot (s, u^\ell)\bigr) \dd s \dd u.
\end{equation}
If $\omega \geqslant 1 + 1/\ell$ then there is $C$ such that for any small $\varepsilon > 0$
\begin{equation}\label{eq:upperboundomega}
\left|\int_{\RR^2} \chi(s,u) G\bigl(\varepsilon^{-1} (\Phi \circ \psi)(s,u)\bigr) \dd s \dd u\right| \leqslant \begin{cases} 
C\varepsilon^\omega |\log \varepsilon| & \text{ if } \omega = 1 + 1/\ell, \\
C\varepsilon^\omega & \text{ if } \omega > 1 + 1 / \ell.
\end{cases}
\end{equation}

\end{lemma}
\begin{proof}
In what follows, we will denote by $H_\varepsilon(s,u)=  \chi(s,u) G\bigl(\varepsilon^{-1} (\Phi \circ \psi)(s,u)\bigr)$, the integrand of the integral on the LHS of  \eqref{eq:integralsimlemma}. Then we have
\begin{equation}\label{eq:firstint}
\int_{\RR^2} H_\varepsilon(s,u) \dd s \dd u = \int_0^\infty \vol\left\{H_\varepsilon > t\right\} \dd t.
\end{equation}
The change of variable $\mathrm{w}_\varepsilon : (s,u) \mapsto (\varepsilon s, \varepsilon^{1/\ell}u)$ yields
\begin{equation}\label{eq:changeneps}
\vol\left\{H_\varepsilon > t\right\} = \varepsilon^{1 + 1/\ell} N_\varepsilon(t) \quad \text{where} \quad N_\varepsilon(t) = \vol\{H_\varepsilon \circ \mathrm{w}_\varepsilon > t\}.
\end{equation}
So \eqref{eq:firstint} becomes
\begin{eqnarray}\label{eq:secondint}
\int_{\RR^2} H_\varepsilon(s,u) \dd s \dd u = \varepsilon^{1 + 1/\ell} \int_0^\infty   N_\varepsilon(t) \dd t.
\end{eqnarray}
Now we have $\psi \circ \mathrm{w}_\varepsilon(s,u) = \varepsilon (s, a\varepsilon^{1/\ell}su + u^{\ell}\rho_{\varepsilon}(s,u))$ where $\rho_{\varepsilon} = \rho \circ \mathrm{w}_\varepsilon$. So we have
\begin{eqnarray}\label{eq: Heps rescaled by w}
H_\varepsilon \circ \mathrm{w}_\varepsilon(s,u) = \chi(\mathrm{w}_\varepsilon(s,u)) G\bigl(\varepsilon^{-1}\Phi_{\varepsilon}\bigl(\psi_{\varepsilon}(s,u)\bigr)\bigr)
\end{eqnarray}
%where $\Psi_{\varepsilon} = \varepsilon^{-1} \Psi (\varepsilon\, \cdot)$ {\color{red} I am confused. I thought \eqref{eq: def of rescaled geodesic coord} already defines $\Psi_\varepsilon$.} 
where $\Phi_{\varepsilon} =\Phi (\varepsilon\, \cdot)$ 
\begin{equation}\label{eq:psi_epsilon}
\psi_{\varepsilon}(s,u) = (s, a\varepsilon^{1/\ell}su + u^{\ell}\rho_{\varepsilon}(s,u))
\end{equation}
In what follows we denote $L = \dd\Phi(0) \in \mathrm{GL}(\RR^2)$. Since $\chi$ is compactly supported, there is $c_1 > 0$ such that
\begin{equation}\label{eq:linearizedPsieps}
\varepsilon^{-1}\left|
\Phi_{\varepsilon}(x)\right| \geqslant  c_1|x| \qquad \text{ for all } \ \varepsilon > 0 \ \text{ and } \ x \in \supp \chi.
\end{equation}
Since $\rho(0) = 1$ we see that for any $(s,u) \in \RR^2$ there holds
\begin{equation}\label{eq:claimbis}
H_\varepsilon(\mathrm{w}_\varepsilon(s,u)) \to H_0(s,u) = \chi(0) G(L(s, u^\ell))
\end{equation}
%{\color{red} do you mean $u^\ell$ instead of $u^3$?}
as $\varepsilon \to 0$.   %By Sard's theorem the set of singular values of $H_0$ is of Lebesgue measure $0$ {\color{red} $G$ (and therefore $H_0$) is only assumed to be continuous. My understand of Sard's thm is that it applies to smooth functions. Is there something I am missing here? On the other hand, the $G_\alpha$ that you wish to apply this to is smooth away from the boundary of the unit ball so maybe this is only a matter of rephrasiing the Lemma?}. Hence, for almost every $t > 0$ (i.e. those which are not singular values of $H_0$), 
First note that there is $D > 0$ such that for all $\varepsilon, T > 0$ one has
\begin{equation}\label{eq:implication}
\Bigl((s,u) \in \supp(\chi \circ \mathrm w_\varepsilon) \ \ \text{and} \ \ |\psi_\varepsilon(s,u)| \leqslant T\Bigr) \ \implies \   \Bigl( |s|\leqslant T \ \ \text{and} \ \ |u| \leqslant DT^{1/\ell}\Bigr).
\end{equation}
%Indeed, suppose $(s,u) \in \supp(\chi \circ \mathrm w_\varepsilon)$ and $|\psi_\varepsilon(s,u)| \leqslant T$. Then the first condition yields $|u| \leqslant C_1 \varepsilon^{-1/\ell}$ for some $C_1 > 0$. The second one yields $|a\varepsilon^{1/\ell}su + u^\ell\rho_\varepsilon(s,u)| \leqslant T$ as well as $|s|\leqslant T$. Now if $c_1 > 0$ is such that $\rho \geqslant c_1$ on $\supp \chi$ then we obtain
%$
%|u|^\ell \leqslant (c^{-1}_1 + C_1|a|)T,
%$
%which yields \eqref{eq:implication}. 
Now combining \eqref{eq:linearizedPsieps} and \eqref{eq:implication} we see that there is $C > 0$ such that
\begin{equation}\label{eq:implication2}
\begin{aligned}
&\Bigl\{(s,u) \in \supp(\chi \circ \mathrm w_\varepsilon)~:~\varepsilon^{-1}|\Phi_{\varepsilon}(\psi_\varepsilon(s,u))|\leqslant T\,\Bigr\}\\
& \hspace{4cm} \subset \Bigl\{(s,u)~:~|s|\leqslant CT \text{ and } |u| \leqslant CT^{1/\ell}\Bigr\}
\end{aligned}
\end{equation}
for any small $\varepsilon > 0$ and any $T > 0$. However it follows from \eqref{eq:Ghypothesis} that there is $D > 0$ such that 
$
\{G > t\} \subset \bigl\{x \in \RR^2~:~|x|\leqslant Dt^{-1/\omega}\bigr\}.
$
Therefore, using expression \eqref{eq: Heps rescaled by w} for $H_\varepsilon \circ \mathrm w_\varepsilon$ and the set inclusion \eqref{eq:implication2}, we obtain that for some $C_2 > 0$ there holds
\begin{equation}\label{eq:boundsu}
\bigl\{H_\varepsilon \circ \mathrm w_\varepsilon > t\bigr\} \subset \bigl\{(s,u)~:~|s|\leqslant C_2t^{-1/\omega} \text{ and } |u| \leqslant C_2t^{-\frac{1}{\ell \omega}}\bigr\}
\end{equation}
for any $t > 0$ and small $\varepsilon > 0$. Of course the same estimate holds also if we replace $H_\varepsilon \circ \mathrm w_\varepsilon$ by $H_0$. In particular there is $C_3$ such that
\begin{equation}\label{eq:boundneps}
N_\varepsilon(t) \leqslant C_3 t^{-1/\omega - 1/(\ell \omega)} \quad \text{ for all } t > 0 \text{ and small } \varepsilon \geqslant 0.
\end{equation}
Next recall that we assumed 
$
1 / \omega + 1 /( \omega \ell) < 1.
$
Therefore $t \mapsto t^{-1/\omega - 1/(\ell \omega)}$ is integrable near the origin, and so is $t \mapsto N_0(t)$, hence we get 
$$
\int H_0(s,u) \dd s \dd u = \int_0^\infty N_0(t) \dd t < \infty
$$
which means $H_0 \in L^1(\RR^2)$. In particular, $\vol\{H_0 = t\} < \infty$ for almost every $t > 0$ and for such $t$ the convergence \eqref{eq:claimbis} yields
\begin{eqnarray}\label{eq: convergence of indicator}
\mathbf 1_{\{H_\varepsilon \circ \mathrm w_\varepsilon > t\}} \to \mathbf 1_{\{H_0 > t\}} \quad \text{ almost everywhere on } \RR^2
\end{eqnarray}
as $\varepsilon \to 0$. Hence by \eqref{eq:boundsu} we may apply the dominated convergence theorem to obtain
\begin{equation}\label{eq:claim}
N_\varepsilon(t) \to N_0(t) = \vol\{H_0 > t\} \quad \text{ for almost every } t > 0,
\end{equation}
as $\varepsilon \to 0$. Combining this with the bound \eqref{eq:boundneps} we may once again use the dominated convergence theorem to take the limit in the integral in the right-hand side of \eqref{eq:secondint} to obtain
$$
\int_{\RR^2} \chi(s,u) G\Bigl(\varepsilon^{-1} (\Phi \circ \psi)(s,u)\Bigr) \dd s \dd u \sim \varepsilon^{1 + 1/\ell} \int_0^\infty \vol\{H_0 > t\} \dd t
%+ o(\varepsilon^{1 + 1/\ell})
$$
as $\varepsilon\to 0$. By \eqref{eq:claimbis} the integral in the right-hand side is exactly 
$$\chi(0) \int_{\RR^2} G\bigl(\dd \Phi(0) \cdot (s, u^\ell)\bigr) \dd s \dd u<\infty.$$ Hence it remains to deal with the case $\omega \geqslant 1 + 1/\ell$. For this we notice that the integral $I_\varepsilon$ in the left-hand side of \eqref{eq:integralsimlemma} is bounded by
$$
I_\varepsilon \leqslant \int_{\supp \chi} \left(1 + \varepsilon^{-1}|\Phi \circ \psi(s,u)|\right)^{-\omega} \dd s \dd u.
$$
Since $\chi$ is compactly supported in $B(p_\star, c)$, up to reducing $c > 0$ we may assume that there are $c_2, c_3 > 0$ such that
$$
|\Phi \circ \psi(s,u)| \geqslant c_2 (|s| + |su + u^\ell\rho(s,u)|) \geqslant c_3(|s| + |u|^\ell) \quad \text{for all }(s,u) \in \supp \chi.
$$
Hence we obtain the estimate
$$
I_\varepsilon \leqslant \int_{\supp \chi} \left(1 + c_3 \varepsilon^{-1}(|s| + |u|^\ell)\right)^{-\omega} \dd s \dd u.
$$
Using again the change of variables $\mathrm{w}_\varepsilon : (s,u) \mapsto (\varepsilon s, \varepsilon^{1/\ell}u)$ we see that \eqref{eq:upperboundomega} holds.
\end{proof}

Because of \eqref{eq: G decay rate} we may apply the above lemma to each term in the sum in the right-hand side of \eqref{eq:avarrhochi2} and deduce that for each $j$ there is $c_j > 0$ such that as $\varepsilon \to 0$ we have
\begin{equation}\label{eq:integralsim}
 \int_{U_i'} \chi_j(s,u) {G}_{\alpha}\Bigl(\varepsilon^{-1}(\Psi_j^{-1} \circ \psi_j)(s,u)\Bigr) \dd s \dd u \sim c_j \varepsilon^{1 + \frac{1}{k_j + 1}}
\end{equation} 
provided that $\omega = 2 - 4\alpha > 1 + 1/(k_j + 1)$. If $\omega \leqslant 1 + 1 / (k_j + 1)$ instead this integral is bounded by $C\varepsilon^\omega |\log \varepsilon|$. This is negligible in front of $\varepsilon^{1 + 1/(k+1)}$ if $\omega > 1 + 1 / (k + 1)$ where $k = \max k_j$. Hence we get
$$
\A_\varrho \left[\chi (G_\alpha \circ \Psi_\varepsilon^{-1})\right](p) \sim c_p \, \varepsilon^{1 + 1 / (k+1)}
$$
for some $c_p > 0$, provided $\omega > 1 + 1 / (k + 1)$, which is equivalent to
$$
4\alpha < 1 - 1/(k+1).
$$
Recalling Proposition \ref{prop:reduceintegral} we obtain that whenever $\varepsilon$ is small enough
$$
c_p(1-\delta)^2\varepsilon^{1 + 4\alpha + 1/(k+1)} \leqslant \overline{u}_{\star, \varepsilon} - u_{\star, \varepsilon}(p) \leqslant c_p(1+\delta)^2\varepsilon^{1 + 4\alpha + 1/(k+1)}.
$$
Since $\delta > 0$ is arbitrary we obtain (ii) of Theorem \ref{thm:main}.

\subsection{The infinite degree case}\label{subsec:infinitedegree}
In this paragraph we assume $k = \infty$. The set $Y_\varrho(p, p_\star)$ is compact hence one can find $v_1, \dots, v_N \in Y_\varrho(p, p_\star)$, neighbourhoods $U_j, U_j', V_j, V_j'$ as well as diffeomorphisms $\kappa_j, \eta_j$ with the same properties as in the preceding case, except that there must be some $j$ with $k_j = \infty$, and in that case one has 
$$
\psi_j(s,u) = (s, su + \varphi_j(s, u))
$$
for some $\varphi_j : \RR^2 \to \RR$ sending $0$ on $0$ and such that $\partial_u^\ell\varphi_j(0) = 0$ for each $\ell \in \mathbb N_{\geqslant 0},$ see Proposition \ref{prop:normalformexp}. (Indeed, if by contradiction we had $k_j < \infty$ for all $j$ then there would hold $Y_{\varrho}(p,p_\star) \cap U_j = \{v_j\}$ and $k$ would be finite.) 
%In what follows for continuous $\varphi : \RR \to \RR$ we will make use of the notation
%$$
%m_{\varphi}(\varepsilon) = \sup_{|s|\leqslant \varepsilon}\vol\bigl\{u \in [-1,1]~:~|\varphi(s,u)| \leqslant \varepsilon\bigr\}, \quad \varepsilon > 0, 
%$$
Suppose $k_j = \infty$. By \eqref{eq: G decay rate} one has
$$
I_j(\varepsilon) \asymp \int_{\supp \chi} \left(1 + \varepsilon^{-1}|\Psi^{-1}_j(s, su + \varphi_j(s,u))|\right)^{-\omega} \dd s \dd u.
$$
where $f(\varepsilon) \asymp g(\varepsilon)$ means that there is $C > 0$ such that $C^{-1}g(\varepsilon) \leqslant f(\varepsilon) \leqslant C g(\varepsilon)$ for every small $\varepsilon > 0$. Next, up to reducing $c > 0$, there is $C_3 > 0$ such that 
$$
C_3^{-1} (|s| + |\varphi(s,u)|)\leqslant |\Psi_j^{-1}(s, su + \varphi_j(s,u))|\leqslant C_3(|s| + |\varphi(s,u)|)
$$
for all $(s,u) \in \supp \chi$. This yields 
$$
I_j(\varepsilon) \asymp \int_{\supp \chi} \left(1 + \varepsilon^{-1}|s| + \varepsilon^{-1} |\varphi_j(s,u)|\right)^{-\omega} \dd s \dd u.
$$
Changing of variable then yields
\begin{equation}\label{eq:f_j}
I_j(\varepsilon) \asymp \varepsilon f_j(\varepsilon) \quad \text{where} \quad f_j(\varepsilon) = \int_{K_\varepsilon} (1 + |s| + \varepsilon^{-1}\varphi_j(\varepsilon s, u))^{-\omega} \dd s \dd u
\end{equation}
where $K_\varepsilon = \{(\varepsilon^{-1}s, u)~:~(s,u) \in \supp \chi\}.$
Next, we claim that for any $\delta > 0$ there is $c_\delta > 0$ such that for any small $\varepsilon $
\begin{equation}\label{eq:mvarphi}
f_j(\varepsilon) \geqslant c_\delta \varepsilon^\delta.
\end{equation}
%Indeed, for $|s|\leqslant \varepsilon$ we have $\varphi_j(s,u) = \varphi_j(0,u) + \mathcal O(\varepsilon)$. In particular for some $c_1> 0$ one has 
%$$
%f_j(\varepsilon) \geqslant \vol\{u \in [-1,1]~:~|\varphi(0,u)|\leqslant c\varepsilon\}
%$$
%However for any $N$ we have $|\varphi_j(0, u)| = \mathcal O(|u|^\infty)$ which gives \eqref{eq:mvarphi}. Next, recalling \eqref{eq:avarrhochi2} and proceeding as in the preceding case, one therefore obtains
Indeed, up to taking $c > 0$ smaller, there is $C_4 > 0$ such that for any $(s,u) \in K_\varepsilon$ 
$$
|\varphi_j(\varepsilon s, u)| \leqslant |\varphi_j(0,u)| + C_4
$$
Hence we obtain that for any small $\varepsilon$
$$
f_j(\varepsilon) \geqslant C_5 \int_{K_\varepsilon} (1 + |s| + \varepsilon^{-1}|\varphi_j(0,u)|)^{-\omega} \dd s \dd u \geqslant C_6 \int_{|u|\leqslant c} (1 + \varepsilon^{-1}|\varphi_j(0,u)|)^{-\omega} \dd u.
$$
Now let $\delta>0$. Since $\varphi_j(0,\cdot)$ is flat at $0$, for every integer $N\geqslant 1$ there exists $C_N>0$ such that
$$
|\varphi_j(0,u)|\leqslant C_N |u|^N
$$
for all $u$ sufficiently small. Choose $N$ so large that $N\delta>1$. Then, for $|u|\leqslant C_N^{-1/N}\varepsilon^\delta$, one has
$
|\varphi_j(0,u)|
\leqslant
\varepsilon^{N\delta}
\leqslant
\varepsilon
$
for $\varepsilon$ small enough. Hence
$
\left(1+\varepsilon^{-1}|\varphi_j(0,u)|\right)^{-\omega}
\geqslant
2^{-\omega}
$
for $|u|\leqslant C_N^{-1/N}\varepsilon^\delta$ and we get
$$
f_j(\varepsilon)
\geqslant
C_6
\int_{|u|\leqslant C_N^{-1/N}\varepsilon^\delta}
2^{-\omega}\,\dd u
\geqslant
c_\delta \varepsilon^\delta,
$$
which proves \eqref{eq:mvarphi}. Now let $f(\varepsilon) = \sum_{k_j = \infty} f_j(\varepsilon)$. Then recalling \eqref{eq:avarrhochi2} and \eqref{eq:f_j} and making use Proposition \ref{prop:reduceintegral} and of the estimates obtained in \S\ref{subsec:finitedegree} for $k_j < \infty$, one obtains
$$
D^{-1}\varepsilon^{1+4\alpha}  f(\varepsilon) \leqslant \overline{u}_{\star, \varepsilon} - u_{\star, \varepsilon}(p) \leqslant D\varepsilon^{1+4\alpha} f(\varepsilon)
$$
fore some $D > 0$ and some non-decreasing function $f : \RR_+^* \to \left]0,2\right]$ such that
$$
\varepsilon^{-\delta} f(\varepsilon) \underset{\varepsilon \to 0}{\longrightarrow} \infty \quad \text{for every } \delta > 0.
$$
This proves (iii) of Theorem \ref{thm:main}.
%Now, assume that $V\tau$ vanishes on some interval of the circle $S_p\Sigma$. Then up to changing the $v_j$'s one can assume that $\varphi_j = 0$ for some $j$. In that case, the above function $f$ is just the constant function and we get $D^{-1}\varepsilon^{-1 + 4\alpha} \leqslant \overline{u}_{\star, \varepsilon} - u_{\star, \varepsilon}(p) \leqslant D \varepsilon^{-1 + 4\alpha}$. Hence we proved (iii) in Theorem \ref{thm:main}.

Finally, assume that $\mathrm{Conj}(p_\star)$ is a point $p$. In that case we have $V\tau = 0$ on $\Lambda_{p_\star}$ hence Lemmas \ref{lem:conjtime} and \ref{lem:conjmap} yield $\tau(z) = \pi$ for each $z \in \Lambda_{p_\star}$ and $\exp_{p_\star}(v) = p$ for $v \in \pi\Lambda_{p_\star}$. These properties also hold reversing the roles of $p_\star$ and $p$ and we obtain that for $v \in T_p\Sigma$ such that $|v|$ is close to $\pi$,
$$
\dd(p_\star, \exp_p(v)) = ||v| - \pi|.
$$
In particular up to taking $c$ smaller, the above equality holds whenever $\exp_p(v) \in \supp \chi$. Now recall $|\Psi_{\varepsilon}^{-1}(q)| = \dd(p_\star, q) / \varepsilon$ for $q \in B(p_\star, c)$. Hence taking radial coordinates in $T_p\Sigma$ we see that the integral \eqref{eq:avarrhochi} can be written as
$$
\int_{v \in \pi \Lambda_p}  \int_{\pi-c}^{\pi+c} \varrho(t) \chi_1(t,v) G_\alpha(\varepsilon^{-1}|t-\pi|) t \dd t \dd v = \int_{-c}^c \chi_2(t) G_\alpha(\varepsilon^{-1}|t|) \dd t
$$
for some smooth functions $\chi_1 \in \mathscr C^\infty_c(\Lambda_{p} \times \left]\pi-c, \pi+c\right[)$ and $\chi_2 \in \mathscr C^\infty_c(\left]-c,c\right[)$ such that $\chi_2(0) > 0$. Now using \eqref{eq:Galpha} it is not hard to see that $\eqref{eq:avarrhochi}$ is asymptotic to $c_4\,\varepsilon^{-1 + 4\alpha}$ when $\varepsilon \to 0$, for some $c_4 > 0$. Combining this with Proposition \ref{prop:reduceintegral} proves that (iv) holds in Theorem \ref{thm:main}.

{}{}{}{
\subsection{The non orientable case}\label{subsec:notorientable}
We conclude this section by a few words on the non-orientable case. If $(\Sigma, g)$ is a non orientable Zoll surface then Theorem \ref{thm:zolltopology} implies that every point on a closed geodesic $\gamma$ is conjugated only to itself along $\gamma$. Then it follows from classical wavefront set considerations that the operator $A_\varrho$ defined by \eqref{eq: def A rho} is smoothing. Indeed, by \cite[Lemmas 4.3.2, 4.3.4]{Lefeuvre2025} if $u \in \mathcal D'(\Sigma)$ we have 
$$
\WF(\mathrm R_\varrho \pi^* u) \subset \Bigl\{(z, \zeta)~:~ \exists t \in \supp \varrho,\, \exists (q, \eta) \in T^*\Sigma \setminus \underline 0~:~\zeta = \dd \varphi_{t}(z)^\top \dd \pi(z)^\top\eta\Bigr\}
$$
where $\mathrm R_\varrho$ is defined in \eqref{eq:relationrrho}. Applying once more \cite[Lemma 4.3.4]{Lefeuvre2025} we obtain 
$$
\WF(\mathrm A_\varrho u) \subset \Bigl\{(p, \xi)~:~ \exists t \in \supp \varrho,\, \exists (q, \eta) \in T^*\Sigma \setminus \underline 0~:~\dd \pi(p)^\top\xi = \dd \varphi_{t}(z)^\top \dd \pi(z)^\top\eta\Bigr\}.
$$
The set on the right-hand side is empty.  Indeed, the relation above can hold only if $p$ and $q$ are conjugate along a geodesic segment of length $t$, but by Theorem~\ref{thm:zolltopology} no such conjugate pairs exist for $t\in\supp\varrho$. Hence $\WF(\mathrm A_\varrho u) = \emptyset$. This implies that $\mathrm A_\varrho$ is smoothing and Proposition \ref{prop:generator} yields that $\mathrm A$ is an elliptic classical pseudo-differential operator of order $2\alpha$ with principal symbol $|\eta|_g^{2\alpha}$. In particular the full parametrix $\mathrm A^+$ is an elliptic classical pseudo-differential operator with symbol $|\eta|_g^{-2\alpha}$.
}
{}{}{}{
In particular we may reproduce the arguments above (see \cite[\S4.9]{chaubet2025levy}) to obtain 
$$
u_{\star, \varepsilon}(p) = \overline{u}_{\star, \varepsilon} + \mathcal O(1)
$$
as $\varepsilon \to 0$, for any $p \neq p_\star$.
}

%%%% APPENDIX %%%%

\appendix

\section{Normal forms for 1-singular maps of the plane}
\label{app:morin}

\begin{definition}[$1$-singular maps]
  Let $f : \RRR^2 \to \RRR^2$ be a smooth map such and let $\lambda = \det \dd f$. We say that $f$ is \textit{$1$-singular at $x$} if $\lambda(x) = 0$ and $\dd \lambda(x) \neq 0$.
\end{definition}
The  following result is immediate.
\begin{lemma}\label{lem:degree}
  Let $f : \RRR^2 \to \RRR^2$ be a smooth map which is $1$-singular at the origin and set $\lambda = \det \dd f : \RRR^2 \to \RRR.$ Then the singular set $C = \{\lambda = 0\}$ is locally a one-dimensional submanifold of $\RRR^2$ and $\operatorname{rank} \dd f(x) = 1$ for every $x \in C$ close enough to $0$. Let $\eta : \RRR^2 \to \RRR^2$ be a non-vanishing map such that $\eta(x) \in \ker \dd f(x)$ for all $x \in C$ close enough to zero. Let $\tilde \lambda : \RRR^2 \to \RRR$ be a local defining function for $C$, that is, we have $C = \{\tilde \lambda = 0\}$ and $\dd \tilde \lambda|_{C} \neq 0$. In what follows we will denote
   $$
  k(f) = \inf\{k \in \mathbb N~:~\eta^k \tilde \lambda(0) \neq 0\} \in \mathbb N \cup \{\infty\}.
  $$
  Then the following holds
  \begin{enumerate}[label=\emph{(\roman*)}]
   \item The number $k(f)$ is independent of the choice of the vector field $\eta$ and the local defining function $\tilde \lambda$. 
   \item This number $k(f)$ is invariant under changes of coordinates: if $\tilde \kappa, \kappa : \RRR^2 \to \RRR^2$ are any maps such that $\tilde \kappa(0) = \kappa(0) = 0$ that are local diffeomorphisms near the origin then $\tilde f = \tilde \kappa \circ f \circ \kappa$ is $1$-singular at the origin and $k(\tilde f) = k(f)$.
  \end{enumerate}
\end{lemma}

\begin{proof}
By assumption $\dd\lambda(x)\neq 0$ for $x$ near the origin so by the implicit function theorem the set $C$ is a $1$-dimensional submanifold of $\RR^2$ near the origin. By definition, $\dd f(x)$ is not invertible for any $x\in C$. Hence its rank is at most $1$ and thus exactly $1$ by upper semi-continuity of the rank.
%To see that the rank is indeed $1$, suppose for some $x\in C$ near the origin $\dd f(x)$ has null-rank. That is, $\\d f(x)$ is a $2\times 2$ matrix with only zero entries. This would mean that $\lambda (x) = 0$ to second order which contradicts the $\dd \lambda(x) \neq0$ definition of $1$-singular.

To verify (i), let $\tilde \lambda$ be a local boundary defining function of $C$. Then $\tilde \lambda(x) = c(x) \lambda(x)$ for some smooth positive function $c$ defined near the origin. Similarly, suppose $\tilde \eta$ is another non-vanishing vector field such that 
$$
\tilde \eta(x)\in {\rm ker}(\dd f(x)) = {\rm span}\{\eta(x)\} \quad \text{ for every }x \in C.
$$
Then we must have $\tilde \eta(x) = a(x) \eta(x)$ for some positive function $a$ defined near the origin. Suppose 
$$
\eta^k \lambda(0)  \neq 0 \quad \text{ but } \quad \eta^j \lambda(0) = 0 \quad \text{ for all } j\leqslant k-1.
$$
Then we have that $\tilde \eta^j \tilde\lambda(0) = 0 $ for all $j\leqslant k-1$ while $\tilde \eta^k \tilde \lambda = a(0)c(0) \eta^k \lambda(0)\neq 0$. The converse holds by swapping the role of $\tilde \eta $ with $\eta$ and $\tilde \lambda$ with $\lambda$. So we have that $k(f)$ is independent of choice of local defining function and the vector field generating ${\rm ker}(\dd f(x))$ at every $x \in C$.

To verify (ii), let $\tilde f = \tilde \kappa  \circ f\circ \kappa$ and set $\lambda_{\tilde \kappa} = \det \dd \kappa$ and define $\lambda_\kappa$ similarly. Then
\begin{eqnarray}\label{eq: tilde la}
\tilde \lambda = {\rm det}\,\dd \tilde f = (\lambda_{\tilde \kappa} \circ f \circ \kappa) \cdot (\lambda \circ f) \cdot \lambda_\kappa
\end{eqnarray}
vanishes at the origin. Furthermore $\dd \tilde\lambda(0) =   \det \dd \tilde\kappa \det\dd \kappa \dd \lambda(0) \neq 0$ so $\tilde f$ is also $1$-singular. To verify that $k(f) = k(\tilde f)$ we suppose 
$\tilde \eta \in {\rm ker}(\dd \tilde f)$ so that 

\begin{eqnarray}\label{eq: k condition}
\tilde \eta^k \tilde\lambda(0)\neq 0,\quad {\rm but}\quad \tilde \eta^j\tilde\lambda(0) = 0\quad {\rm for} \quad j\leqslant k-1.
\end{eqnarray}
Then $\eta = \dd\kappa(\tilde \eta) \in \ker \dd f $ and by \eqref{eq: tilde la}
$$ \eta^j \lambda=\tilde \eta^j  (\lambda\circ \kappa) = \tilde\eta^j\left (\tilde \lambda  \left((\det \dd \tilde\kappa)\circ f\circ\kappa \right)^{-1} (\det\dd \kappa)^{-1}\right).$$
The function $\tilde \lambda  \left((\det \dd \tilde\kappa)\circ f\circ\kappa \right)^{-1} (\det\dd \kappa)^{-1}$ is a defining function for $\{\tilde\lambda = 0\}$ so we by part i), condition \eqref{eq: k condition} holds for this function in place of $\tilde\lambda$. Therefore,
$$ \eta^j \lambda = 0\quad {\rm for}\ j\leqslant k-1,\quad \eta^k\lambda \neq 0$$
which completes the proof of (ii).
\end{proof}

\begin{definition}
  Let $f : \RRR^2 \to \RRR^2$ be a smooth map which is one-singular at the origin. Then the number $k(f)$ from Lemma \ref{lem:degree} is called the \textit{degree} of $f$.
\end{definition}

\begin{proposition}[Normal form for $1$-singular maps]\label{prop:normalform}
  Let $f : \RRR^2 \to \RRR^2$ be $1$-singular at the origin and let $k \in \mathbb N \cup \{\infty\}$ be its degree. Then there exist open neighbourhoods $U, V \subset \RRR^2$ of the origin as well as diffeomorphisms $\kappa : U \to \tilde U \subset \RRR^2$ and $\tilde \kappa : V \to \tilde V \subset \RRR^2$ with $\kappa(0) = 0 = \tilde \kappa(0)$ and such that, setting
  $$
  \tilde f = \tilde \kappa \circ f \circ\,\kappa^{-1} : \tilde U \to \tilde V,
  $$
  the following holds.
  \begin{enumerate}[label=\emph{(\roman*)}]
    \item If $k = 1$, then $\tilde f(s,u) = (s, u^2)$.
    \item If $k \in \mathbb{N}_{\geqslant 2}$, then $\tilde f(s,u) = (s, su + u^{k+1}\rho(s,u))$ for some smooth $\rho : \tilde U \to \RR$ such that $\rho(0) = 1$.
    \item If $k = \infty$, then $\tilde f(s,u) = (s, su + \varphi(s,u))$ for some smooth $\varphi$ defined near the origin in $\RRR^2$ such that $\partial_u^{\ell}\varphi(0) = 0$ for each $\ell \geqslant 0.$
  \end{enumerate}
\end{proposition}

\begin{remark}\label{rem:rieger}
Let us briefly comment on this result. By a result of Whitney \cite{whitney1955singularities} the cases $k = 1$ and $k = 2$ correspond to the generic case: they are the only stable singularities of in dimension $2$. If fact, when $k = 2$ the function $\rho$ can be taken equal to $1$. By the work Rieger \cite{rieger1987families} we can also take $\rho = 1$ when $k = 3$, see \cite{rieger1987families}*{Lemma 3.1:3}, but this is no longer possible in general for $k\geqslant 4$. For example, if $k = 4$ then the germs $(s,t) \mapsto (s, st + t^5)$ and $(s,t) \mapsto (st + t^5 + t^7)$ are \textit{not} {equivalent}, meaning that there are no source and target changes of coordinates sending one germ to the other, see \cite{rieger1987families}*{Proposition~3.1:2}. Finally we mention that the singularities of $1$-singular stable maps in higher dimensions were classified by Morin, see \cite{morin1965formes}.
\end{remark}

\begin{lemma}\label{lem: nice coord}
Let $f:\mathbb R^2\to\mathbb R^2$ be $1$-singular at the origin with $f(0) = 0$. Then, there is a neighbourhood $U \subset \mathbb R^2$ containing the origin and a diffeomorphism $\kappa: U\to \tilde U = \kappa(U)$ such that
$$ (f \circ \kappa^{-1})(s,t) = (s, g(s,t))$$
for some $g: \tilde U\to \mathbb R$ smooth with $g(0) =0$.
\end{lemma}

\begin{proof}
We have ${\rm rank}\,\dd f(0) = 1$ by Lemma \ref{lem:degree} so one of $\dd f_1(0)$ or $\dd f_2(0)$ is non-zero. Without loss of generality, $\partial_1 f_1(0)\neq 0$. This means that there are $\delta, \varepsilon>0$ such that for all $|s| <\varepsilon$ and $|t|<\varepsilon$, there exists a unique solution $x_1(t,s)\in (-\delta, \delta)$ solving 
\begin{eqnarray}\label{eq: implicit on f1}
f_1(x_1(t,s), t) = s,\quad x_1(0,0) = 0
\end{eqnarray}
depending smoothly on $(s,t)$ satisfying $x_1(0,0) = 0$. If $\varepsilon > 0$ is small enough then the map $(-\varepsilon,\varepsilon)^2\to \mathbb R^2$ given by $(s,t) = (x_1(t,s), t)$ is a diffeomorphism and we denote by $\kappa$ its inverse. Then by \eqref{eq: implicit on f1}, 
$$
f_1(\kappa^{-1}(s,t)) =f_1(x_1(t,s),t) =s.
$$
Therefore, $f(\kappa^{-1}(s,t))= (s, f_2(\kappa^{-1}(s,t)))$ and the proof is complete.
\end{proof}

\begin{proof}[Proof of Proposition \ref{prop:normalform}]
  Since $\dd f$ does not vanish, one can find coordinates $(s, t)$ such that near zero one has $f(s, t) = (s, g(s, t))$ for some smooth $g : \RRR^2 \to \RRR$. In these coordinates, there holds
  $$
  \dd f(s,t) =
  \begin{pmatrix} 1 & 0 \\ \partial_s g(s,t) & \partial_t g(s,t)
  \end{pmatrix} \quad \text{and} \quad \lambda(s,t) = \partial_tg(s,t).
  $$
  We have $C = \{(s,t)~:~\partial_t g(s,t) = 0\}$. In particular setting $\eta(s,t) = \partial_t$ one has
  $$
  \eta(s,t) \in \ker \dd f(s,t) \quad \text{for any } (s,t) \in C.
  $$
  Moreover $\eta^\ell \lambda(s,t) = \partial_t^{\ell + 1}g(s,t)$ for all $\ell \geqslant 0$. Assume that the degree $k$ of $f$ is not infinite. Then
  \begin{equation}\label{eq:derivativeg}
  \partial_t^\ell g(0) = 0 \quad \text{for all } \ell = 1, \dots, k \quad \text{and} \quad \partial_t^{k+1}g(0) \neq 0.
  \end{equation}
  
  Assume $k = 1$. Then $\partial_tg(0) = 0$ and $\partial_t^2g(0) \neq 0$ by \eqref{eq:derivativeg}. Hence there are intervals $I,J \subset \RR$ containing $0$ such that for each $s \in I$ there is a unique critical point $t(s) \in J$ of the map $t \mapsto g(s,t)$. Now setting $(s,v) = (s, t - t(s)) = \kappa(s,t)$ we have
  $$
  g\circ\kappa^{-1}(s, v) = g(s, t(s)) + v^2 \mu(s,v) \quad \text{where} \quad \mu(s,v) = \int_0^1 \partial_t^2g(s, t(s) + cv) (1-c) \dd c.
  $$
  After possibly changing the sign of the second target coordinate,
we may assume that $\mu(0,0)>0$. Now set $b(s) = g(s,t(s))$. Then up to composing $f$ by the local diffeomorphisms
  $$
  (s,t) \mapsto (s, [t - t(s)] \mu(s,t)^{-1/2}) \quad \text{and} \quad (x,y) \mapsto (x, y-b(x))
  $$
  on the right and on the left, respectively, we may assume $g(s,t) = t^2$. This is (i).
  
  Next, assume that $k > 1$ is finite. Then by \eqref{eq:derivativeg} we can find two smooth maps $\alpha, \beta : \RRR^2 \to \RRR$  such that $\beta(0) \neq 0$ as well as  $a : \RRR \to \RRR$ such that
  \begin{equation}\label{eq:g(s,t)}
  g(s, t) = a(s) + s t \alpha(s, t) + t^{k+1}\beta(s, t)
  \end{equation}
  for $s,t$ close enough to zero. Up to composing on the left by $(x,y) \mapsto (x, y - a(x))$ one can assume $a = 0.$ Notice that $\det \dd f(s,t) = \partial_t g(s,t)$. Then, since $f$ is $1$-singular, we have $\dd \partial_t g(0) \neq 0$. Therefore by \eqref{eq:derivativeg} and \eqref{eq:g(s,t)} we must have
  $$
  \alpha(0) = \partial_s\partial_tg(0) \neq 0.
  $$
Hence $(s,u) = (s, t \alpha(s,t)) = \kappa(s,t)$ are coordinates and setting $\tilde \alpha = \alpha \circ \kappa^{-1}$ and $\tilde \beta = \beta \circ \kappa^{-1}$ we have
$$
g \circ \kappa^{-1}(s,u) = su + \rho(s,u) u^{k+1} \quad \text{where} \quad \rho(s,u) = \tilde \beta(s,u) \tilde \alpha (s,u)^{-(k+1)}.
$$
Now up to making a linear change of variables at the target and source level we may assume $\rho(0) = 1$, that is, (ii) holds.
  
   Finally, assume $k = \infty$. In that case \eqref{eq:derivativeg} yields
   $$
   g(s, t) = a(s) + s t \alpha(s, t) + \beta(s, t) \quad \text{with} \quad \partial_t^\ell\beta(0) = 0 \text{ for all } \ell = 0, 1, \dots
   $$
   Again, up to composing on the left by $(x,y) \mapsto (x, y-a(x))$ we may assume $a = 0$ and in the preceding case we have $\alpha(0) \neq 0$. Then as above the change of variable $(s,u) = (s, t \alpha(s,t)) = \kappa(s,t)$ allows us to obtain (iii).
   \end{proof}

\section{Topology of Zoll surfaces}
\label{app:topology}
The purpose of this appendix is to recall well known facts about the topology of Zoll surfaces. The arguments presented below essentially follow those of LeBrun--Mason \cite{LeBrun2002Zoll} which are valid for general projective Zoll structures; for the convenience of the reader, we give a condensed and self-contained exposition in the Riemannian setting relevant to the present work. 

Let $(\Sigma, g)$ be a Riemannian surface and denote by
$$
\pi:S\Sigma\to\Sigma
$$
its unit tangent bundle. We will say that $(\Sigma,g)$ is \textit{Zoll} if all its geodesics are closed, simple, and of length $2\pi$.
\begin{theorem}[\cites{Besse1978Zoll, LeBrun2002Zoll}]\label{thm:zolltopology}
Let $(\Sigma, g)$ be a Zoll surface. 
\begin{enumerate}[label=\emph{(\roman*)}]
\item If $\Sigma$ is orientable, then $\Sigma$ is diffeomorphic to the $2$-sphere. Moreover for each point $p$ on the image $C$ of a closed geodesic , there is exactly one other point $q \in C$ which is conjugate to $p$.
\item If $\Sigma$ is not orientable, then $\Sigma$ is diffeomorphic to $\RRR P^2$. Moreover the only conjugate point of some point along a closed geodesic is itself.
\end{enumerate}
\end{theorem}

{}{}{}{
\begin{remark}\label{rem:green}
If $\Sigma$ is a Zoll surface which is not orientable it follows from a result of Green \cite{green1963wiedersehensflachen} that $\Sigma$ must be \textit{isometric} to $\RRR P^2$, endowed with its canonical Riemannian metric. In fact, this is also true if we only assume that all the geodesics of $\Sigma$ are closed (and not necessarily simple with same length) by a result of Pries \cite{pries2009geodesics}.
\end{remark}
}

We start with the following fact.

\begin{lemma}\label{lem:topologyzoll}
Let $(\Sigma, g)$ be a Riemannian surface all of whose geodesics are closed. Then $\Sigma$ is diffeomorphic to $S^2$ or to $\RRR P^2$.
\end{lemma}

\begin{proof}
Let $S^1 = \RRR / 2\pi \ZZ.$ Let $c : S^1 \to \Sigma$ be a smooth curve. Assume that $c$ is not homotopic to a point. Then there exists a closed geodesic $\gamma : S^1 \to \Sigma$ which is homotopic to $c$, as can be seen by taking a minimising curve in the homotopy class of $c$ or by using the shortening curve flow of Grayson \cite{grayson1989shortening}. Now let $(p,v) = (\gamma(0), \dot \gamma(0)) \in S\Sigma$. Let $w : S^1 \to S_p\Sigma$ be a parametrization. For any $\theta \in S^1$ we consider $\gamma_\theta$ the closed geodesic starting at $p$ with velocity $w(\theta)$. Then $[\gamma_\theta] = [\gamma] = [c] \in \pi_1(\Sigma)$ for each $\theta \in S^1$. Therefore $[c] = [\gamma_\pi] = [\gamma]^{-1} = [c]^{-1}$. This shows that every element in $\pi_1(\Sigma)$ is of order $2$ hence $\Sigma = S^2$ or $\RRR P^2$ by classification of surfaces.
\end{proof}

Let us introduce some notation. We consider the reversing action of $\mathbb Z_2$ on $S\Sigma$ given by $a \cdot (p,v) = (p,v)$ if $a = 1$ and $a \cdot (p, -v)$ if $a = -1$. Then we will denote by $\mathbf PT\Sigma = S\Sigma / \ZZ_2$ the projective tangent space of $\Sigma$. In what follows, we consider $\mathcal G$ the manifold of (unoriented) closed geodesics in $\Sigma$ and let $\widetilde{\mathcal G}$ be the manifold of oriented closed geodesics in $\Sigma$. More precisely, the manifold $\widetilde{\mathcal G}$ is simply $S\Sigma / S^1$ where the $S^1$ action is free and properly discontinuous and given by $(z, t) \mapsto \varphi_t(z)$. Then $\mathcal G = \widetilde{\mathcal G} / \mathbb Z_2$ where the action of $\mathbb Z_2$ on $\mathcal G$ is given by $(a, [z]) \mapsto [az]$.
 
\begin{lemma}\label{lem:topologygeodesics}
The manifolds $\widetilde{\mathcal G}$ and $\mathcal G$ are diffeomorphic to $S^2$ and $\RRR P^2$, respectively.
\end{lemma}

\begin{proof}
The projection map $\mathbf PT\Sigma \to {\mathcal G}$ is smooth and surjective. Moreover its fibers are path connected so the induced map $\pi_1(\mathbf PT\Sigma) \to \pi_1({\mathcal G})$ is surjective. Therefore $\pi_1({\mathcal G})$ is finite, hence $\mathcal G = S^2$ or $\RRR P^2$ by classification of surfaces. However $\widetilde{\mathcal G} \to \mathcal G$ is a double cover of $\mathcal G$, so $\mathcal G$ cannot be $S^2$. We get the sought result.
\end{proof}

Let $\mathbf P T\mathcal G = T \mathcal G / \RRR^\times$ be the projective tangent bundle of $\mathcal G$. 

\begin{lemma}\label{lem:pi_1}
The fundamental groups of $\mathbf P T \Sigma$ and $\mathbf P T\mathcal G$ are finite.
\end{lemma}
\begin{proof}
If $\Sigma$ is orientable we have $S\Sigma \simeq S S^2 \simeq SO(3)$. In particular $\pi_1(S\Sigma)$ is finite and so is $\pi_1(\mathbf PT\Sigma)$ since there is a double covering $S\Sigma \to \mathbf P T\Sigma$. If $\Sigma$ is not orientable consider the double covering $\widetilde \Sigma \to \Sigma$. It induces a double covering $\mathbf P T \widetilde \Sigma \to \mathbf P T \Sigma$. However $\pi_1(\mathbf P T \widetilde \Sigma)$ is finite by the orientable case hence the result follows.
\end{proof}

We now introduce a map 
$$
\Phi : \mathbf PT\Sigma \to \mathbf P T\mathcal G,
$$
as follows. Let $[z] = [(p,v)] \in \mathbf P T \Sigma$ and $w : S^1 \to S_p\Sigma$ be one of the two natural parametrization such that $w(0) = v$. For $\theta \in S^1$ we let $C_\theta \in \mathcal G$ be determined by $[(p, w(\theta))]$ and set $C = C_0$ for convenience. Then we define
$$
\Phi([z]) = \bigl[\partial_\theta|_{\theta = 0}C_\theta\bigr] \in \mathbf PT_C\mathcal G.
$$ 
Note that this is independent of our choices since changing the parametrization would only change the sign of $\partial_\theta|_{\theta = 0} C_\theta \in T_C\mathcal G$. 

\begin{lemma}\label{lem:fibers}
Consider the restriction 
$$
\Phi_C = \Phi|_{\mathbf PTC} : \mathbf PTC \to \mathbf P T_C\mathcal G
$$
where $\mathbf P T C \subset \mathbf P T\Sigma$ is the projective tangent bundle of $C \subset \Sigma$. Note that $\mathbf P T C$ is naturally identified with $C$ itself. Then $\Phi_C$ is an immersion and for any $p,q \in C$ we have 
$$
\Phi_C(p) = \Phi_C(q) \in \mathbf PT_C\mathcal G \quad \text{iff} \quad p \text{ and } q \text{ are conjugate along }C.
$$
\end{lemma}
\begin{proof}
Indeed, let $q = \pi(\varphi_t(p,v))$ for some $t \in \RRR$, $z =(p,v) \in S\Sigma$ with $v \in TC$.  Then $q$ is conjugate to $p$ along the curve $s \mapsto \pi(\varphi_s(z))$ if and only if $\partial_\theta|_{\theta = 0} \pi(\varphi_{s}(p, w(\theta))) = 0$. The latter condition is equivalent to the fact that for any parametrization $u : S^1 \to S_q\Sigma$ the vector $\partial_\theta|_{\theta = 0}\left[\varphi_s(p,w(\theta)) - (q,u(\theta))\right]$ lies in the kernel of the linearized action of $S^1$ on $S\Sigma$, which implies $\Phi_C(p) = \Phi_C(q)$. The reverse implication follows from the fact that two Jacobi fields along $C$ vanishing at a same point are multiple of each other. Note also that since Jacobi fields are non degenerate, $\partial_\theta|_{\theta = 0} \pi(\varphi_{s}(p, w(\theta))) = 0$ implies that $\partial_s\partial_\theta|_{\theta = 0} \pi(\varphi_{s}(p, w(\theta)))$ is not zero and orthogonal to $TC$. This shows that the map $\dd \Phi_C(p) : T_{[z]}\mathbf PTC \simeq C \to \mathbf PT_C\mathcal G$ is injective. This completes the proof.
\end{proof}

\begin{proof}[Proof of Theorem \ref{thm:zolltopology}]
Note that the composition $\widetilde \pi \circ \Phi$ is a submersion where $\widetilde \pi$ is the natural projection $\mathbf P T\mathcal G \to \mathcal G$, simply because $\widetilde \pi \circ \Phi = \pi$. Let $[z] \in \mathbf P T \Sigma$ and write $\Phi([z]) = (C, \rho)$ for $C \in \mathcal G$ and $\rho \in \mathbf P T_C\mathcal G$. Let $A = \mathrm{Im}(\dd \Phi([z])) \subset T_{(C,\rho)} \mathbf PT\mathcal G$. Since $\widetilde \pi \circ \Phi$ is a submersion we have that $A$ is transversal to $\ker \dd \widetilde \pi(C,\rho)$. However the latter space coincides with $\mathbf P T_C\mathcal G$ which is contained in $A$ since by Lemma \ref{lem:fibers} $\dd \Phi_C([z])$ is an isomorphism 
$$
 T_{[z]}\mathbf P TC  \to T_{(C,\rho)}\mathbf P T_C\mathcal G.
$$
 Hence $A = T_{(C,\rho)} \mathbf P T \mathcal G$. It follows that $\Phi$ is covering map. By Lemma \ref{lem:fibers}, for each $C \in \mathcal G$ we have that $\Phi_C$ is a covering map $C \to \mathbf PT_C\mathcal G \simeq \RRR P^1$ whose degree coincides with the number of conjugate points along $C$ of any point of $C$ (including itself). Since $\Phi$ maps the fibers of $\mathbf P T \Sigma \to \mathcal G$ to the fibers of $\mathbf P T \mathcal G \to \mathcal G$, we conclude that for any $C \in \mathcal G$ and $p \in C$, the number of conjugate points along $C$ of any point of $C$ (including itself) coincides with the degree of the covering map $\Phi$. However by Lemma \ref{lem:pi_1}, the fundamental groups of $\mathbf P T \Sigma$ and $\mathbf P T \mathcal G$ are finite. Since $\Phi$ is a covering map, $\Phi_* : \pi_1(\mathbf P T \Sigma) \to \pi_1(\mathbf P T \mathcal G)$ is injective and we obtain
$$
\deg(\Phi) = [\pi_1(\mathbf P T \mathcal G) : \Phi_* \pi_1(\mathbf P T \Sigma)] = \frac{|\pi_1(\mathbf P T \mathcal G)|}{|\Phi_*\pi_1(\mathbf P T \Sigma)|} = \frac{|\pi_1(\mathbf P T \mathcal G)|}{|\pi_1(\mathbf P T \Sigma)|}
$$
In particular, this number is independent of the Riemannian metric $g$. If $\Sigma$ is orientable, then $\Sigma$ is diffeomorphic to the sphere. If $g$ is the round metric, then we know that for every great circle $C \subset \Sigma$ and $p \in C$, the conjugate points of $p$ along $C$ are $p$ and its antipodal point $q$. Hence $\deg(\Phi) = 2$. If $\Sigma$ is not orientable, we get that $\Sigma$ is diffeomorphic to $\RRR P^2$. For the canonical Riemannian metric on $\RRR P^2$ we have that for any closed geodesic $C$ and any conjugate points $p,q \in C$ it holds $p = q$. Hence $\deg(\Phi) = 1$. This completes the proof.
\end{proof}

We conclude this appendix by stating the following theorem, due to Gromoll--Grove \cite{gromoll1981metrics}.

\begin{theorem}[\cite{gromoll1981metrics}]
Let $(\Sigma, g)$ be an orientable Riemannian surface such that all whose geodesics are closed. Then all its primitive closed geodesic are simple and have the same length.
\end{theorem}

Since we will not make use of this result we only outline below the proof of Gromoll--Grove. First, a result of Epstein \cite{epstein1972periodic} implies that any smooth flow on a compact three manifold whose orbits are circles is periodic up to reparametrization of the flow. Combining this with the fact that closed geodesics are critical points for the length function implies that the geodesic flow itself is periodic, say of period $2\pi$. Hence the length of any closed geodesics must be of the form $2\pi / k$ for some $k \in \NN_{>1}$. Now the orientation preserving diffeomorphism $\psi_k = \varphi_{2\pi/k}$ has finite order and from basic differential topology this implies that its fixed point set is smooth and of dimension $1$. Therefore it must coincides with the union of the closed orbits of $(\varphi_t)$ whose minimal period is $2\pi / k$. Lifting everything to $\widetilde{S\Sigma} \simeq S^3$ one obtains a Seifert fibration of $S^3$ by circles (indeed by Lemma \ref{lem:topologyzoll} we have that $\Sigma$ diffeomorphic to $S^2$). By a classical result of Seifert \cite{seifert1933topologie} there are at most two degenerate circles, which correspond to closed orbits of period $2\pi / k$ for some $k > 1$. However by the Lusternik-Schnirelmann theorem \cites{lusternik1929probleme, grayson1989shortening} there are at least three simple closed geodesics on $\Sigma$. Hence at least one of their lifts to $\widetilde{S\Sigma}$ is not singular for the Seifert fibration. This means that there exists one simple closed geodesic of length $2\pi$. Next, notice that the set of vectors in $S\Sigma$ whose geodesic trajectory is closed and simple in $\Sigma$ is open. But it is also closed, since it is the complementary of the set of vectors of $S\Sigma$ whose associated closed geodesic has a transversal self-intersection. By connectedness this set coincides with $S\Sigma$.

\bibliographystyle{plain}
\bibliography{./refs.bib}

@book{hormander2007PDE3,
  author    = {H\"ormander, Lars},
  title     = {The analysis of linear partial differential operators. {III}},
  series    = {Classics in Mathematics},
  publisher = {Springer, Berlin},
  year      = {2007},
  pages     = {viii+525},
  isbn      = {978-3-540-49937-4},
  mrclass   = {35-02 (35Sxx 47F05 47G30 58J40)},
  mrnumber  = {2304165},
  doi       = {10.1007/978-3-540-49938-1}
}

@book{hormander2009PDE4,
  author    = {H\"ormander, Lars},
  title     = {The analysis of linear partial differential operators. {IV}},
  series    = {Classics in Mathematics},
  note      = {Fourier integral operators,
               Reprint of the 1994 edition},
  publisher = {Springer-Verlag, Berlin},
  year      = {2009},
  pages     = {viii+352},
  isbn      = {978-3-642-00117-8},
  mrclass   = {35-02 (35S30 47F05 47G30 58J40)},
  mrnumber  = {2512677},
  doi       = {10.1007/978-3-642-00136-9},
  url       = {https://doi.org/10.1007/978-3-642-00136-9}
}

@article{tully2024levy,
  author     = {Tully, Kevin},
  title      = {A microlocal analysis of the {L}\'evy generator with conjugate
                points},
  journal    = {J. Geom. Anal.},
  fjournal   = {Journal of Geometric Analysis},
  volume     = {34},
  year       = {2024},
  number     = {12},
  pages      = {Paper No. 357, 25},
  issn       = {1050-6926,1559-002X},
  mrclass    = {58J40 (58J65)},
  mrnumber   = {4802998},
  mrreviewer = {Julio\ Delgado},
  doi        = {10.1007/s12220-024-01813-4},
  url        = {https://doi.org/10.1007/s12220-024-01813-4}
}

@article{chaubet2025levy,
  author   = {Chaubet, Yann and Guedes Bonthonneau, Yannick  and Lefeuvre,
              Thibault and Tzou, Leo},
  title    = {Geodesic {L}\'evy flights and expected stopping time for
              random searches},
  journal  = {Probab. Theory Related Fields},
  fjournal = {Probability Theory and Related Fields},
  volume   = {191},
  year     = {2025},
  number   = {1-2},
  pages    = {235--285},
  issn     = {0178-8051,1432-2064},
  mrclass  = {60G51 (31C12 58J40 60G53)},
  mrnumber = {4869256},
  doi      = {10.1007/s00440-024-01327-8},
  url      = {https://doi.org/10.1007/s00440-024-01327-8}
}

@article{holman2018xray,
  author     = {Holman, Sean and Uhlmann, Gunther},
  title      = {On the microlocal analysis of the geodesic {X}-ray transform
                with conjugate points},
  journal    = {J. Differential Geom.},
  fjournal   = {Journal of Differential Geometry},
  volume     = {108},
  year       = {2018},
  number     = {3},
  pages      = {459--494},
  issn       = {0022-040X,1945-743X},
  mrclass    = {53C65 (58J40)},
  mrnumber   = {3770848},
  mrreviewer = {B.\ S.\ Rubin},
  doi        = {10.4310/jdg/1519959623},
  url        = {https://doi.org/10.4310/jdg/1519959623}
}

@book{Besse1978Zoll,
  author     = {Besse, Arthur L.},
  title      = {Manifolds all of whose geodesics are closed},
  series     = {Ergebnisse der Mathematik und ihrer Grenzgebiete [Results in
                Mathematics and Related Areas]},
  volume     = {93},
  note       = {With appendices by D. B. A. Epstein, J.-P. Bourguignon, L.
                B\'erard-Bergery, M. Berger and J. L. Kazdan},
  publisher  = {Springer-Verlag, Berlin-New York},
  year       = {1978},
  pages      = {ix+262},
  isbn       = {3-540-08158-5},
  mrclass    = {53C20 (53C22 58G99)},
  mrnumber   = {496885},
  mrreviewer = {R.\ L.\ Bishop}
}

@article{Zoll1903Zoll,
  author   = {Zoll, Otto},
  title    = {Ueber {F}l\"achen mit {S}charen geschlossener geod\"atischer
              {L}inien},
  journal  = {Math. Ann.},
  fjournal = {Mathematische Annalen},
  volume   = {57},
  year     = {1903},
  number   = {1},
  pages    = {108--133},
  issn     = {0025-5831,1432-1807},
  mrclass  = {99-04},
  mrnumber = {1511201},
  doi      = {10.1007/BF01449019},
  url      = {https://doi.org/10.1007/BF01449019}
}

@article{LeBrun2002Zoll,
  author     = {Le{B}run, Claude and Mason, L. J.},
  title      = {Zoll manifolds and complex surfaces},
  journal    = {J. Differential Geom.},
  fjournal   = {Journal of Differential Geometry},
  volume     = {61},
  year       = {2002},
  number     = {3},
  pages      = {453--535},
  issn       = {0022-040X,1945-743X},
  mrclass    = {53C22 (32J15 53C28)},
  mrnumber   = {1979367},
  mrreviewer = {Massimiliano\ Pontecorvo},
  url        = {http://projecteuclid.org/euclid.jdg/1090351530}
}

@article{Guillemin1976Radon,
  author     = {Guillemin, Victor},
  title      = {The {R}adon transform on {Z}oll surfaces},
  journal    = {Advances in Math.},
  fjournal   = {Advances in Mathematics},
  volume     = {22},
  year       = {1976},
  number     = {1},
  pages      = {85--119},
  issn       = {0001-8708},
  mrclass    = {58G15 (53C20)},
  mrnumber   = {426063},
  mrreviewer = {J.\ Eells},
  doi        = {10.1016/0001-8708(76)90139-0}
}

@article{warner1965conjugateRiemann,
  author     = {Warner, Frank W.},
  title      = {The conjugate locus of a {R}iemannian manifold},
  journal    = {Amer. J. Math.},
  fjournal   = {American Journal of Mathematics},
  volume     = {87},
  year       = {1965},
  pages      = {575--604},
  issn       = {0002-9327,1080-6377},
  mrclass    = {53.72},
  mrnumber   = {208534},
  mrreviewer = {N.\ J.\ Hicks},
  doi        = {10.2307/2373064}
}

@incollection{chaubet_dynamical,
  author    = {Chaubet, Yann},
  title     = {Dynamical zeta functions},
  booktitle = {in {Microlocal Analysis in Hyperbolic Dynamics and Geometry}, by {Thibault Lefeuvre}},
  series    = {Cours Spécialisés},
  volume    = {32},
  publisher = {Société Mathématique de France},
  address   = {Paris},
  year      = {2025},
  pages     = {220--264}
}

@article{applebaum2000,
  author     = {Applebaum, D. and Estrade, A.},
  title      = {Isotropic {L}\'evy processes on {R}iemannian manifolds},
  journal    = {Ann. Probab.},
  fjournal   = {The Annals of Probability},
  volume     = {28},
  year       = {2000},
  number     = {1},
  pages      = {166--184},
  issn       = {0091-1798,2168-894X},
  mrclass    = {58J65 (60J25)},
  mrnumber   = {1756002},
  mrreviewer = {Ming\ Liao},
  doi        = {10.1214/aop/1019160116},
  url        = {https://doi.org/10.1214/aop/1019160116}
}

@book{singer1976lecture,
  author    = {Singer, I. M. and Thorpe, John A.},
  title     = {Lecture notes on elementary topology and geometry},
  publisher = {Scott, Foresman \& Co., Glenview, IL},
  year      = {1967},
  pages     = {v+214},
  mrclass   = {53.40 (54.00)},
  mrnumber  = {213982}
}

@book{Grigis-Sjostrand-94,
  author     = {Grigis, Alain and Sj\"{o}strand, Johannes},
  title      = {Microlocal analysis for differential operators},
  series     = {London Mathematical Society Lecture Note Series},
  volume     = {196},
  note       = {An introduction},
  publisher  = {Cambridge University Press, Cambridge},
  year       = {1994},
  pages      = {iv+151},
  isbn       = {0-521-44986-3},
  mrclass    = {35A27 (35Sxx 58G07 58G15 58G17)},
  mrnumber   = {1269107},
  mrreviewer = {Vesselin M. Petkov},
  doi        = {10.1017/CBO9780511721441},
  url        = {https://doi-org.revues.math.u-psud.fr/10.1017/CBO9780511721441}
}

@article{gromoll1981metrics,
  author     = {Gromoll, Detlef and Grove, Karsten},
  title      = {On metrics on {$S\sp{2}$}\ all of whose geodesics are closed},
  journal    = {Invent. Math.},
  fjournal   = {Inventiones Mathematicae},
  volume     = {65},
  year       = {1981},
  number     = {1},
  pages      = {175--177},
  issn       = {0020-9910,1432-1297},
  mrclass    = {58E10 (53C22)},
  mrnumber   = {636885},
  mrreviewer = {Gudlaugur\ Thorbergsson},
  doi        = {10.1007/BF01389300},
  url        = {https://doi.org/10.1007/BF01389300}
}

@article{epstein1972periodic,
  author     = {Epstein, David BA},
  title      = {Periodic flows on three-manifolds},
  journal    = {Ann. of Math. (2)},
  fjournal   = {Annals of Mathematics. Second Series},
  volume     = {95},
  year       = {1972},
  pages      = {66--82},
  issn       = {0003-486X},
  mrclass    = {57.48},
  mrnumber   = {288785},
  mrreviewer = {V.\ Poenaru},
  doi        = {10.2307/1970854},
  url        = {https://doi.org/10.2307/1970854}
}

@article{seifert1933topologie,
  author   = {Seifert, Herbert},
  title    = {Topologie {D}reidimensionaler {G}efaserter {R}\"aume},
  journal  = {Acta Math.},
  fjournal = {Acta Mathematica},
  volume   = {60},
  year     = {1933},
  number   = {1},
  pages    = {147--238},
  issn     = {0001-5962,1871-2509},
  mrclass  = {99-04},
  mrnumber = {1555366},
  doi      = {10.1007/BF02398271},
  url      = {https://doi.org/10.1007/BF02398271}
}

@article{lusternik1929probleme,
  title   = {Sur le probl{\`e}me de trois g{\'e}od{\'e}siques ferm{\'e}es sur les surfaces de genre 0},
  author  = {Lusternik, Lazar and Schnirelmann, Lev},
  journal = {CR Acad. Sci. Paris},
  volume  = {189},
  pages   = {269--271},
  year    = {1929}
}

@article{grayson1989shortening,
  author     = {Grayson, Matthew A.},
  title      = {Shortening embedded curves},
  journal    = {Ann. of Math. (2)},
  fjournal   = {Annals of Mathematics. Second Series},
  volume     = {129},
  year       = {1989},
  number     = {1},
  pages      = {71--111},
  issn       = {0003-486X,1939-8980},
  mrclass    = {53C22 (58E10)},
  mrnumber   = {979601},
  mrreviewer = {Gudlaugur\ Thorbergsson},
  doi        = {10.2307/1971486},
  url        = {https://doi.org/10.2307/1971486}
}

@article{morin1965formes,
  title    = {Formes canoniques des singularit{\'e}s d'une application diff{\'e}rentiable},
  author   = {Morin, Bernard},
  journal  = {CR Acad. Sci. Paris},
  volume   = {260},
  pages    = {5662--5665},
  year     = {1965},
  issn     = {0001-4036},
  mrclass  = {57.20},
  mrnumber = {180982}
}

@article{rieger1987families,
  author     = {Rieger, J. H.},
  title      = {Families of maps from the plane to the plane},
  journal    = {J. London Math. Soc. (2)},
  fjournal   = {Journal of the London Mathematical Society. Second Series},
  volume     = {36},
  year       = {1987},
  number     = {2},
  pages      = {351--369},
  issn       = {0024-6107,1469-7750},
  mrclass    = {58C27 (57R45)},
  mrnumber   = {906153},
  mrreviewer = {J.\ O.\ Bedford},
  doi        = {10.1112/jlms/s2-36.2.351},
  url        = {https://doi.org/10.1112/jlms/s2-36.2.351}
}

@book{laudenbach2011transversalite,
  title     = {{Transversalit{\'e}, courants et th{\'e}orie de Morse: un cours de topologie diff{\'e}rentielle}},
  author    = {Laudenbach, Fran{\c{c}}ois},
  year      = {2011},
  publisher = {Editions Ecole Polytechnique},
  mrnumber  = {3088239}
}

@book{duistermaat1996fourier,
  author    = {J. J. Duistermaat},
  title     = {Fourier Integral Operators},
  series    = {Modern Birkh{\"a}user Classics},
  publisher = {Birkh{\"a}user},
  address   = {Boston},
  year      = {1996},
  isbn      = {978-0-8176-4673-7},
  mrnumber = {1362544}
}

@incollection{whitney1955singularities,
  title={On singularities of mappings of Euclidean spaces. I. Mappings of the plane into the plane},
  author={Whitney, Hassler},
  booktitle={Hassler Whitney Collected Papers},
  pages={370--406},
  year={1955},
  publisher={Springer}
}

@article{UribeZelditch1993,
  author  = {Uribe, Alejandro and Zelditch, Steven},
  title   = {{Spectral Statistics on Zoll Surfaces}},
  journal = {Communications in Mathematical Physics},
  volume  = {154},
  pages   = {313--346},
  year    = {1993},
  mrnumber ={1224082}
}

@article{Zelditch1996,
  author  = {Zelditch, Steven},
  title   = {{Maximally Degenerate Laplacians}},
  journal = {Annales de l'Institut Fourier},
  volume  = {46},
  number  = {2},
  pages   = {547--587},
  year    = {1996},
  doi     = {10.5802/aif.1524},
  mrnumber = {1393525}
}

@article{Zelditch1997,
  author  = {Zelditch, Steven},
  title   = {{Fine Structure of Zoll Spectra}},
  journal = {Journal of Functional Analysis},
  volume  = {143},
  number  = {2},
  pages   = {415--460},
  year    = {1997},
  mrnumber = {1428823}
}

@article{MaciaRiviere2016,
  author  = {Maci{\`a}, Fabricio and Rivi{\`e}re, Gabriel},
  title   = {{Concentration and Non-Concentration for the Schr{\"o}dinger Evolution on Zoll Manifolds}},
  journal = {Communications in Mathematical Physics},
  volume  = {345},
  number  = {3},
  pages   = {1019--1054},
  year    = {2016},
  mrnumber  = {3519588}
}

@article{Funk1913,
  author  = {Funk, Paul},
  title   = {{\"Uber Fl{\"a}chen mit lauter geschlossenen geod{\"a}tischen Linien}},
  journal = {Mathematische Annalen},
  volume  = {74},
  pages   = {278--300},
  year    = {1913},
  doi     = {10.1007/BF01456044},
  mrnumber = {1511763}
}

@article{Weinstein1974,
  author  = {Weinstein, Alan},
  title   = {On the Volume of Manifolds All of Whose Geodesics Are Closed},
  journal = {Journal of Differential Geometry},
  volume  = {9},
  number  = {4},
  pages   = {513--517},
  year    = {1974},
  mrnumber = {390968}
}

@article{mazzucchelli2018characterization,
  title={{A characterization of Zoll Riemannian metrics on the 2-sphere}},
  author={Mazzucchelli, Marco and Suhr, Stefan},
  journal={Bulletin of the London Mathematical Society},
  volume={50},
  number={6},
  pages={997--1006},
  year={2018},
  publisher={Wiley Online Library},
  mrnumber = {3891938}
}

@article{seeger1991regularity,
  title={Regularity properties of {F}ourier integral operators},
  author={Seeger, Andreas and Sogge, Christopher D and Stein, Elias M},
  journal={Annals of Mathematics},
  volume={134},
  number={2},
  pages={231--251},
  year={1991},
  publisher={JSTOR},
  mrnumber = {1127475}
}

@article{green1963wiedersehensflachen,
  title={Auf wiedersehensfl{\"a}chen},
  author={Green, Leon W},
  journal={Annals of Mathematics},
  volume={78},
  number={2},
  pages={289--299},
  year={1963},
  publisher={JSTOR}
}

@article{pries2009geodesics,
  title={Geodesics closed on the projective plane},
  author={Pries, Christian},
  journal={Geometric and Functional Analysis},
  volume={18},
  number={5},
  pages={1774--1785},
  year={2009},
  publisher={Springer}
}

@book{PaternainSaloUhlmann2023,
  author    = {Paternain, Gabriel P. and Salo, Mikko and Uhlmann, Gunther},
  title     = {Geometric Inverse Problems},
  series    = {Cambridge Studies in Advanced Mathematics},
  publisher = {Cambridge University Press},
  year      = {2023},
  doi       = {10.1017/9781009041423}
}

@book{Lefeuvre2025,
  author    = {Lefeuvre, Thibault},
  title     = {Microlocal Analysis in Hyperbolic Dynamics and Geometry},
  series     = {Cours Spécialisés},
  volume     = {32},
  publisher  = {Société Mathématique de France},
  address    = {Paris},
  year       = {2025}
}

@article{colin1979spectre,
 author = {Colin de Verdi{\`e}re, Yves},
 title = {Sur le spectre des op{\'e}rateurs elliptiques {\`a} bicaracteristiques toutes p{\'e}riodiques},
 fjournal = {Commentarii Mathematici Helvetici},
 journal = {Comment. Math. Helv.},
 issn = {0010-2571},
 volume = {54},
 pages = {508--522},
 year = {1979},
 language = {French},
 doi = {10.1007/BF02566290},
 url = {https://eudml.org/doc/139800},
 zbMATH = {3720047},
 Zbl = {0459.58014}
}

@article{weinstein1977asymptotics,
 author = {Weinstein, Alan},
 title = {Asymptotics of eigenvalue clusters for the {Laplacian} plus a potential},
 fjournal = {Duke Mathematical Journal},
 journal = {Duke Math. J.},
 issn = {0012-7094},
 volume = {44},
 pages = {883--892},
 year = {1977},
 language = {English},
 doi = {10.1215/S0012-7094-77-04442-8},
 zbMATH = {3599194},
 Zbl = {0385.58013}
}

@article{duistermaat1975spectrum,
 author = {Duistermaat, J. J. and Guillemin, V. W.},
 title = {The spectrum of positive elliptic operators and periodic bicharacteristics},
 fjournal = {Inventiones Mathematicae},
 journal = {Invent. Math.},
 issn = {0020-9910},
 volume = {29},
 pages = {39--79},
 year = {1975},
 language = {English},
 doi = {10.1007/BF01405172},
 url = {https://eudml.org/doc/142329},
 zbMATH = {3481135},
 Zbl = {0307.35071}
}

@book{taylor2013partial,
  title={Partial differential equations II: Qualitative studies of linear equations},
  author={Taylor, Michael},
  volume={116},
  year={2013},
  publisher={Springer Science \& Business Media}
}

@article{waters2017bifurcations,
 author = {Waters, Thomas},
 title = {Bifurcations of the conjugate locus},
 fjournal = {Journal of Geometry and Physics},
 journal = {J. Geom. Phys.},
 issn = {0393-0440},
 volume = {119},
 pages = {1--8},
 year = {2017},
 language = {English},
 doi = {10.1016/j.geomphys.2017.04.003},
 zbMATH = {6738980},
 Zbl = {1367.53039}
}

@article{waters2019conjugate,
 author = {Waters, Thomas},
 title = {The conjugate locus on convex surfaces},
 fjournal = {Geometriae Dedicata},
 journal = {Geom. Dedicata},
 issn = {0046-5755},
 volume = {200},
 pages = {241--254},
 year = {2019},
 language = {English},
 doi = {10.1007/s10711-018-0368-8},
 zbMATH = {7065310},
 Zbl = {1419.53005}
}

\end{document}